\documentclass[a4paper, 12pt]{amsart}
\usepackage{amssymb, amsmath, amsthm, color}
\usepackage[colorlinks=true, breaklinks=true, linkcolor=black, citecolor=blue, urlcolor=red]{hyperref} 
\usepackage[francais,english]{babel}
\usepackage{tikzit} 

\numberwithin{equation}{section}
\newtheorem{lettertheo}{Theorem}

\newtheorem{lettercor}[lettertheo]{Corollary}

\newtheorem{theorem}{Theorem}[section]

\newtheorem{corollary}[theorem]{Corollary}
\newtheorem{proposition}[theorem]{Proposition}
\newtheorem{question}[theorem]{Question}

\newtheorem{observation}[theorem]{Observation}

\theoremstyle{definition} 
\newtheorem{definition-proposition}[theorem]{Definition-Proposition}
\newtheorem{definition}[theorem]{Definition}

\newtheorem{remark}[theorem]{Remark}
\newtheorem{example}[theorem]{Example}

\newcommand{\act}{\curvearrowright}

\DeclareMathOperator{\ad}{ad}
\newcommand{\al}{\alpha}
\newcommand{\bs}{\backslash}
\newcommand{\betti}{\beta^{(2)}_1}
\newcommand{\cC}{\mathcal C}
\DeclareMathOperator{\cm}{cm}
\newcommand{\cD}{\mathcal D}
\newcommand{\cE}{\mathcal E}
\newcommand{\cF}{\mathcal F}

\DeclareMathOperator{\Fix}{Fix}
\DeclareMathOperator{\Frac}{Frac}
\DeclareMathOperator{\FC}{ForC}
\DeclareMathOperator{\FM}{ForM}
\DeclareMathOperator{\Forest}{Forest}
\newcommand{\cG}{\mathcal G}
\DeclareMathOperator{\Gr}{Gr}
\DeclareMathOperator{\Groupoid}{Groupoid}
\newcommand{\ga}{\gamma}
\newcommand{\Ga}{\Gamma}
\DeclareMathOperator{\Gar}{Gar}

\DeclareMathOperator{\Hilb}{Hilb}
\DeclareMathOperator{\Hom}{Hom}
\newcommand{\cI}{\mathcal I}
\DeclareMathOperator{\id}{id}
\newcommand{\into}{\hookrightarrow}
\newcommand{\la}{\langle}
\DeclareMathOperator{\Leaf}{Leaf}
\DeclareMathOperator{\length}{length}
\DeclareMathOperator{\link}{link}
\newcommand{\linkd}{\link_{\downarrow}}
\DeclareMathOperator{\mcm}{mcm}
\DeclareMathOperator{\Mon}{Mon}
\newcommand{\N}{\mathbf{N}}
\newcommand{\cO}{\mathcal O}
\DeclareMathOperator{\ob}{ob}
\newcommand{\onto}{\twoheadrightarrow}
\newcommand{\ot}{\otimes}
\newcommand{\ov}{\overline}
\newcommand{\cP}{\mathcal P}
\newcommand{\Q}{\mathbf{Q}}
\newcommand{\R}{\mathbf{R}}
\newcommand{\ra}{\rangle}

\DeclareMathOperator{\Root}{Root}
\newcommand{\cS}{\mathcal{S}}
\DeclareMathOperator{\Set}{Set}
\DeclareMathOperator{\Stab}{Stab}
\DeclareMathOperator{\Sp}{Spine}

\newcommand{\cT}{\mathcal T}
\newcommand{\ti}{\tilde}
\DeclareMathOperator{\TR}{TR}
\newcommand{\cU}{\mathcal U}
\newcommand{\cUF}{\mathcal{UF}}
\DeclareMathOperator{\Ver}{Ver}

\newcommand{\wh}{\widehat}
\DeclareMathOperator{\word}{word}

\newcommand{\Z}{\mathbf{Z}}

\begin{document}

\title[Forest-skein groups I]{Forest-skein groups I: between Vaughan Jones' subfactors and Richard Thompson's groups}

\thanks{
AB is supported by the Australian Research Council Grant DP200100067.}
\author{Arnaud Brothier}
\address{Arnaud Brothier\\ School of Mathematics and Statistics, University of New South Wales, Sydney NSW 2052, Australia}
\email{arnaud.brothier@gmail.com\endgraf
\url{https://sites.google.com/site/arnaudbrothier/}}

\begin{abstract}
Vaughan Jones discovered unexpected connections between Richard Thompson's group and subfactor theory while attempting to construct conformal field theories (in short CFT).
Among other this founded {\it Jones' technology:} a powerful new method for constructing actions of {\it fraction groups} which had numerous applications in mathematical physics, operator algebras, group theory and more surprisingly in knot theory and noncommutative probability theory. 

We propose and outline a program in the vein of Jones' work but where the Thompson group is replaced by a family of groups that we name {\it forest-skein groups}. 
These groups are constructed from diagrammatic categories, are tailor-made for using Jones' technology, capture key aspects of the Thompson group, and aim to better connect subfactors with CFT.
Our program strengthens Jones' visionary work and moreover produces a plethora of concrete groups which satisfy exceptional properties.

In this first article we introduce the general theory of forest-skein groups, provide criteria of existence, give explicit presentations, prove that their first $\ell^2$-Betti number vanishes, construct a canonical action on a totally ordered set, establish a topological finiteness theorem showing that many of our groups are of type $F_\infty$, and finish by studying a beautiful class of explicit examples. 
\end{abstract}

\maketitle


\tableofcontents


\section*{Introduction}
We start by explaining how subfactor theory, conformal field theory, and the groups of Richard Thompson got connected. Second, we outline a general program: what do we wish and what are our motivations and aims. Third, we explain our general formalism based on the manipulation of diagrams and explain how it relates to existing constructions in the literature of group theory. We end by describing the content of this present article and briefly mentioning some works in progress.

{\bf Subfactor, planar algebra, conformal field theory, and braid.}
In the 1980's Vaughan Jones initiated subfactor theory which rapidly became a major field in operator algebras \cite{Jones83}.
Connections with mathematical physics naturally appeared (see \cite{Evans-Kawahigashi92, Kawahigashi18-ICM,Evans22}) but also with seemingly completely unrelated subjects such as knot theory with the introduction of the Jones polynomial \cite{Jones85,Jones87}.
Jones worked very hard in finding new point of views and powerful formalisms to study objects. 
This led to the introduction of Jones' planar algebras: a description of the standard invariant of subfactors using diagrams similar to string diagrams in tensor categories \cite{Jones99}.

{\bf Richard Thompson's groups $F,T,$ and $V$.}
On a rather different part of mathematics lives the three Richard Thompson group $F\subset T\subset V$ \cite{Cannon-Floyd-Parry96}.
These groups naturally appear in various parts of mathematics such as logic, topology, dynamics, complexity theory to name a few.
They follow rather unusual properties and display new phenomena in group theory.
For instance, the groups $T$ and $V$ were the first examples found of finitely presented simple infinite groups. 
Moreover, Brown and Geoghegan proved that $F$ satisfied the topological finiteness property $F_\infty$ and was of infinite geometric dimension, providing the first torsion-free group of this kind, see Section \ref{sec:def-finiteness} for definitions \cite{Brown-Geoghegan84}.
These groups seem completely unrelated to subfactors although they share one common feature: elements of Thompson's groups can be described by planar diagrams (a pair of rooted trees) as shown by Brown \cite{Brown87}.

{\bf From subfactors to almost CFT, and Thompson's group $T$.}
In the 2010's Jones discovered an unexpected connection between subfactors and the Thompson groups \cite{Jones17}, see also the survey \cite{Brothier20-survey}.
The story goes as follows.
Subfactors and conformal field theory (in short CFT) in the formalism of Doplicher-Haag-Roberts (DHR), have been closely linked since the late 1980's when Longo quickly realised that the Jones index is equal to the square of the statistical dimension in DHR theory \cite{Doplicher-Haag-Roberts71, Longo89}. 
Moreover, in the mid 1990's Longo and Rehren prove that any CFT produces a subfactor \cite{Longo-Rehren95}.
However, the converse remains mysterious and is one of the most important question in the field, see  \cite{Evans-Gannon11, Bischoff17, Xu18}.
Jones famous question was: "Does every subfactor has something to do with conformal field theory?"
There are some case by case reconstruction results but the most exotic and interesting subfactors are not known to provide CFT at the moment.
Using planar algebras Jones constructed a field theory associated to {\it each} subfactor that is not quite conformal but has a rather discrete group of symmetry.
This symmetry group is nothing else than Thompson's group $T$. 

{\bf New connections and Jones' technology.}
From there, new field theories were introduced but also connections between subfactors, Thompson's groups, and braid groups \cite{Jones17,Jones16}.
For the last connection we invite the reader to consult the survey of Jones \cite{Jones19-survey}.
Moreover, Jones found an efficient technology for constructing actions of Thompson's group $F$ using diagrams.
Originally, this was done by transforming tree-diagrams of Thompson group into string diagrams in a planar algebra or a nice tensor category. 
Jones quickly realised that it could be greatly generalised founding {\it Jones' technology}.
It works as follows.
Consider a monoid or a category $\cC$ admitting a calculus of fractions. 
This allows to formally inverting elements (or more precisely morphisms) of $\cC$ obtaining a fraction groupoid $\Frac(\cC)$ and fraction groups $\Frac(\cC,e)$ (that are isotropy groups at object $e$).
Now, the new part of Jones is that any functor $\Phi:\cC\to \cD$ starting from $\cC$ and ending in any category can produce an action of $\Frac(\cC)$ and thus of $\Frac(\cC,e)$.
Hence, if a complicated group $G$ can be expressed as $\Frac(\cC,e)$ from a somewhat simpler category $\cC$, then we can produce many actions of $G$ using the simpler structure $\cC$.
For instance, any isometry $H\to H\ot H$ between Hilbert spaces provides a unitary representation of $F$ (which in fact extends to $V$).
This is done by applying Jones' technology to the categorical description of the Thompson groups.
Moreover, Jones' technology is very explicit giving practical algorithms for extracting information of the action constructed such as matrix coefficients in the case of unitary representations.
Applications has been given in various fields such as: knot theory, mathematical physics, general group theory, study of Thompson's groups, operator algebras, noncommutative probability theory \cite{Jones19-survey,Aiello-Brothier-Conti21, Jones18-Hamiltonian, Brothier-Stottmeister19,Brothier-Stottmeister-P, Brothier19WP, Brothier-Jones19,Brothier-Jones19bis, Kostler-Krishnan-Wills20, Kostler-Krishnan22}. 

\subsection*{Program: wishes, motivations, and aims.}
Our program is to consider groups $G=\Frac(\cC,e)$ constructed from categories $\cC$ of diagrams that are well-suited for approximating CFT, but also for applying Jones' technology, and are interesting groups on their own.
More specifically:
\begin{itemize}
\item (Calculus of fractions)
We wish to be able to decide easily if a category of diagrams $\cC$ admits a  cancellative {\it calculus of fractions} requiring $\cC$ to be left-cancellative and satisfying Ore's property (i.e.~if $t,s\in\cC$ have same target, then there exists $f,g\in\cC$ satisfying $tf=sg$). This permits to construct groups.
\item (From subfactors to CFT) Given any subfactor, Jones constructed from its planar algebra $\cP$ a field theory where Thompson's group $T$ took the role of the space-time diffeomorphism group.
With $\cP$ fixed we wish to perform a similar construction but having a more sophisticate and specific symmetry group $\Frac(\cC,e)$. 
\item (Exceptional group properties) Forgetting about CFT and subfactors we wish to construct fraction groups $\Frac(\cC,e)$ that are interesting on their own: groups satisfying some of the exceptional properties of the Thompson groups. This would provide new examples of rare phenomena in group theory but also would help understanding why the Thompson groups are so special. 
\item (Applying Jones' technology)
We wish to apply efficiently Jones' technology to the fraction groups $\Frac(\cC,e)$ for constructing actions. 
This would occur if $\cC$ admits a concrete presentation with few generators as a category or better describe by explicit skein relations of diagrams (see below).
\item (Operator algebras) The {\it Pythagorean algebra} $P$ was a C*-algebra that naturally appeared when applying Jones' technology to the Thompson groups \cite{Brothier-Jones19bis}.
It provides a powerful method for constructing representations of the Cuntz algebra using only {\it finite dimensional} operators, new class of representations of the Thompson groups, and moreover exhibits a rare example of a non-nuclear C*-algebra satisfying the lifting property \cite{Cuntz77,Brothier-Jones19bis,Brothier-Wijesena22,Brothier-Wijesena22bis,Courtney21}.
We wish to extend these connections and applications in a larger setting.
\end{itemize}

\subsection*{Our class of diagrams and structures.}
{\bf Between forests and string diagrams.}
Here is our choice that we believe fit into all the constraints and wishes enunciated above.
This class is inspired by the tree diagrams used by Brown to describe elements of Thompson's group $F$ and string diagrams appearing in Jones' planar algebras.

Recall that an element of $F$ is described by a class of pairs of trees $(t,s)$ that have the same number of leaves.
Here, {\it tree} means a finite ordered rooted binary tree. Two such pairs defined the same element of $F$ when they are equal up to removing or adding corresponding carets on each tree.
We often describe the element associated to $(t,s)$ by a diagram in the plane with the roots of $t$ on the bottom and its leaves on top that are connected to the leaves of $s$ where $s$ is placed upside down on top of $t$.
Now removing a pair of corresponding carets is the operation of removing a diamond as described in the following diagrams:
\[\includegraphics{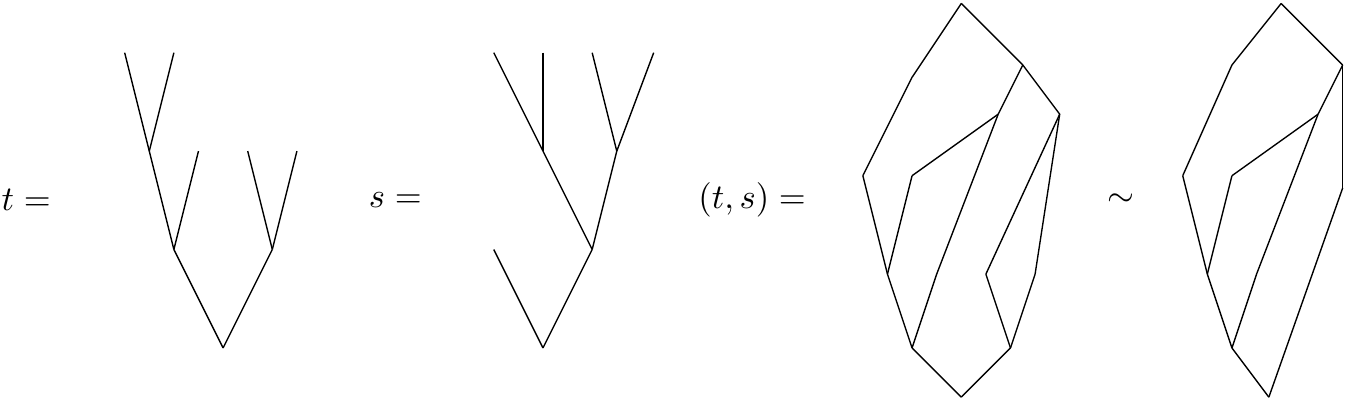}\]
More generally, we can consider pairs of {\it forests} rather than {\it trees} and obtaining the Thompson groupoid.

Elements of Jones' planar algebras (or other similar structures like Etingof-Nikshych-Ostrik's fusion categories or rigid C*-tensor categories) are described by linear combinations of {\it string diagrams} \cite{Jones99,Etingof-Nikshych-Ostrik05}.
String diagrams are planar diagrams made of a large outer disc, some inner labelled discs, and some non-intersecting strings that start and end at certain boundary points of discs and perhaps some closed loops. 
Here is an example with three inner discs having 4,2,4 boundary points (when read from top to bottom and left to right) and one closed loop:
\[\includegraphics{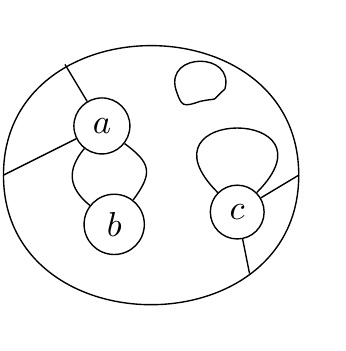}\]
A class of such diagrams, as originally defined by Jones, is generated by a collection of labelled $n$-valent vertices and satisfies relations expressed by linear combinations of string diagrams constructed from the generators.
We call the later {\it skein relations} (as in the context of knot theory and Conway tangles) that are identities closed under taking concatenation of diagrams. 
Explicit examples of skein relations of planar algebras can be found for instance in \cite{Peters10} or \cite[Section 3]{Morrison-Peters-Snyder10} or in the PhD thesis of Liu \cite{Liu-thesis}.
Having such planar presentations for Jones' planar algebras is powerful as it permits to express in few symbols very rich algebraic structures.

We define classes of certain planar diagrams called {\it forest-skein categories} using forest diagrams whose vertices are labelled and satisfy certain skein relations, although removing any linear structures. The set of skein relations completely describes the forest-skein category and provides a {\it skein presentation}. 

{\bf Forest-skein categories.}
We consider all (finite ordered rooted binary) forests that we represent as planar diagrams with roots on the bottom and leaves on top that we order from left to right.
We fix a set $S$ and colour each interior vertex of a forest with an element of $S$ obtaining a {\it coloured forest}.
Here is an example of coloured forest with three roots, eleven leaves, and colour set $S=\{a,b\}:$
\[\includegraphics{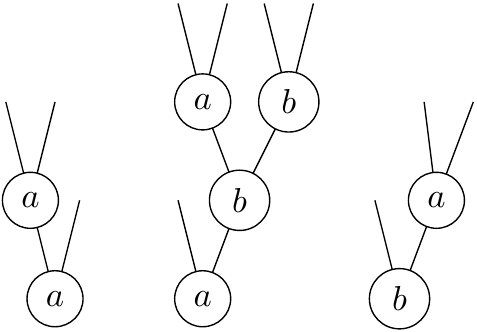}\]
Note, we do not colour the leaves just like boundary points of the exterior disc of a string diagram is not coloured in Jones' framework.
This produces a small category (where composition is given by horizontal stacking) that we call the {\it free forest-skein category over $S$}. 
This category is always left-cancellative.
Now, if $S$ is a single colour, then we recover the usual category of {\it monochromatic} forests whose fraction group is Thompson's group $F$.
If $S$ has more than two colours, then this category never satisfies Ore's property and thus does not admit any calculus of fractions nor produces fraction groups.

The idea for obtaining Ore's property is to mod out by some relations similar to skein relations.
These relations are expressed by a pair of coloured trees $(u,u')$ that have the same number of leaves.
Now, two forests $f,f'$ are equivalent if by substituting {\it subtrees} of $f$ isomorphic to $u$ by $u'$ (or the opposite) we can transform $f$ into $f'$.
The quotient is a small category $\cF$ of equivalence classes of diagrams that we call a {\it forest-skein category}.
Such a $\cF$ is expressed in a compact way by a {\it skein presentation} $(S,R)$ where $S$ is the set of colours and $R$ is a set of pairs of trees $(u,u')$ that are {\it skein relations}.
A skein presentation allows to apply Jones' technology efficiently and to study in an effective manner associated structures like fraction groups.

Here is an example of a forest-skein category $\cF$ presented by $(S,R)$ with two colours $S=\{a,b\}$ and one relation $R=\{(u,u')\}$ so that 
\[\includegraphics{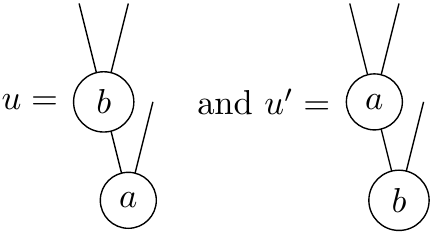}\]
For instance, the following two forests are equivalent:
\[\includegraphics{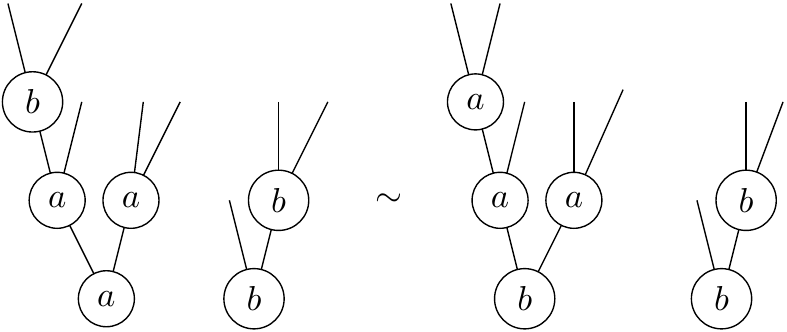}\]

{\bf Forest-skein monoids.}
A skein presentation $(S,R)$ defines a monoid $\cF_\infty:=\FM\la S|R\ra$ that we call a {\it forest-skein monoid}.
It is obtained by considering infinite forests with roots and leaves indexed by $\N_{>0}$ that have finitely many nontrivial trees.
It is indeed a monoid with composition being vertical stacking of diagrams.
Moreover, note that this monoid can be obtained from the forest-skein category by performing a certain inductive limit.
It is an auxiliary structure that is useful for studying $\cF$.

{\bf Forest-skein groups.}
In some occasion $\cF$ is a Ore category (i.e.~it is left-cancellative and satisfies Ore's property) and thus we may formally invert forests obtaining a {\it fraction groupoid} $\Frac(\cF)$.
Now, by considering any $r\geq 1$ we have an isotropy group $\Frac(\cF,r)$ corresponding to all $f\circ g^{-1}$ where $f,g$ have the same number of leaves and have $r$ roots.
We choose a favourite group $G:=\Frac(\cF,1)$ (produced with trees) which we call {\it the} fraction group of $\cF$ and call all these groups {\it forest-skein groups}.
Note that $\Frac(\cF,r)\simeq G$ for any $r$ since we work with binary forests (implying that the fraction groupoid $\Frac(\cF)$ is {\it path-connected} and thus all its isotropy groups are isomorphic).
Similarly, the forest-skein monoid $\cF_\infty$ provides a fraction group $H:=\Frac(\cF_\infty)$ that can be obtained as an inductive limit of the $\Frac(\cF,r)$ (letting $r$ tending to infinity).
All these groups embed in each other but are in general not isomorphic in an obvious way (unlike the classical Thompson group case).
Jones' technology is more effective for studying $G$ rather than $H$ and we thus focus more on $G$.

{\bf Previous constructions in the literature.} 
Examples of groups constructed from diagrams of forests with more than one caret have been previously considered. 
The first appeared in the work of Stein by allowing carets of various degrees, then in the work of Brin and the description of partition of hypercubes with several binary carets (one caret per dimension), and more recently in the work of Burillo, Nucinkis, and Reeves in the description of the Cleary irrational-slope Thompson group using two different binary carets \cite{Stein92,Brin-dV1, Burillo-Nucinkis-Reeves21}. 
They all act by transformation of the unit interval or higher dimensional cube that are piecewise affine.
We will come back to these examples but before let's say that the last example cited fits exactly into the framework of this article, the first fits in an extension of it, and the second not quite as we will explain later. 

{\bf Generalisation and extension of the formalism.}
In this first article we present the general theory of forest-skein categories and their forest-skein groups for coloured binary forests as briefly explained above. 
This produces forest-skein groups similar to Thompson's group $F$.
We list below some classical tree-diagrams that have been considered to construct Thompson-like groups by various authors. 
Each of these extension of Thompson's group $F$ can be applied and combined in our framework of forest-skein categories such as: 
\begin{itemize}
\item add permutations of leaves (obtaining Thompson's groups $T$ and $V$ see \cite{Brown87});
\item replace binary forests by $n$-ary forests for any $n\geq 2$ (by Higman and Brown obtaining Higman-Thompon's groups $V_{n,r}$ and the Brown subgroups $F_{n,r}, T_{n,r}$ \cite{Higman74,Brown87});
\item allow forests whose vertices have various degrees (by Stein \cite{Stein92});
\item label leaves with elements of an auxiliary group (independently by Tanushevski, Witzel-Zaremsky, the author \cite{Tanushevski16,Witzel-Zaremsky18, Brothier22, Brothier21});
\item add braids on top of leaves (independently by Brin and Dehornoy obtaining the braided Thompson group $BV$ \cite{Brin-BV1,Brin-BV2,Dehornoy06});
\item perform more sophisticate Brin-Zappa-Sz\'ep products between forests and an auxiliary category (by Witzel-Zaremsky via general cloning systems of groups \cite{Witzel-Zaremsky18}).
\end{itemize}
All these extensions are interesting and rich in applications. 
We can adapt these constructions in the more general framework of coloured forests with no more technical problems than in the classical case of the Thompson groups. 
In this first article we stick to the binary case for simplicity and clarity of the exposition. 
Although, we do explain how to incorporate permutations and braids using the approaches of Brown and Brin \cite{Brown87,Brin-BV1}.
This produces {\it $X$-forest-skein categories} denoted $\cF^X$ (for $X=F,T,V,BV$) and {\it $X$-forest-skein groups} similar to the three Thompson groups $F,T,V$ and the braided Thompson group $BV$.

{\bf Forest-skein groups among Thompson-like groups.}
Here, we want to locate our class of forest-skein groups among various Thompson-like groups existing in the literature. 
We have listed above a number of ways to generalise the notion of forest-skein groups but here we discuss about the smaller class of forest-skein groups obtained from forest-skein categories of coloured {\it binary} forests without any decoration on leaves, so no permutations, braids, nor auxiliary group, (so manifestly constructed like $F$).
We start by giving properties of forest-skein groups we have observed, then list some known groups that happen to be forest-skein groups, and finally discuss about few families of Thompson-like groups that are related or not to our class of groups.

{\it Properties and observations on forest-skein groups.}
A forest-skein group contains a copy of $F$. In particular, they inherent various properties of $F$ such as not being elementary amenable, having exponential growth, and having infinite geometric dimension.
Although, they do not share all properties of $F$.
Indeed, most of the examples we have encounter have torsion in their abelianisation and have sometimes torsion themselves. 
There are some forest-skein group that contains a copy of the free group of rank two.
Some forest-skein groups are nontrivial extension of $F$ and so are their derived group (preventing them to be simple).
There exist forest-skein groups that decompose as nontrivial direct products and can have some finite conjugacy classes or even a nontrivial center.
Every forest-skein group admits a nice simplicial complex on which it acts. 
Although, this complex does not have any obvious cubical structure and we don't know any CAT(0) cubical complex on which a generic forest-skein group acts.
However, via this complex, we will prove that many forest-skein groups are of type $F_\infty$: for each $k\geq 1$ they admit a classifying space with finitely many cells of dimension smaller than $k$.
As commented below, most Thompson-like groups have natural actions on an interval or a similar structure. No such actions seem to exists for forest-skein groups but we will construct a canonical totally ordered set on which they act.

{\it Some forest-skein groups, see Section \ref{sec:example}}

{\it Thompson's group.}
Thompson's group $F$ is a forest-skein group. It is the unique forest-skein group whose underlying forest-skein category is monochromatic.
Its (standard) skein presentation consists on one colour and no relations.

{\it Higman-Thompson group.}
The group $F_{3,1}$ obtained from ternary (monochromatic) trees is isomorphic to the forest-skein group arising from the skein relation given by
\[\includegraphics{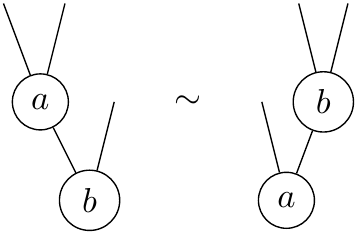}\]
By anticipating definitions and symbols given in Section \ref{sec:formalism} we can write its underlying forest-skein category by
$$\FC\la a,b| b_1 a_1 = a_1 b_2\ra.$$
The author was surprised to find an isomorphism and did not anticipated it from the diagrammatic description. The advantage of having a bicoloured description of $F_{3,1}$ is to have several embeddings of $F$ and $F_{3,1}$ inside it and to produce a priori new groups that are the $T,V,$ and $BV$-versions of it.

{\it Cleary's irrational-slope Thompson's group.}
The irrational-slope Thompson group of Cleary is a group of homeomorphism of the unit interval like $F$ but with elements having slopes powers of the golden ratio rather than powers of two \cite{Cleary00}.
Burillo, Nucinkis, and Reeves have provided a tree-diagrammatic description of it in \cite{Burillo-Nucinkis-Reeves21}, see also the follow up article with the $T$ and $V$-versions of this group \cite{Burillo-Nucinkis-Reeves22}.
It is easy to derive from this tree-like description that Cleary's group is a forest-skein group obtained from a forest-skein category with two colours and one relation
$$\FC\la a,b| a_1a_2=b_1b_1\ra.$$
We recommend the article of Cannon, Floyd, and Parry and the two of Burillo, Nucinkis, and Reeves as they provide very detailed and well-illustrated descriptions specific groups that can naturally be interpreted as forest-skein groups \cite{Cannon-Floyd-Parry96,Burillo-Nucinkis-Reeves21,Burillo-Nucinkis-Reeves22}.

{\it Brin's higher-dimensional Thompson groups and Belk-Zaremsky's extensions of it.}
Brin has constructed higher-dimensional Thompson groups $dV$ for each dimension $d\geq 1$ \cite{Brin-dV1,Brin-dV2}.
Informally, this is done by considering piecewise affine bijections between standard dyadic partitions of hypercubes of dimension $d$.
An element of $dV$ is described by two such partitions with same number of pieces and one correspondence saying which piece is sent to which.
A partition is described by a binary tree having carets coloured by $\{1,\cdots, d\}$ (one colour per dimension) and a labelling of the leaves corresponding to a labelling of the pieces of the partition.
Refining a partition corresponds to grow the associated tree. Unfortunately, they seem to not exists any meaningful way for labelling the leaves of the grown tree that is compatible with the labelling of the elements of the partitions, see the original articles of Brin but also an article of Burillo and Cleary for additional details \cite{Burillo-Cleary10}.
Hence, we may define a forest-skein group admitting similar diagrammatic description but it will not describe transformation of hypercubes and moreover are not isomorphic to Brin's groups, see Section \ref{sec:example}.
For $d=2$, such a forest-skein group admits the following skein presentation:
$$\FC\la a,b| a_1 b_1b_3=b_1a_1a_3\ra$$
which translates the two way to cut a square into four equal pieces. 
Note that Brin's construction extends to any nonempty set $S$ (including infinite sets) giving the group $SV$ of certain transformations of a hypercube where the axis are labelled by $S$.
Belk and Zaremsky have extended Brin's construction by adding in the data a group action $G\act S$ \cite{Belk-Zaremsky22}, see also the expository \cite{Zaremsky22}.
This provides {\it twisted Brin-Thompson groups} $SV_G$ that are particularly interesting when $S$ is infinite. 
Elements of $SV_G$ admit similar descriptions using forests with colour set $S$ but as in the original case of Brin this diagrammatic description must be taken with precaution.

{\it Auxiliary group and morphism.}
Given a group $\Ga$ and group morphism $\phi:\Ga\to\Ga\oplus\Ga$ one can construct a Thompson-like group $G(\Ga,\phi)$ obtained by taking binary (monochromatic) trees, decorating their leaves with elements of $\Ga$, and using $\phi$ for growing trees. 
This was independently discovered by Tanushevski, Witzel-Zaremsky, and the author \cite{Tanushevski16, Witzel-Zaremsky18, Brothier22}.
See the appendix of \cite{Brothier22} for details and a precise comparison of the three approaches.
We have discovered with surprise that this class of groups intersects nontrivially our class of forest-skein groups.
The intersection we found between the two classes is exactly the following.
All groups $G(\Ga,\phi)$ where $\phi(g)=(g,e)$ ($e$ being the neutral element of $\Ga$) and $\Ga=\ker(H\to \Z)$ where $H$ is the fraction group of a Ore monoid admitting a homogeneous presentation and where $H\to\Z$ is the word-length map of the presentation.
We will present in details this connection in the immediate successor of this article \cite{Brothier22-HPM}.

{\it Guba-Sapir's diagram groups.}
It is natural to wonder if forest-skein groups are related to Guba-Sapir diagram groups \cite{Guba-Sapir97}.
Indeed, both classes are constructed by planar diagrams similar to coloured forests (see in particular their dual diagrammatic representations appearing in \cite[Section 4]{Guba-Sapir97}), are related to the Thompson groups, and admit obvious $T$ and $V$-versions.
However, their resemblance seems misleading. 
Indeed, all diagram groups (the $F$-version) are torsion-free with torsion-free abelianisations and act freely properly and isometrically on a CAT(0) cubical complex in contrast with our examples of forest-skein groups.

{\it Hughes' FSS groups.}
Note that picture groups (i.e.~braided diagram groups in Guba-Sapir terminology corresponding to Thompson-like group similar to $V$) forms the same class than Hughes' finite similarity structure groups (in short FSS groups) by \cite{Hughes09, Farley-Hughes17}.
Hence, from our observations of above, it seems that FSS groups and forest-skein groups don't have a large intersection.

{\it Piecewise affine maps or action on trees.}
Many Thompson-like groups act as piecewise affine maps on an interval or a similar structure.
This is the case for the Higman-Thompson groups, Stein groups, Brin's $dV$, and more generally a class of groups defined by Farley and Hughes \cite{Farley-Hughes20}.
Similarly, they typically act on totally disconnected spaces like the Cantor space of binary sequences (by swapping finite prefixes) or more generally as almost automorphisms of trees. The family of groups just mentioned act in that way and other like R{o}ver-Nekrashevych groups, Brin's $dV$ groups and its generalisations like the one of Martinez-Perez and Nucinkis or Belk-Zaremsky's twisted Brin-Thompson groups \cite{Rover99,Nekrashevych04, Martinez-Perez-Nucinkis13, Belk-Zaremsky22}.

So far we have not find any canonical piecewise affine actions nor almost automorphism actions on trees for a generic forest-skein group. 
By construction our class of groups differ greatly from a number of previous families of Thompson-like groups and seems to behave differently as well.
Although, we found for any forest-skein group a canonical action on a {\it totally ordered space} which playes the role of the dyadic rationals of the torus for $F,T,V$, see Section \ref{sec:Q-space}.

{\it Thumann's operad groups.}
In his PhD thesis Thumann defines {\it operad groups} which play a special role in this article \cite{Thumann17}. 
A nice operad defines three operad groups: a planar, symmetric, and braided one corresponding to $F,V$, and $BV$.
The class of operad groups is huge and as far as the author observed all Thompson-like groups similar to $F,V,BV$ can be naturally identified with operad groups including our class of forest-skein groups.

Examples that seem to not admit an obvious description as operad groups are the one constructed using Brin-Zappa-Sz\'ep products as in the theory of cloning systems of Witzel-Zaremsky such as Thompson group $T$ constructed using cyclic permutations or $V^{mock}$ constructed from mock-symmetric permutations that seat in between $F$ and $V$, see \cite[Section 9]{Witzel-Zaremsky18}. Although, all these examples are probably isomorphic to operad groups anyway by encoding the Brin-Zappa-Sz\'ep product in the operad (like every group is the fraction group of itself).
Thumann provides a powerful abstract and very general framework.
Indeed, using some of Thumann's result we deduce that many forest-skein groups are of type $F_\infty$ (for any $n\geq 1$ the group admits a classifying space with finite $n$-skeleton, see Section \ref{sec:def-finiteness}).
We could present all the theory of forest-skein categories and groups using Thumann's operad approach. 
Although, we found many advantages in developing a more specific formalism that enables us to easily and explicitly construct new classes of examples, studying them with specific diagrammatic tools, and capture key features that we want all our structures to share in our program. 

\subsection*{Content of this present article.}
In this present article we introduce the formalism of forest-skein categories and their associated fraction groups called forest-skein groups.
In particular, we define {\it skein presentations} which will play a key role in applying Jones' technology in future articles.
We provide general criteria to decide that a forest-skein category admits a calculus of fractions.  
{\bf This is a key strength of our formalism: we are able to decide for many skein presentations if they can produce groups or not.}
We present two canonical actions of forest-skein groups: one on a simplicial complex giving a classifying space and another on a totally ordered space which mimics the classical action of the Thompson groups on the set of dyadic rationals of the unit torus or interval.
We deduce explicit presentations of forest-skein groups.

We do not apply Jones' technology nor explore connections with CFT in this article. 
Instead, we consider an exceptional property that the Thompson groups $F,T,V$ share. 
Indeed, we prove that a large class of forest-skein groups satisfy the topological finiteness property of being of type $F_\infty$.
It is a rare property for (infinite) groups that has pushed the study of $F,T,V$: Thompson-like groups are the main source of infinite simple groups with good finiteness properties such as $F_\infty$, see Section \ref{sec:def-finiteness} for definitions and references.
We end by providing a huge class of explicit forest-skein categories that are all Ore categories and whose fraction groups are of type $F_\infty$. 

As previously mentioned, the framework chosen has been inspired by Jones' planar algebra, Jones' reconstruction program of CFT, and Brown's diagrammatic description of Thompson's groups. 
The presentation does not use much categorical language but is rather set theoretical and constructive in the vein of the styles of Jones, Brin, and Dehornoy which inspired the author.
On the technical side we use previous work of Dehornoy on monoids and Thumann on operad groups.
The first permits to prove the existence of certain forest-skein groups and the second to establish a finiteness property.

{\bf Detailed plan and main results.}
In Section \ref{sec:formalism} we define forest-skein categories and monoids using presentations (generators and skein relations).
Forest-skein categories are monoidal small categories but we treat them as classical algebraic structures: sets equipped with two binary operations.
We define morphisms between them and in particular forest subcategories and quotients. 
Presented forest-skein categories can be alternatively defined as the solution of a universal problem.
This permits us to announce how Jones' technology will be used.
We provide explicit {\it category} presentations of forest-skein categories and monoids.

Section \ref{sec:Dehornoy} is about left-cancellativity and Ore's property. These are the two conditions we want for constructing groups.
We start by giving definition and some obvious but useful reformulations of them. 
Then we present some more advance techniques due to Dehornoy that we apply to our specific class of forest-skein categories and monoids. 
These techniques provide powerful characterisations and criteria for checking left-cancellativity of forest-skein categories and in fewer cases Ore's property.
In practice we prove Ore's property by explicitly constructing a cofinal sequence of tree.

Sections \ref{sec:forest-groups} and \ref{sec:presentation} are about forest-skein groups and presentations of them.
Given a presented Ore forest-skein category $\cF=\FC\la S|R\ra$ we define the forest-skein groups $G:=\Frac(\cF,1)$ and $H:=\Frac(\cF_\infty).$
Using obvious diagrammatic maps we prove that $G$ and $H$ embed in each other and moreover contain Thompson group $F$.
The monoid presentation of $\cF_\infty$ obtained in Section \ref{sec:formalism} provides a group presentation of $H$.
For $G$ we define a complex $E\cF$ canonically constructed from $\cF$ on which $G$ acts. 
The complex $E\cF$ is obtained by considering the two actions 
$$G\act \Frac(\cF)\curvearrowleft \cF$$
given by restricting the composition of the groupoid $\Frac(\cF)$.
From there we can define a $G$-poset and its associated order complex which is $E\cF$.
This latter is a free $G$-simplical complex and Ore's property of $\cF$ implies that $E\cF$ is contractible.
Hence, the quotient $B\cF:=G\bs E\cF$ is a classifying space of $G$ and in particular $G$ is isomorphic to its Poincar\'e group (taken at any point).
We deduce an infinite presentation of $G$ determined by a skein presentation $(S,R)$ and the choice of a colour $a\in S$.
Using a similar reduction performed for the classical Thompson groups we deduce presentations with less generators and relations that are sometime finite and obtain the following.

\begin{lettertheo}
Let $\cF=\FC\la S|R\ra$ be a presented Ore forest-skein category with forest-skein groups $G=\Frac(\cF,1)$ and $H=\Frac(\cF_\infty)$.
The groups $G,H$ admit explicit group presentations in terms of $S,R$, see Theorem \ref{theo:groupG-presentation}.
Moreover, if $S$ is finite (resp.~$S$ and $R$ are finite), then $G$ and $H$ are finitely generated (resp.~finitely presented).
\end{lettertheo}

Note that the converse is false: there exists a Ore forest-skein category that does not admit any finite skein presentation but its forest-skein group is finitely presented and in fact of type $F_\infty$, see Section \ref{sec:example}.
This theorem provides explicit finite presentations for many forest-skein groups and in particular recovers results of the literature like the presentation of Cleary irrational-slope Thompson group given in \cite{Burillo-Nucinkis-Reeves21}.
Note, in this article we only cover the case of $F$-forest-skein groups but strongly believe that group presentations can be obtained (with some substantial work) to the other $T,V,BV$ cases by combining our presentations of $\Frac(\cF,1)$ with the presentations of $T,V,$ and $BV$ given by Cannon-Floyd-Parry and by Brin \cite{Cannon-Floyd-Parry96,Brin-BV2}.

In Section \ref{sec:Betti} we consider the first $\ell^2$-Betti numbers $\betti(G)$ of forest-skein groups $G$ (restricting to the countable ones) originally defined by Atiyah \cite{Atiyah76}. Using a key result of Peterson and Thom saying that $\betti(G)\leq \betti(H)$ when $H\subset G$ satisfies a weak normality condition (called {\it wq-normal}) we are able to conclude that all our groups have trivial $\ell^2$-Betti numbers \cite{Peterson-Thom11}.
\begin{lettertheo}
If $G$ is a countable forest-skein group and $X=F,T,V,$ or $BV$, then $\betti(G^X)=0$ where $\betti$ stands for the first $\ell^2$-Betti number.
\end{lettertheo}
The proof proceeds as follows: we consider a Ore forest-skein category $\cF$ with countably many colours with associated groups $H:=\Frac(\cF_\infty)$ and $G:=\Frac(\cF,1)$. 
We construct a sequence of generators $g_1,g_2,\cdots$ of the group $H$ satisfying that $g_j$ commutes with $g_{j+1}$ and moreover all of them have infinite order. From there we can conclude that the copy of $\Z$ generated by $g_1$ inside $H$ is a wq-normal subgroup implying that $\betti(H)\leq \betti(\Z)$. We conclude that $\betti(H)=0$ via the well-known result: $\betti(\Z)=0$.
To get the group $G$ and more generally all the other versions $G^X$ for $X=T,V,BV$ we realise $H$ as a wq-normal subgroup of $G$ (implying that $\betti(G)=0$) then show that $G\subset G^X$ is wq-normal as well obtaining the remaining cases.

Note that $F$ was proven to have {\it all} its $\ell^2$-Betti numbers equal to zero by L\"uck \cite[Theorem 7.20]{Luck02}. 
A different proof was later given by Bader, Furman, and Sauer which could be extended to Thompson's group $T$ \cite[Theorem 1.8]{Bader-Furman-Sauer14}.

In Section \ref{sec:Q-space} we construct a canonical action $G^V\act Q_{\cF}$ for all $V$-forest-skein groups $G^V:=\Frac(\cF^V,1)$.
This extends the classical action of Thompson's group $V$ on the set of dyadic rationals $\Z[1/2]/\Z$ in the unit torus. 
We provide three equivalent descriptions of it when restricted to $G^T$. 
One is simply the homogeneous space $G^T/G$, the second is obtained by quotienting a piece of the fraction groupoid $\Frac(\cF)$, and the third is constructed using Jones' technology.
The later approach permits to extend {\it canonically} the action to $G^V$.
The poset structure of $Q_\cF$ permits to characterise $G^F$ (resp.~$G^T$) inside $G^V$ as the subgroups of order-preserving (resp.~order preserving up to cyclic permutations) transformations.
We prove some strong transitivity statements analogous to the classical Thompson group case. 
Using these results and adapting an elegant argument due to Brown and Geoghegan we deduce the following theorem on finiteness properties, see \cite[Section 4B, Remark 2]{Brown87} and \cite[Theorem 9.4.2]{Geoghegan-book}.

\begin{lettertheo}
Let $G$ be a forest-skein group and let $G^T$ be its $T$-version.
If $G$ satisfies the topological finiteness property $F_n$ for a certain $n\geq 1$, then so does $G^T$.
\end{lettertheo}

In most examples, a finiteness property holding for the $F$-version of a Thompson-like group generally holds for its $T$-version. Although, it is usually proved by following the whole argument of the $F$-case and adapting it step-by-step to the $T$-case.
Our theorem allows us to prove a permanence property that holds for all forest-skein groups regardless of the strategy adopted to prove that $G$ satisfies a given finiteness property.
There is a (obvious) homological version of this theorem but we have not found any application of it. 
However, the topological version of above have application as we are about to see.

Section \ref{sec:Thumann} is about Thumann's theory of operad groups and his finiteness theorem \cite{Thumann17}.
Thumann formalism and theorem are impressively general.
He covers and adapts in a single categorical approach a number of key techniques pioneered by Brown that have been refined and extended (in a nontrivial way and using new ideas) to study many more groups like Stein groups, Brin's higher dimensional $dV$, braided Thompson group $BV$, Guba-Sapir diagram groups and Hughes' finite similarity structure groups, to cite a few \cite{Brown87,Stein92,Fluch-Marschler-Witzel-Zaremsky13, Bux-Fluch-Marschler-Witzel-Zaremsky16,Farley-Hughes15}. 
We refer the reader to the articles of Zaremsky and Witzel for a first read on these techniques and strategy \cite{Zaremsky21,Witzel19}.
We prove that all forest-skein groups are operad groups and establish a dictionary between the two theories.
This allows us to precisely translate and specialise the theorem of Thumann to forest-skein groups.
We define the {\it spine} of a forest-skein category $\cF$ that is roughly speaking the mcm-closure of the set of trees with two leaves (where mcm stands for ``minimal common multiples'') and deduce the following.

\begin{lettertheo}\label{ltheo:Finfty}
If $\cF$ is a Ore forest-skein category with a finite spine, then the $X$-forest-skein groups $\Frac(\cF^X,1)$ is of type $F_\infty$ for $X=F,T,V,BV.$
\end{lettertheo}

Note that Thumann's result takes care of the $F,V,$ and $BV$-cases using three parallel technical arguments corresponding to planar, symmetric, and braided operad groups, respectively.
Using our previous theorem we add the missing $T$-case.
Moreover, note that this theorem covers a class of groups that is closed under taking braiding. 
This is remarkable but possible thanks to the difficult proof on the braided Thomson group $BV$ \cite{Bux-Fluch-Marschler-Witzel-Zaremsky16} that was extended very cleverly by Thumann and recovered by us in this context of forest-skein groups.
Using Dehornoy's criteria on left-cancellativity and specialising to the two colours case we deduce the following result.

\begin{lettercor}
If $\cF$ is a forest-skein category with two colours and one relation (a pair of trees with roots of different colours), then if it satisfies Ore's property then it is a Ore category (hence is left-cancellative) and the $X$-forest-skein groups $\Frac(\cF^X,1)$ is of type $F_\infty$ for $X=F,T,V,BV.$
\end{lettercor}

These results can be generalised in the spirit of Higman, Brown, and Stein by considering forests with vertices of various valencies and by considering groups formed by pairs of {\it forests} rather than trees, see Remark \ref{rem:Stein-Finfty}.
Although, these finiteness theorems have their limits. 
We present the seemingly very elementary forest-skein category $\cF:=\FC\la a,b| a_1b_1=b_1a_1, a_1a_1=b_1b_1\ra$ with two colours, two relations (of length two) in Section \ref{sec:rebel}. 
It produces a group $G$ of type $F_\infty$ but this fact is surprisingly difficult to prove and resists to Thumann's theorem and a second approach due to Witzel \cite{Witzel19}. We will prove $F_\infty$ in a future article using a third approach inspired from the work of Tanushevski, Witzel-Zaremsky, and connections with our work \cite{Tanushevski16,Witzel-Zaremsky18,Brothier22-HPM}.

We end this first article by providing a huge class of skein presentations $(S,R)$ for which the associated forest-skein category $\cF$ is automatically a Ore category and whose fraction group is of type $F_\infty$ (and so are its $T,V,$ and $BV$-versions).
Given any nonempty family $\tau=(\tau_a:\ a\in S)$ of monochromatic trees with the same number of leaves we define a forest-skein category $\cF_\tau$, see Section \ref{sec:class-example}.

\begin{lettertheo}
The forest-skein category $\cF_\tau$ is a Ore category whose spine injects in the index set $S$ union a point.
In particular, $\cF_\tau$ admits forest-skein groups $G^X_\tau:=\Frac(\cF^X_\tau,1)$ for $X=F,T,V,BV.$
If $S$ is finite of order $n$, then $G_\tau^X$ is of type $F_\infty$.
Moreover, the $F$-group $G_\tau$ admits an explicit presentation with no more than $4n-2$ generators and $8n^2-4$ relations. 
\end{lettertheo}

This theorem permits to construct easily and explicitly a plethora of interesting examples of groups.
In particular, if $S=\{a,b\}$, then each pair of $(t,s)$ of monochromatic trees with the same number of leaves provides a group $G_{(t,s)}$ of type $F_\infty$ and we can provide an explicit group presentation of $G_{(t,s)}$ with 6 generators and 28 relations.
Though, the skein presentation of $\cF_\tau$ is very small with two colours and one relation. This makes very easy to construct actions of $G_{(t,s)}$ using Jones' technology.
It would be interesting to know when they are pairwise isomorphic and which further properties they satisfy other than being of type $F_\infty$.

Adding up all properties of groups considered in this article we have constructed a large family of groups $G$ satisfying that
\begin{itemize}
\item $F$ embeds in $G$ implying that $G$ is non-elementary amenable, has exponential growth, has infinite geometric dimension (i.e.~does not admit any finite dimensional classifying space);
\item $G$ has an explicit presentation that can be read from the presentation of its underlying category;
\item some $G$ admits an explicit positive homogeneous presentation;
\item the first $\ell^2$-Betti number of $G$ is trivial;
\item $G$ admits a canonical action on a totally order set;
\item some $G$ is of type $F_\infty$ (i.e.~for each $k\geq 1$ it admits a classifying space with finitely many cells of dimension smaller than $k$).
\end{itemize}

{\bf Some future work.}
One of the main strength of our approach is to be able to easily construct new examples of groups.
In order to keep reasonable the length of this first article we have postponed a number of detailed study of specific examples and classes of examples to future articles.
The immediate successor to this paper is about a specific class of forest-skein categories \cite{Brothier22-HPM}.
They are built from homogenenously presented monoids and can be interpreted as forest-skein categories with {\it one-dimensional} skein relations.
They produce very specific forest-skein groups that are split extensions of the Thompson groups. 
We classify them and relate precisely certain properties of the monoids and the associated forest-skein groups. 

After this second article we will continue working on the program described earlier with an emphasise on constructing and studying specific examples or classes of examples. 
Choice of examples may be dictate by specific Jones' planar algebra and CFT or in view of establishing group properties or to construct interesting operator algebras.
This is a vast program that will demand time to develop and grow.
The author is happy to help and answer questions to anybody wishing to work on it.
We are also certainly not against suggestions, advices, and collaborations!

\begin{center}\begin{large}{\bf Acknowledgement}\end{large}\end{center}
I warmly thank Matt Zaremsky for his enthusiasm, for providing a number of judicious remarks on a first version of this preprint, and for helping me navigating into the bibliography.
I am happy to express my gratitude to Matt Brin for fruitful communications related to this work.

In September 2022 I gave a presentation on a previous version of this work at the Hausdorff Institute of Mathematics (HIM) in Bonn. Sri (Srivatsav Kunnawalkam Elayavalli) immediately suspected that these groups had trivial first $\ell^2$-Betti numbers and suggested a strategy to prove it. Sri was right and his strategy worked without much troubles. It is my pleasure to thank him and to acknowledge the warm hospitality of HIM which has made this exchange of ideas possible.

All figures have been created with Tikzit. We thank the programmers for making freely available such a great tool.

\section{A class of categories built from coloured forests}\label{sec:formalism}

In this section we define what is a {\it forest-skein category}. This is an algebraic structure similar than a monoid but with several units (a small category). It is roughly equal to a collection of forest that can be concatenated and whose vertices have various colours. 
We start by fixing the terminology and introducing what we call {\it coloured trees and forests}.
From there we introduce {\it free forest-skein categories} that are analogues of free monoids but obeying nontrivial relations induced by the diagrammatic structure.
Finally, we introduce relations and quotients of these free forest-skein categories. This leads to the notion of {\it skein presentations}, {\it forest-skein categories}, and morphisms between them.

\subsection{Coloured forests and trees}
\subsubsection{Monochromatic forests and trees}
We start by defining what we intend by {\it monochromatic} trees and forests (hence without colours). 
Consider the infinite regular rooted binary monochromatic tree $\cT_2$. 
The vertex set of $\cT_2$ is the set of all finite binary sequences of $0$ and $1$ (also called words) including the trivial one that is the {\it root} of $\cT_2$.
The edge set is equal to all sets $\{w, w0\}$ and $\{w,w1\}$ for $w$ a binary sequence.
We direct the tree keeping only the pairs $(w,w0)$ and $(w,w1)$ corresponding in having all edges going away from the root.
We call them {\it left and right edges}, respectively.
We say that $w0$ and $w1$ are the {\it children} of $w$ and $w$ is their {\it parent}. 
If $u=wv$ for some words $w,v$ with $v$ nontrivial, then $u$ is a {\it descendent} of $w$ and $w$ is an {\it ancestor} of $u$.
The union of $(w,w0)$ and $(w,w1)$ is called a {\it caret}.
We identify $\cT_2$ with a graph embedded in the plane with its root at the bottom and where a left edge is going to the top left and a right edge to the top right such as
\[\includegraphics{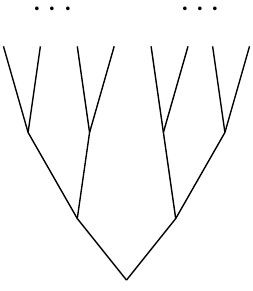}\]

A {\it monochromatic tree} $t$ is a finite rooted subtree of $\cT_2$ so that each of its vertex has either no children or two.
A vertex of $t$ with no children is called a {\it leaf} and otherwise an {\it interior vertex}.
We may drop the terms monochromatic if the context is clear.
We order from left to right the leaves of $t$ using the embedding of $\cT_2$ in the plane.
The {\it trivial tree}, denoted $I$, is a tree with a single vertex. 
This vertex is equal to its root and is at the same time its unique leaf.
For pictorial reason we represent $I$ by a vertical bar rather than a point.
There is a unique tree with two leaves that we denote $Y$. 
Note that $Y$ is equal to a single caret and more generally any tree with $n+1$ leaves is equal to the union of $n$ carets.

A monochromatic {\it forest} $f$ is a finite list of trees $(f_1,\cdots, f_n)$. 
We order the roots of $f$ from left to right. 
Note that every forest can be obtained from two trees. 
Indeed, if $s$ is a rooted subtree of $t$, then the complement of $s$ inside $t$ is a forest $f$ and all forest arise in that way.
We define composition of forests so that $t=s\circ f$. That is: $t$ is obtained by stacking vertically $f$ on top of $s$ and lining up leaves of $s$ with roots of $f$.
More generally we can compose any two forests as long as we match number of leaves and roots.
We write $\Leaf(f), \Root(f),\Ver(f)$ for the set of leaves, roots, and interior vertices, respectively, of a forest $f$. 
Note that the vertex set of $f$ is equal to the disjoint union of $\Ver(f)$ and $\Leaf(f)$. 
Moreover, $\Ver(f)$ is in bijection with the carets of $f$.

Here is the composition of two forests $f,g$.
The forest $f$ has two roots, three leaves, one interior vertex while $g$ has three roots, five leaves, and two interior vertices.
The composition $f\circ g$ is a forest with two roots, fives leaves, and three interior vertices.
\[\includegraphics{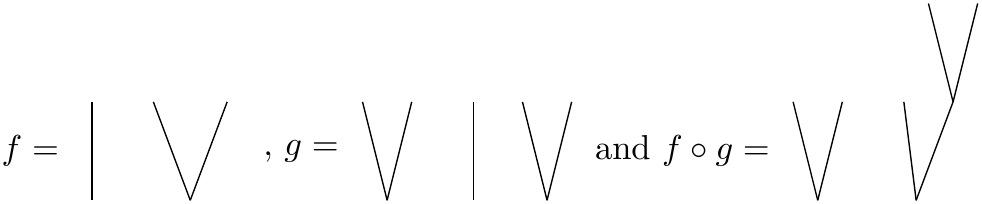}\]

\subsubsection{Coloured forests and trees}

We now add colours.
Let $S$ be a {\it nonempty} set.
A \textit{coloured forest (over $S$)} is a pair $(f,c)$ where $f$ is a monochromatic forest (as defined above) and a {\it colouring map} $c:\Ver(f)\to S$. 
A coloured tree is a coloured forest with one root.
We call $S$ the {\it set of colours} and its element a {\it colour}.
If the context is clear we may drop the term colour and write $f$ for $(f,c)$.
A caret coloured by $a$ or a {\it $a$-caret} is a caret whose origin is a vertex coloured by $a$.
We may interpret a monochromatic forest as a forest coloured by a single colour.

{\bf Notation.}
Typically we write $a,b,c$ or $x,y,z$ for colours, $t,s,u,v$ for trees, $f,g,h,k$ for forests. 
Recall that $I$ designates the trivial tree. 
For each colour $a\in S$ there is a unique tree with two leaves of colour $a$ which we denote by $Y_a, Y(a).$

\subsection{The collection of all forests, operations, quotient}

Let $S$ be a set of colour and write $\cUF=\cUF(S)=\FC\langle S\rangle$ for the set of all coloured forests over $S$. 

\subsubsection{Compositions and tensor products of forests}

We define two binary operations on $\cUF$.

{\bf Composition.}
The first one is called \textit{composition} and is only partially defined.
Consider $n\geq 1$, $(f,c_f),(g,c_g)\in\cUF$ so that $f$ has $n$ leaves and $g$ has $n$ roots.
Let $f\circ g$ or simply $fg$ be the forest obtained by concatenating vertically $g$ on top of $f$ where the leaves of $f$ are lined up with the roots of $g$: the $j$th leaf of $f$ gets attached to the $j$th root of $g$ for $1\leq j\leq n.$
We have the following obvious identifications: 
$$\Ver(f\circ g)=\Ver(f)\sqcup \Ver(g), \ \Root(f\circ g) = \Root(f), \text{ and }\Leaf(f\circ g)=\Leaf(g).$$
Define the colouring map $c_{f\circ g}:\Ver(f\circ g)\to S$ so that $c_{f\circ g}(v)=c_f(v)$ if $v\in \Ver(f)$ and $c_g(v)$ otherwise.
The composition of $(f,c_f)$ with $(g,c_g)$ denoted by $(f,c_f)\circ (g,c_g)$  is equal to $(f\circ g, c_{f\circ g}).$
We will often omit the colouring maps.
Diagrammatically we express as follows the composition of two coloured forests where the colours are $a,b,c$:
\[\includegraphics{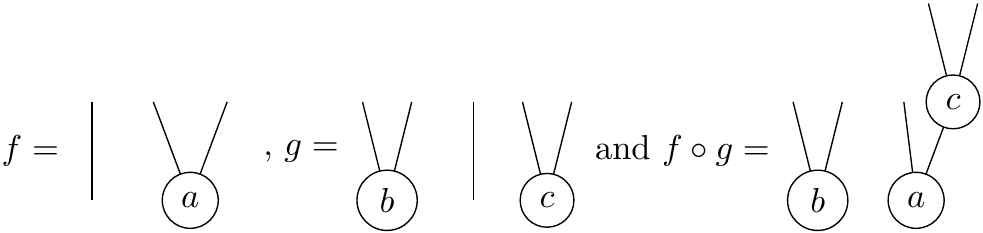}\]

{\bf Convention.} 
Given $f,g\in\cUF$, if we write $f\circ g$ we implicitly assume that this composition makes sense meaning that $f$ has its number of leaves equal to the number of roots of $g$.
In that case we may say that the pair $(f,g)$ is composable. 
Similarly, we may consider $k$-tuples that are composable.

{\bf Tensor product.}
We define a second binary operation called the \textit{tensor product} that is defined everywhere.
Consider two forests $(f,c_f),(g,c_g)\in\cUF$. 
Write $f\ot g$ for the forest obtained by concatenating horizontally $f$ with $g$ where $f$ is placed to the left of $g$.
We have the following obvious identifications
$$\Ver(f\ot g) = \Ver(f)\sqcup \Ver(g),\ \Root(f\ot g) = \Root(f)\sqcup \Root(g),$$
and 
$$\Leaf(f\ot g) = \Leaf(f)\sqcup \Leaf(g).$$
Define the colouring map $c_{f\ot g}:\Ver(f\ot g)\to S$ so that $c_{f\ot g}(v)=c_f(v)$ if $v\in \Ver(f)$ and $c_g(v)$ otherwise.
The tensor product of $(f,c_f)$ with $(g,c_g)$ denoted $(f,c_f)\ot (g,c_g)$ is equal to $(f\ot g, c_{f\ot g}).$
We will often omit the colouring maps.
Here is an example:
\[\includegraphics{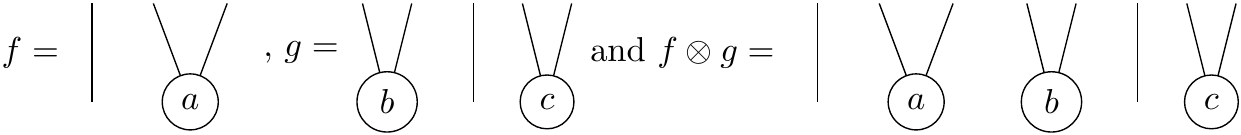}\]
Note that if we consider forests as lists of trees, then the tensor product is the concatenation of lists:
$$(f_1,\cdots,f_n)\ot (g_1,\cdots,g_m):=(f_1,\cdots,f_n, g_1,\cdots, g_m).$$

{\bf Subtrees and subforests.}
If $t$ is a tree, then $s$ is a \textit{rooted subtree} of $t$ if $s$ is a tree and there exists a forest $f\in\cUF$ satisfying $t=s\circ f$. 
A \textit{subtree} $u$ of a tree $t$ is a tree satisfying that there exist a rooted subtree $s$ of $t$ and some forests $f,g,h$ satisfying that 
$t= s\circ (f \ot u \ot g)\circ h$.
Similarly, we define rooted subforests and subforests.

{\bf Partial order.}
We define a partial order $\leq$ on $\cUF$.
If $f,g\in\cUF$, then we write $f\leq g$ if and only if there exists $h\in\cF$ satisfying $g=f\circ h$.
Hence, $f\leq g$ if and only if $f$ is a rooted subforest of $g$.
This is indeed a partial order. 
Moreover, if $f\leq g$, then $h\ot f\ot k\leq h\ot g\ot k$ and $r\circ f\leq r\circ g$ for all forests $f,g,h,k,r$.
If $f\leq g$, then we may say that $g$ is obtained by {\it growing} the forest $f$.

\begin{remark}
We have adopted the diagrammatic convention that the roots of a forest are at the bottom and the leaves on top. Moreover, $f\circ g$ stands for the forest with $f$ on the bottom and $g$ on top.
Hence, the composition can be perceived as concatenation of forests when reading the mathematical symbols from left to right corresponds in reading the diagram from bottom to top.
In the literature there exist all possible conventions; each of them having their advantages and drawbacks. 
We warn the reader that even the author have been using opposite conventions in previous articles. 
Hence, covariant functors appearing in \cite{Brothier21} would become contravariant functors for the conventions of this present article.
\end{remark}

\subsubsection{Equivalence relations on the set of forests}\label{sec:ER}

{\bf Relations and equivalence relations.}
Consider $S,\cUF$ as above and write $\cT$ for the set of trees of $\cUF$.
A \textit{skein relation in $\cUF$} or simply a \textit{relation} is a pair of trees $(t,t')$ in $\cT$ so that $t$ and $t'$ have the same number of leaves. 
If $R$ is a set of relations in $\cUF$, then we consider $\ov R\subset \cUF\times\cUF$ the smallest equivalence relation in $\cUF$ that contains $R$ and is closed under taking compositions and tensor products.
Hence, if $(f,f')\in \ov R$ and $g,h,k,p\in \cUF$, then $(g\circ f\circ h,g\circ f'\circ h)\in \ov R$ and $(k\ot f \ot p, k\ot f'\ot p)\in \ov R$.
Given $f,f'\in \cUF$ we write $f\sim_R f'$ or simply $f\sim f'$ for indicating that $(f,f')\in\ov R$.

{\bf Quotient.}
Let $R$ be a set of relations in $\cUF$ with associated equivalence relation $\ov R$.
We write 
$$\cF:=\FC\langle S| R\rangle$$ 
for the quotient of $\cUF$ with respect to equivalence relation $\ov R$.
If $f\in\cUF$, then we write $[f]$ for its class in $\cF$.
In order to keep light notations we may identify the equivalence class $[f]$ with a representative $f$.

{\bf Binary operations.}
By definition of $\ov R$ we have that if $f,f',g,g',h,h'\in\cUF$ satisfy $f\sim f', g\sim g', h\sim h'$, then $f\circ g\sim f'\circ g'$ and $f\ot h\sim f'\ot h'.$
This allows us to define on $\cF$ a composition and a tensor product as follows:
$$[f]\circ [g] := [f\circ g] \text{ and } [f]\ot [h]:=[f\ot h],$$
for $f,g,h\in\cUF.$
We will assume that $\cF$ is equipped with these two binary operations.

{\bf Partial order.}
The partial order $\leq$ of $\cUF$ provides a partial order on $\cF$ as follows.
Define the relation: $[f]\leq [g]$ if and only if $[g] = [f]\circ [h]$ for some $[h]\in\cF$ where $[f],[g]\in\cF$.
Note that $[f]\leq [g]$ if and only if there exists $f',g'\in\cUF$ in the classes of $[f],[g]$, respectively, satisfying $f'\leq g'.$

{\bf Elementary forests.}
If $a\in S$, $1\leq j\leq n$, then we write $a_{j,n}$ or simply $a_j$ for the forest having $n$ roots, its $j$th tree is $Y_a$, and all the other trees are trivial.
Hence, 
$$a_{j,n} = I^{\ot j-1} \ot Y_a \ot I^{\ot n-j}.$$
For instance:
\[\includegraphics{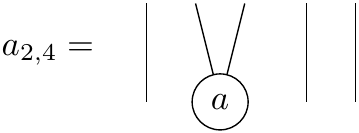}\]
Here we use the notation $I^{\ot m}$ for the forest equal to $m$ trivial trees place next to each other taking the convention that $I^{\ot m}$ is the empty diagram if $m\leq 0$.
We say that $a_{j,n}$ (or its class $[a_{j,n}]$) is an \textit{elementary forest} and write $\cE(S)$ for the set of all elementary forests coloured by $S$.

{\bf Forest with at most one nontrivial tree.}
By extending the notation of elementary forests we write $t_{j,n}$ for the forest with $n$ roots having its $j$th tree equal to $t$ and all other trivial.
For instance:
\[\includegraphics{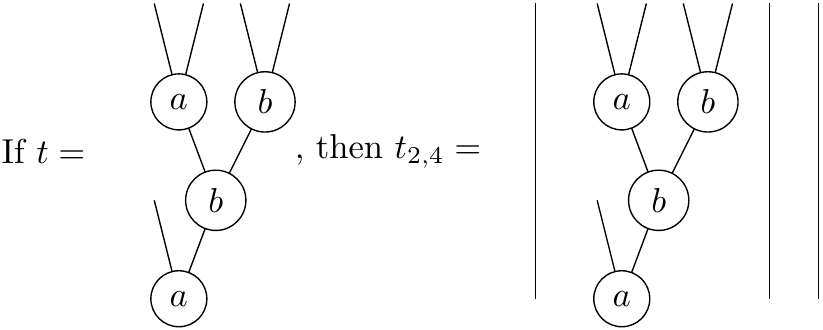}\]

{\bf Generators.}
It is not hard to see that any forest $f$ of $\cUF$ (resp. $[f]$ of $\cF$) is a finite composition of elementary forests.
Hence, the elementary forests are \textit{generators} of $\cUF$ and $\cF$ for the composition.

Note that the smallest subset of $\cUF$ (resp. $\cF$) containing $\{Y_a:\ a\in S\}$ (resp. $\{[Y_a]:\ a\in S\}$) that is closed under taking composition and tensor products with the trivial tree is equal to $\cUF$ (resp. $\cF$). 

\begin{remark}
Instead of taking pairs of {\it trees} for relations we could take pairs of {\it forests} $(f,f')$ providing they have the same number of leaves but also {\it roots}. This seems more general but in fact it is not.
Indeed, a relation of the form $(f,f')$ can be obtained from the set of all relations $(tf,tf')$ where $t$ runs over all trees composable with $f$.

Most of the examples investigate will have $S$ and $R$ finite. In fact, many interesting examples arise when $S$ has two elements $a,b$ and $R$ has one or two relations.
\end{remark}

\subsection{Forest-skein category}

\subsubsection{Algebraic structure}
We now analyse the algebraic structure of the set $\cF$ equipped with the two operations $\circ, \ot.$

The triple $(\cF,\circ,\ot)$ has an obvious structure of a monoidal small category, i.e.~a category whose collections of objects and morphisms are sets and equipped with a monoidal product (or tensor product).
We interpret $\cF$ as a category for the convenience of the terminology. 
Although, we think of $\cF$ as a classical algebraic structure: a set equipped with two composition laws.

{\bf Categorical structure.}
The set of object $\ob(\cF)$ of $\cF$ is the set of natural number $\N:=\{0,1,2,\cdots\}$ (taking the convention that $0$ is a natural number). 
If $m,n\in\N$, then we write $\cF(m,n)$ for the set of morphism from $m$ to $n$ equal to all forests having $m$ leaves and $n$ roots.
Hence, a morphism is a forest with source its number of leaves and target its number of roots.
Note that this is the opposite convention of \cite{Brothier21}. 
Observe that $\cF(m,n)$ is empty if $m<n$ and $\cF(m,m)$ contains one element: the trivial forest $I^{\ot m}$ with $m$ roots. 
This is the identity automorphism of $\cF(m,m)$.
Moreover, note that $\cF(m,0)$ is empty unless $m=0$. (Here we identify the empty forest with the trivial forest with $0$ roots.)
The composition is associative: it is clear in the free case since it corresponds in concatenating labelled graphs. The general case follows since it is a free forest-skein category mod out by an equivalence relation closed under composition.
All together we deduce that $(\cF,\circ)$ is a small category.

{\bf Monoidal structure.}
The monoidal structure is given by the binary operation $\ot.$
On objects it is defined as $n\ot m:=n+m$ for $n,m\in\N.$
For morphisms (i.e.~forests) it is defined as above by horizontal concatenations of forests.
The tensor unit is the empty forest. 
A similar reasoning than above shows that all the axioms of a monoidal category are satisfied.
Hence, $(\cF,\circ,\ot)$ is a monoidal small category.
We will continue to write $f\in\cF$ for a forest $f$ and will say that $f$ is an element of $\cF$ (rather than calling $f$ a morphism of the category $\cF$). 
Note that the object $0$ and the empty diagram are rather useless. We only added them in order to have a tensor unit and to fulfil precisely the axioms of a monoidal category.

\begin{definition}
\begin{enumerate}
\item The monoidal category $(\cF,\circ,\ot)$ is called a \textit{forest-skein category}.
\item Let $S$ be a nonempty set and $R$ a set of pairs of coloured trees over $S$ that have the same number of leaves.
Consider $\cF:=\FC\langle S| R\rangle$ as defined above and the binary operations of composition $\circ$ and tensor product $\ot$.
We say that $(S,R)$ is a {\it skein presentation}, $\FC\langle S| R\rangle$ a {\it presented forest-skein category}, $S$ the {\it set of colours}, and $R$ the {\it set of skein relations}.
We may drop the words ``forest'' and "skein" if the context is clear.
\item When the set of relations $R$ is empty we write $\FC\langle S\rangle$ or $\cUF\la S\ra$ rather than $\FC\langle S| \ \emptyset \rangle$ and call it the \textit{universal forest-skein category} or the \textit{free forest-skein category} over the set $S$ where the symbol $\cU$ stands for ``universal''.
\end{enumerate}
\end{definition}

It is sometime convenient to discuss about {\it abstract} forest-skein categories where a skein presentation has not been specified. Similarly, we may want to discuss about colours in the absence of a presentation.

\begin{definition}
\begin{enumerate}
\item An {\it abstract forest-skein category} is a monoidal small category $(\cC,\circ,\ot)$ with set of object $\N$ so that there exist a presented forest-skein category $\FC\langle S| R\rangle$ and a monoidal functor $\phi:\FC\langle S| R\rangle\to \cC$ that is the identity on object and is bijective on morphisms.
In that case we say that $(S,R)$ is \textit{a skein presentation} of the abstract forest-skein category $\cC$. 
\item An abstract forest-skein category is called \textit{free} or \textit{universal} if it admits a presentation $(S',R')$ where $R'=\emptyset$. 
\item An {\it abstract colour} of an abstract forest-skein category $\cF$ is a tree with two leaves.
\end{enumerate}
\end{definition}

\begin{remark}\label{rem:colour}
We may drop the term {\it abstract} if the context is clear.
Note that if $\cF=\FC\la S|R\ra$ is a presented forest-skein category, then $S\to\cF(2,1), a\mapsto Y_a$ realises a surjection from the set of colours $S$ to the set of abstract colours $\cF(2,1)$.
This justifies the last definition of abstract colours.
\end{remark}

\subsubsection{Forest-skein monoids}

{\bf Infinite forests.}
An \textit{infinite forest} $f$ (over $\cF$) is a sequence of trees $(f_n:\ n\geq 1)$.
Pictorially, we imagine $f$ has an infinite planar diagram contained in a horizontal strip of the plane with $f_{n}$ a tree placed to the left of $f_{n+1}$.
We define and order the leaves of $f$ from left to right just as we did for forests. 
We compose two infinite forests using vertical concatenations.
This provides a well-defined associative binary operation that is defined for {\it all} pairs of infinite forests.
We are interested in a smaller set of infinite forests.

{\bf Monoid of finitely supported forests.}
Let $f$ be an infinite forest over $\cF$.
The {\it support} of $f$ is the set of $n\geq 1$ satisfying that $f_n$ is not the trivial tree (the tree with one leaf).
Define $\cF_\infty$ to be the set of all {\it finitely supported} infinite forests over $\cF$.
Note that the composition of two finitely supported infinite forests is still finitely supported. 
Moreover, the trivial infinite forest is obviously in $\cF_\infty$ giving a unit for the composition.
We deduce that $(\cF_\infty,\circ)$ is a monoid.
We call $\cF_\infty$ the \textit{forest-skein monoid} associated to $\cF$.
If $P=(S,R)$ is a skein presentation of $\cF$, then we write 
$$\cF_\infty:=\FM\langle S| R\rangle$$
and say that $P$ is a {\it skein presentation} of the monoid $\cF_\infty$.
We may also call $\cF_\infty$ the \textit{forest-skein monoid over the presentation $P$}.

{\bf Notation.}
Note the distinction of symbols: $\cF=\FC\langle S|  R\rangle$ and $\cF_\infty=\FM\langle S|  R\rangle$ where $\FC$ and $\FM$ stand for forest-skein category and forest-skein monoid, respectively, and the symbol $\infty$ stands for forests with infinitely many trees.
We may write $I^{\ot\infty}$ for the unit of $\cF_\infty$.
If $f\in\cF_\infty$ is supported on $\{1,\cdots,n\}$, then we may write $f_1\ot f_2\cdots \ot f_n\ot I^{\ot \infty}$ for $f$ where $f_j$ is the $j$th tree of $f$. 

{\bf Elementary infinite forests and generators.}
If $a\in S$ and $j\geq 1$, then we write $a_j\in \cF_\infty$ for the infinite forest $I^{\ot j-1}\ot Y_a\ot I^{\ot \infty}$ having for $j$th tree $Y_a$ and all other trees trivial.
We say that $a_j$ is an \textit{elementary forest} of $\cF_\infty$ and write $\cE(S)$ or $\cE(S)_\infty$ or $S_\infty$ the set of all of them.
Note that we have a bijection
$$S\times \N_{>0}\to S_\infty, \ (a,j)\mapsto a_j$$
and $S_\infty$ generates the monoid $\cF_\infty$.
We extend the notation $a_j$ to all tree $t$ and index $j$ by writing $t_j$ for the infinite forest $I^{\ot j-1}\ot t\ot I^{\ot \infty}$ having a its $j$th tree equal to $t$ and all other trivial.

{\bf Partial order.}
If $\cF_\infty$ is a forest-skein monoid, then we can define a partial order $\leq$ such as:
$f\leq f'$ if there exists $p\in \cF_\infty$ satisfying $f'=f\circ p$.
This is indeed a partial order analogous to the partial order defined on forest-skein categories.

\begin{remark}\label{rk:forest-monoid}
(Forest-skein monoids obtained from directed systems)
Define the map 
$$\phi:\cF\to \cF_\infty,\ (f_1,\cdots,f_n)\mapsto (f_1,\cdots,f_n, I, I,\cdots)$$ where $(f_1,\cdots,f_n)$ is a finite list of trees.
Intuitively, $\phi(f)$ is obtained by adding infinitely many trivial trees to the right of $f$ and thus we may simply write $\phi(f)=f\ot I^{\ot \infty}.$
Observe that $\phi$ is preserving the composition: $\phi(f\circ g)=\phi(f)\circ \phi(g)$ for $f,g\in\cF$. 
Hence, $\phi$ is a covariant functor from $\cF$ to $\cF_\infty$.
Moreover, $\phi$ preserves the partial orders: if $f\leq f'$, then $\phi(f)\leq \phi(f')$.
It is easy to see that, as a function, $\phi$ is surjective but not injective.
Indeed, $\phi(f\ot I) = \phi(f)$ for all $f\in\cF$.
Although, if we consider the system of inclusion maps 
$$(\iota_{m,n}^k:\cF(m,n)\to \cF(m+k,n+k), f\mapsto f\ot I^{\ot k} \text{ for } k\geq 0 \text{ and } 1\leq n\leq m),$$
then the inductive limit of the directed system $(\cF(m,n), \iota_{m,n}^k:\ k\geq 0 \text{ and } 1\leq n\leq m)$ can be identified in an obvious manner with $\cF_\infty.$

(Exterior law)
Unlike the composition and the order we cannot extend the tensor product of $\cF$ to $\cF_\infty$.
However, one can tensor an element of $\cF$ with an element of $\cF_\infty$ obtaining a map:
$$\cF\times \cF_\infty\to \cF_\infty,\ (f, g)\mapsto f\ot g.$$
This justifies the notation $f\mapsto f\ot I^{\ot\infty}$ defining the map $\phi$ of above.
\end{remark}

\subsection{Usual presentations, universal properties, and Jones' technology}\label{sec:universal-property}

\subsubsection{Category presentation}\label{sec:cat-pres}
Consider a skein presentation $(S,R)$ with associated categories 
$$\cUF=\FC\la S\ra, \cUF_\infty=\FM\la S\ra, \cF=\FC\la S|R\ra, \cF_\infty=\FM\la S|R\ra.$$
Note that a {\it skein} presentation $(S,R)$ is not a presentation in the usual sense for $\cF$ or $\cF_\infty$
Our plan is to deduce from $(S,R)$ a category presentation of $\cF$ and a monoid presentation of $\cF_\infty$.
From there we will deduce universal properties of $\cF$ (and $\cF_\infty$) with respect to \textit{any} category (not only a forest-skein category).
{\bf This universal property combined with Jones' technology will permit us to construct actions of forest-skein groups. It is one of the main motivations of our work which will be extensively investigate in future articles.}

In order to not create confusions we systematically say in this section that $S,R$ are, respectively, sets of \textit{forest generators}, \textit{skein relations}, and $(S,R)$ is a \textit{skein presentation} of $\cF$.
For a (classical) category or monoid $\cC$ we use the terminology \textit{category generators} or {\it monoid generators}, etc., for the usual notions as defined in \cite[Section 1.4]{Dehornoy-book}.

{\bf Thompson-like relations.}
Observe that the diagrammatic structures provide the following category relations that are inherent to all forest-skein categories and forest-skein monoids coloured by $S$:
\begin{align}\label{eq:Thompson-relations}
\TR(S):=&\{(b_{q,n}\circ a_{j,n+1}\ , \ a_{j,n}\circ b_{q+1,n+1}):\ a,b\in S \text{ and } 1\leq j <q\leq n\};\\
\TR(S)_\infty:=&\{(b_{q}\circ a_{j}\ , \ a_{j}\circ b_{q+1}):\ a,b\in S \text{ and } 1\leq j <q\}.
\end{align}
We call them \textit{Thompson-like relations} (over the set $S$).
The first appears in $\cF$ while the second in $\cF_\infty$.

{\bf Skein relations give category relations.}
Consider a skein relation $(u,u')\in R$ which is a pair of trees with the same number of leaves.
The tree $u$ can be written (uniquely up to the Thompson-like relations) as a word of elementary forests so that its $k$th letter is of the form $y_{i_k,k}$ where $y\in S$ and $1\leq i_k\leq k$.
Similarly, write $u'$ as a word with letters in $\cE(S)$.
This provides a relation that we write $R(u,u',1,1)$.
Now, by shifting the $k$th letter $y_{i_k,k}$ to $y_{i_k+j-1, k+n-1}$ we obtain a new category relation $R(u,u',j,n)$ corresponding to the pair of forests $(u_{j,n},u'_{j,n})$ (here $u_{j,n}=I^{\ot j-1} \ot u\ot I^{\ot n-j}$).
Similarly, we obtain a one parameter family of monoid relations $R(u,u',j)$ by erasing the second index $n$.

The important point is that there are no additional category relations in $\cF$ and $\cF_\infty$ as stated in the following proposition.

\begin{proposition}\label{prop:universal-category}
If $\cF:=\FC\langle S| R\rangle$ is a presented forest-skein category, then 
\begin{align*}
&(\cE(S), \{ R(u,u',j,n):\ (u,u')\in R, 1\leq j\leq n\}\cup \TR(S)) \text{ and }\\
&(S_\infty, \{ R(u,u',j):\ (u,u')\in R, 1\leq j\}\cup \TR(S)_\infty)
\end{align*}
are category presentation of $\cF$ and monoid presentation of $\cF_\infty$, respectively.
\end{proposition}

\begin{proof}
First, consider the free forest-skein monoid case with colour set $S$.
An easy adaptation of the monochromatic case (which is a well-known fact, see for instance  \cite[Proposition 2.4.5]{Belk-PhD}) gives that $(S_\infty,\TR(S)_\infty)$ is a monoid presentation of $\cUF_\infty$.
From there we deduce the category presentation $(\cE(S),\TR(S))$ for the free forest-skein category $\cUF$.
Now, going from $\cUF$ to $\cF$ (resp.~$\cUF_\infty$ to $\cF_\infty$) is obvious by the definition of $\cF$ (resp.~$\cF_\infty$) which is the quotient of $\cUF$ (resp.~$\cUF_\infty$) by the equivalence relation $\ov R$ equal to the closure of $R$ by compositions and tensor products.
\end{proof}

\begin{example}A category presentation of the free forest-skein category $\cUF:=\FC\langle S\rangle$ is $(\cE(S), \TR(S)).$
The Thompson-like relation $(b_{3,3}\circ a_{2,4}, a_{2,3}\circ b_{4,4})$ is expressed diagrammatically as follows:
\[\includegraphics{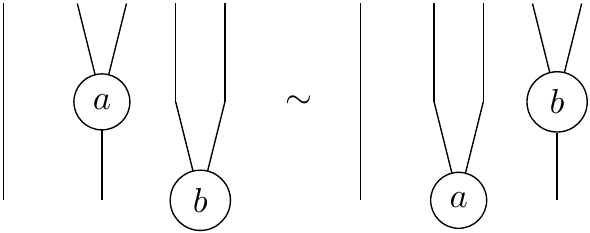}\]
Consider $\cF:=\FC\la a,b| a_1a_2=b_1b_1\ra.$
A category presentation of $\cF$ is given by
the category generator set 
$$\cE(a,b):=\{ a_{j,n}, b_{j,n}:\ 1\leq j\leq n\}$$
and the category relations
$$\begin{cases}(a_{j,n}a_{j+1,n+1} \ , \  b_{j,n} b_{j,n+1}) \text{ for all } 1\leq j\leq n\\
(x_{q,n}\circ y_{j,n+1}\ , \ y_{j,n}\circ x_{q+1,n+1}):\ x,y\in \{a,b\} \text{ and } 1\leq j <q\leq n\end{cases}.$$

Here is a diagrammatic description of one of the relations:
\[\includegraphics{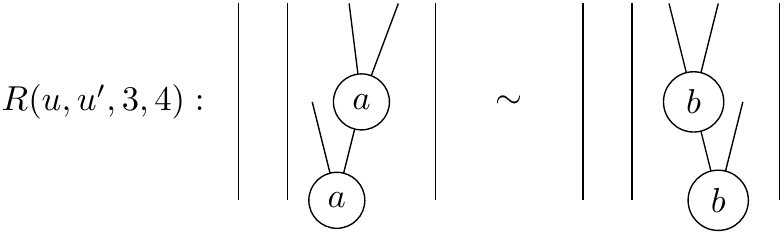}\]

A monoid presentation of $\cF_\infty$ is obtained by removing the second index of the generators, e.g.~$a_{j,n}\leftrightarrow a_j$.
\end{example}
 
\subsubsection{Universal property and Jones' technology}
The next corollary is a direct consequence of the previous proposition and the definition of category presentation.

\begin{corollary}
A presented forest-skein category $\cF:=\FC\langle S| R\rangle$ satisfies the following universal property.
Consider a category $\cC$ containing a set of morphisms $X$. 
Assume that there is a map $\phi:\cE(S)\to X$. 
Extend $\phi$ to finite paths of morphisms of $\cE(S)$ as
 $$f_1\circ \cdots\circ f_n \mapsto \phi(f_1)\circ \cdots \circ \phi(f_n).$$
Assume that, if $(f,f')$ is a category relation of $\cF$, then $\phi(f)=\phi(f')$.
Then there exists a unique covariant functor $\Phi:\cF\to \cC$ 
whose restriction to $\cE(S)$ is $\phi$.
\end{corollary}

\begin{remark}
In practice, we will be considering {\it contravariant} functors for constructing unitary representation or actions on groups via Jones' technology. Although, the covariant case is interesting as well. For instance, it produces actions of Thompson's groups on inverse limits such as totally disconnect compact spaces, profinite groups, and other.
\end{remark}

We will be mainly considering \textit{monoidal} functors and the following corollary. 
In that case we only need to consider a {\it skein} presentation rather than a {\it category} presentation. 
So far, it has been the most common way for producing explicit actions of groups via Jones' technology.

\begin{corollary}
Let $\cF:=\FC\langle S| R\rangle$ be a presented forest-skein category and $(\cC,\circ,\ot)$ a monoidal category.
Assume there exists an object $e$ in the collection of objects of $\cC$ and some morphism $\phi(a):e\to e \ot e$ for all $a\in S$.
Define $\phi$ on elementary forests such that 
$$\phi(a_{j,n}):= \id_e^{\ot j-1}\ot \phi(a)\ot \id_e^{\ot n-j} \text{ for all } a\in S, 1\leq j\leq n,$$
where $\id_e$ denotes the identity automorphism of the object $e$.

\begin{enumerate}
\item There exists a unique contravariant monoidal functor $\Phi: \FC\langle S\rangle\to C$ satisfying that $\Phi$ restricts to $\phi$ on the set of elementary forests $\cE(S)$.
\item If moreover, for each skein relation $(u,u')\in R$ we have that $\Phi(u)=\Phi(u')$, then $\Phi$ factorises uniquely into a contravariant monoidal functor $\ov\Phi:\cF\to \cC$.
\end{enumerate}
\end{corollary}

\begin{example}
Consider the presented forest-skein category $\cF:=\FC \langle a,b | a_1a_{2}=b_1b_{1}\rangle.$
Let $\Hilb$ be the category of (complex) Hilbert spaces having isometries for morphisms and monoidal structure given by the classical tensor product of Hilbert spaces. 
For any Hilbert space $H$ and isometries $A,B:H\to H\ot H$ satisfying
 \begin{equation}\label{eq:ABone}(\id_H\ot A)\circ A = (B\ot \id_H)\circ B.\end{equation}
 there exists a unique contravariant monoidal functor $\Phi:\cF\to \Hilb$ satisfying that $\Phi(Y_a)=A, \Phi(Y_b)=B.$
We will see in Section \ref{sec:example} that $\cF$ is a Ore category whose fraction group $G=\Frac(\cF,1)$ isomorphic to the Cleary irrational-slope Thompson's group \cite{Cleary00}.
Using Jones' technology (which we don't develop in this article) we can build explicitly from $\Phi$ a unitary representation of $G^V$ (the $V$-version of $G$ defined in \cite{Burillo-Nucinkis-Reeves22}).
Hence, any pairs of isometries $(A,B)$ satisfying Equation \ref{eq:ABone} provides a unitary representation of the three irrational-slope Thompson groups $G,G^T,$ and $G^V$.
\end{example}

\subsection{The collection of all forest-skein categories}

Our main object of study is a forest-skein category considered as an algebraic structure like a monoid or a group.
We will now define what are the right notion of maps between them. 
This leads us to introduce the category of forest-skein categories. This is a category that is not captured by classical set theory and we treat it in that way. 
The notion of morphism between forest-skein categories may seem restrictive at a first glance. However, it is the correct notion assuring that range of morphisms are forest-skein categories (leading to the notions of quotient and embedding).
We could spend more time defining various operations and constructions in this category (like free product, direct product, etc.) but leave it to future articles when we will derive applications of these operations.

\subsubsection{The category of forest-skein categories}

{\bf Collection of objects.}
Let $\Forest$ be the collection of all forest-skein categories. 
This is not a set since the class of all nonempty sets embeds in it via $S\mapsto \cUF\la S\ra.$

{\bf Morphisms.}
Consider two forest-skein categories $\cF,\ti\cF$.
A {\it morphism} $\phi$ from $\cF$ to $\ti\cF$ is a covariant monoidal functor from $\cF$ to $\ti\cF$ satisfying that $\phi(1)=1$ for the object $1$.
In other terms, $\phi$ is a map from $\cF$ to $\ti\cF$ such that if $f$ is a forest with $n$ roots and $m$ leaves, then $\phi(f)$ is a forest of $\ti\cF$ with $n$ roots and $m$ leaves.
Moreover, $\phi(f\circ g) = \phi(f)\circ \phi(g)$ and $\phi(f\ot h)=\phi(f)\ot \phi(h)$ for all $f,g,h\in\cF$.
We write $\Hom(\cF,\ti\cF)$ for the class of morphisms form $\cF$ to $\ti\cF$. 
Elements of $\Hom(\cF,\ti\cF)$ are typically written using the symbols $\phi,\varphi,\psi,\chi.$
Note, since there is only one tree with one leaf in each forest-skein category we have $\phi(I)=I$ and thus $\phi$ sends trivial forests to trivial forests. 
Similarly, it maps abstract colours to abstract colours.

\begin{proposition}
The collection of all forest-skein categories together with the collection of morphisms defines a category that we call the \textit{category of forest-skein categories} denoted $\Forest$.
\end{proposition}
\begin{proof}
A composition of covariant monoidal functor is a covariant monoidal functor implying that a composition of morphisms is a morphism.
Since morphisms are functions it is obvious that the composition is associative. 
For each forest-skein category we can consider the identity map which defines an identity functor.
All together we deduce that $\Forest$ is a category.
\end{proof}

{\bf Terminology.}
We will adopt a set-theoretical terminology for qualifying properties of morphisms.
Hence, we will say that $\phi:\cF\to\ti\cF$ is injective or surjective rather than $\phi$ is a monomorphism or an epimorphism, respectively.
It is not hard to see that if a morphism $\phi:\cF\to\ti\cF$ is bijective, then the inverse map is a morphism.
Hence, an isomorphism from $\cF$ to $\ti\cF$ is a bijective morphism $\phi:\cF\to\ti\cF$ with inverse written $\phi^{-1}.$
Similarly, an automorphism of $\cF$ is a bijective endomorphism of $\cF$.

\begin{proposition}\label{prop:colour-morphism}
Consider two presented forest-skein categories $\cF=\FC\langle S|  R \rangle$ and $\ti\cF=\FC\langle \ti S|  \ti R\rangle$. 
A morphism $\phi:\cF\to\ti\cF$ is characterised by its restriction to $\cF(2,1)$ (the set of trees of $\cF$ with 2 leaves). 
In particular, $\Hom(\cF,\ti\cF)$ is a set which is finite when $S$ is and there is an injective map:
$$r:\Hom(\cF,\ti\cF)\to \{S\to \ti S\}, \ \phi\mapsto r(\phi)$$ satisfying
$$\phi(Y_a) = Y_{r(\phi)(a)} \text{ for all } \phi\in \Hom(\cF,\ti\cF), a\in S.$$
\end{proposition}

\begin{proof}
This is a direct consequence of the universal property satisfied by a forest-skein category among monoidal categories, see Proposition \ref{prop:universal-category}.
\end{proof}

Here is a diagrammatic example. Consider a morphism $\phi:\cF\to\ti\cF$ and write $r$ the map $r(\phi)$ from the colour set of $\cF$ to the one of $\ti\cF$.
We have:
\[\includegraphics{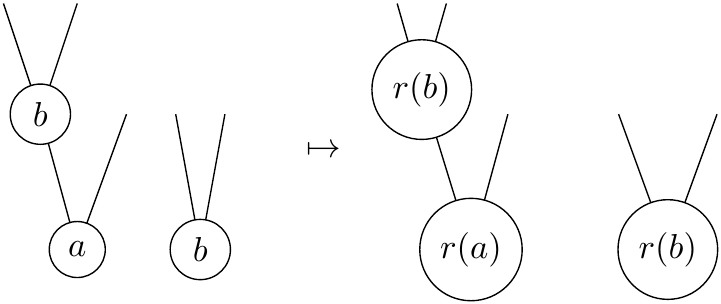}\]

\begin{remark}
Note that there exist covariant monoidal functors between forest-skein categories that are not {\it morphisms} in our sense. 
This means that there exists maps $\phi:\cF\to\cG$ that preserve compositions and tensor products but not the number of roots and leaves.
It rarely happens but it does. We give below an example where $\cF$ and $\cG$ are rather pathological. Note that in many cases a covariant and monoidal functor between two specific forest-skein categories will be automatically a morphism in our sense.
Consider the monochromatic forest-skein category $\cF$ that has only one tree $t_n$ with $n$ leaves for each $n\geq 1.$
This can be obtained using the skein presentation $(S,R)$ where $S$ is a singleton and where $R$ is the set of {\it all} pairs of trees $(t,s)$ with the same number of leaves.
If $\cT\subset \cF$, is the subset of trees, then we set:
$$\phi:\cT\to\cF, t_n\mapsto  t_{2n-1}\ot I.$$
Now, extend $\phi$ on $\cF$ as follows:
$$\phi:\cF\to\cF, f_1\ot\cdots\ot f_r\mapsto \phi(f_1)\ot \cdots \ot \phi(f_r)$$
for all forest $f=f_1\ot\cdots \ot f_r$ where each $f_j$ is a tree.
The map $\phi$ preserves composition implying that it is a covariant functor. Moreover, it is monoidal by construction. 
However, it is not a morphism of forest-skein categories because it doubles the number of roots and leaves.
\end{remark}

\subsubsection{Kernel, quotient, range, factorisation, and forest subcategory}\label{sec:forest-sub}
Consider two forest-skein categories $\cF,\cG$ and a morphism $\phi:\cF\to\cG$.

{\bf Kernel.}
The \textit{kernel} $\ker(\phi)$ of $\phi$ is the set of pairs $(f,f')\in\cF\times \cF$ satisfying $\phi(f)=\phi(f').$
Note that $\ker(\phi)$ is closed under taking composition, tensor product, rooted subforests, and is the graph of an equivalence relation denoted $\sim_\phi.$
If $R_\phi\subset\ker(f)$ is the subset of pairs of {\it trees}, then it is a set of skein relations that generates $\ker(f)$ (i.e.~the closure of this set for composition and tensor product is equal to $\ker(f)$).

{\bf Quotient and factorisation.}
Let $\chi_\phi:\cF\to \cF/\sim_\phi$ be the canonical quotient map.
Observe that $\cF/\sim_\phi$ has an obvious structure of forest-skein category. 
Equip with this structure $\chi_\phi$ is a surjective morphism of forest-skein categories.
Moreover, if $(S,R)$ is a skein presentation of $\cF$, then $(S,R\cup R_\phi)$ is a skein presentation of $\cF/\sim_\phi$.
Finally, the morphism $\phi:\cF\to\cG$ factorises uniquely into an injective morphism $\ov\phi:\cF/\sim_\phi\into\cG$.

{\bf Forest subcategory.}
The range $\phi(\cF)$ of $\phi$ has an obvious structure of forest-skein category and is isomorphic to $\cF/\sim_\phi$ via $\ov\phi$.
We say that $\phi(\cF)$ is a {\it forest subcategory} of $\cF$.
Note that $\phi(\cF)$ is generated by $\phi(\cF)\cap \cG(2,1)$ (its trees with two leaves) where {\it generated} means closed under composition, tensor product, and by adding the trivial tree.
Moreover, this characterises forest subcategories: $\{\cG(X):\ \emptyset\neq X\subset\cG(2,1)\}$ is the set of all forest subcategories of $\cG$.
Note, if $a$ is a colour of $\cG$, then $\cG(\{Y_a\})=\cG(a)$ is the forest subcategory of $\cG$ equals to all monochromatic forests coloured by $a$. 

{\bf Quasi-forest subcategories.}
We slightly extend the definition of above by saying that $\cG(X)$ is a {\it forest quasi-subcategory} of $\cG$ where $X$ is a nonempty subset of {\it trees} of $\cG$ (rather than a nonempty subset of trees with exactly two leaves).
In particular, if $t\in\cG$ is a nontrivial tree, then it defines a forest quasi-subcategory $\cG(\{t\})=\cG(t)$ of trees made exclusively with the ``caret'' $t$.

All these definitions and facts can be adapted to forest-skein monoids in the obvious manner. 

\subsection{Generalisations of forest-skein categories}
We briefly generalise our setting of forest-skein categories adding permutations and braids. This will be used to produce groups similar to Thompson's groups $T,V$, and $BV$.
These constructions are obvious extensions of the classical cases of Thompson's groups $T,V$ exposed in Cannon-Floyd-Parry and by Brin \cite{Cannon-Floyd-Parry96,Brin-BV1}, see also the PhD thesis of Belk \cite[Section 7.4]{Belk-PhD}.
Moreover, they appear as particular cases of the constructions of Thumann of symmetric and braided operads corresponding to the $V$ and $BV$ cases \cite[Section 3.1]{Thumann17}.

\subsubsection{Diagrams for groups}
We start by describing the symmetric, braid, and cyclic groups using diagrams.

{\it Symmetric group.}
Let $\Sigma_n$ be the symmetric group of order $n$: the group of bijections of $\{1,\cdots,n\}$.
Given $\sigma\in \Sigma_n$ we consider $[\sigma]$ to be the isotopy class of the diagram drawn in $\R^2$ equal to $\cup_{1\leq k\leq n} [ (k,1) , (\sigma(k),0)]$ where $[ (k,1) , (\sigma(k),0)]$ is the segment going from $(k,1)$ to $(\sigma(k),0)$ for $1\leq k\leq n.$
We interpret and call the points $(j,0)$ and $(k,1)$ as the $j$th root and $k$th leaf of $[\sigma]$, respectively.
Write $[\sigma]\circ[\tau]$ for the (isotopy class of the) horizontal concatenation of $\tau$ on top of $\sigma$ and note that $[\sigma]\circ[\tau]=[\sigma\circ \tau].$
Here is an example of two permutations of $\Sigma_4$ and their composition:
\[\includegraphics{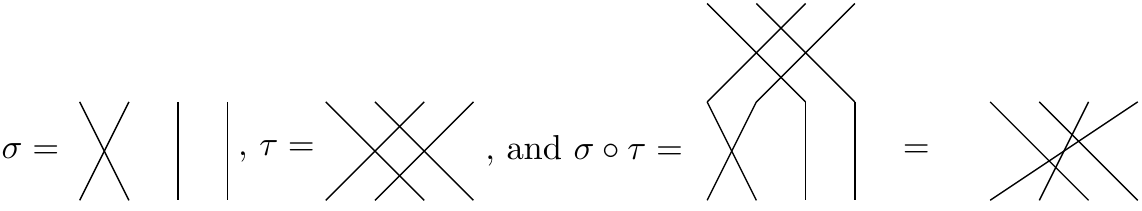}\]

{\it Braid group.}
By considering two types of crossing, over-crossing and under-crossing, we obtain the classical graphical description given by Artin of the braid groups $B_n$ over $n$ strands \cite{Artin47}.
Note, the usual group morphism $B_n\onto \Sigma_n$ corresponds in forgetting the distinction between over and under-crossings.

{\it Cyclic group.}
Consider the cyclic permutation $g:j\mapsto j+1$ of $\Sigma_n$ with $j$ written modulo $n$.
This provides an embedding $\Z/n\Z\to \Sigma_n, 1\mapsto g$. 
We identify $\Z/n\Z$ with this image in $\Sigma_n$ and consider the diagrammatic description of $\Z/n\Z$ deduced from the one of $\Sigma_n$.

\subsubsection{Free $X$-forest-skein categories}
{\it The $V$-case.}
Consider a free forest-skein category $\cUF$ over a set of colours $S$.
If $f\in\cUF$ is a forest with $n$ leaves and $\sigma\in \Sigma_n$, then we write $f\circ \sigma$ for the horizontal concatenation of $f$ with $[\sigma]$ on top of it where the $j$th leaf of $f$ is lined up with the $j$th root of $\sigma.$
Similarly, if $f$ has $r$ roots and $\tau\in \Sigma_r$, then $\tau\circ f$ stands for the diagram obtained by concatenating $f$ on top of $[\tau]$.
Now, $\tau\circ f = f^\tau\circ \tau^f$ for uniquely determined forest $f^\tau$ and permutation $\tau^f$ as explained in \cite[Section 7.4]{Belk-PhD} (up to adding colours).
Here is an example with $\tau\in \Sigma_3$ and $f$ a forest with three roots:
\[\includegraphics{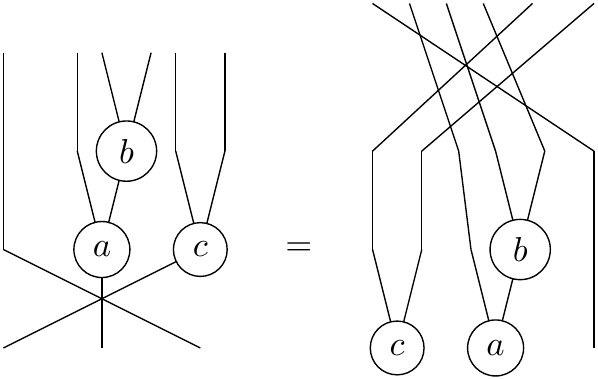}\]
This defines a composition and a small category $\cUF^V$ generated by a copy of $\cUF$ and a copy of $\Sigma_n$ for each $n$.
(We may interpret $\cUF^V$ as a Brin-Zappa-Sz\'ep product of the category $\cUF$ and the groupoid of all finite permutations.)
All elements of $\cUF^V$ (except the empty diagram) can uniquely be written in the form $f\circ \sigma$. 

{\it The $T$ and $BV$-cases.}
We observe that if $\tau$ is cyclic and $f$ is a forest, then $\tau^f$ is again a cyclic permutation. 
This allows us to define  $\cUF^T$ as the subcategory of $\cUF^V$ generated by $\cUF$ and all the cyclic groups $\Z/n\Z.$ 
Similarly, using the graphical composition explained in \cite{Brin-BV1} we define $\cUF^{BV}$ where now $\Sigma_n$ is replaced by $B_n$.

We say that $\cUF^X$ is the free $X$-forest-skein category over the set of colours $S$ for $X=F,T,V,BV$ where $\cUF^F:=\cUF$.

\subsubsection{$X$-forest-skein categories}

Consider a presented forest-skein category $\cF=\FC\la S|R\ra$ and write $\chi:\cUF\to\cF$ for the canonical morphism.
Observe that if $\chi(f)=\chi(f')$, then $f$ and $f'$ have the same number of roots and leaves. 
Moreover, if the $j$th tree of $f$ has $n_j$ leaves, then so does the $j$th tree of $f'$.
This implies that $\tau\mapsto \tau^f$ only depends on the class of $f$ inside $\cF$ where $\tau$ is either a permutation or a braid.
Moreover, $\chi(f)=\chi(f')$ implies that $\chi(f^\tau)=\chi(f'^\tau).$
This permits to define $\cF^X$ from $\cUF^X$ in the obvious manner. 
We call $\cF^X$ a {\it $X$-forest-skein category} or the {\it $X$-version} of the forest-skein category $\cF$.

\subsubsection{$Y$-forest-skein monoids}

A similar construction can be applied to a forest-skein monoid $\cF_\infty$.
Consider the group $\Sigma_{(\N)}$ of {\it finitely supported} permutations of $\N_{>0}.$
Elements of $\Sigma_{(\N)}$ can be represented as diagrams similar to above and together with $\cF_\infty$ forms a monoid $\cF_\infty^V$.
Replacing permutations by braids (that are finitely supported) we obtain a monoid $\cF_\infty^{BV}$.
Now, this construction does not apply to the $T$-case since there are no directed structures for the sequence of finite cyclic groups. 
If $Y=F,V,BV$ (hence missing $T$), then we call $\cF_\infty^Y$ the {\it $Y$-forest-skein monoid} or the {\it $Y$-version} of the monoid $\cF_\infty$ setting by convention $\cF_\infty^F=\cF_\infty$.

\section{Ore forest-skein categories and monoids}\label{sec:Dehornoy}
Our main objective is to construct groups from forest-skein categories using a left-cancellative {\it calculus of fractions}. 
This is equivalent for a category to be left-cancellative and to satisfy a property due to Ore, see \cite[Chapter I]{Gabriel-Zisman67} or \cite[Chapter 3]{Dehornoy-book} for details.
A category with these two properties will be called a \textit{Ore category}. 

In this section we investigate left-cancellativity and Ore's property.
First, we define and rephrase them via obvious but useful characterisations. 
Second, we provide more advanced techniques to decide if a forest-skein category is left-cancellative and in some fewer cases satisfies Ore's property. 
This second analysis is an adaptation of the work of Dehornoy on complete presentations of monoids \cite{Dehornoy03}.
It will provide very useful and powerful criteria that can be, in some cases, directly read from the skein presentation.

\subsection{Definitions and obvious characterisations}

\subsubsection{Cancellative}

A category $\cC$ is {\it left-cancellative} if $f\circ g=f\circ h$ implies $g=h$. 
Right-cancellativity is defined analogously and {\it cancellative} means left and right cancellative.
By choice we will only consider the property of being {\it left}-cancellative.
Here are some straightforward but useful observations.

\begin{observation}\label{obs:LC}
Consider a forest-skein category $\cF$ and its associated forest-skein monoid $\cF_\infty$.
For each $f\in\cF$ we consider the (partially defined) multiplication map:
$L_f: g\mapsto f\circ g.$
The following assertions are equivalent.
\begin{enumerate}
\item The forest-skein category $\cF$ is left-cancellative;
\item The forest-skein monoid $\cF_\infty$ is left-cancellative;
\item The map $L_t$ is injective for all tree $t$ of $\cF$;
\item The map $L_s$ is injective for all tree $s$ of $\cF$ with two leaves;
\item Any forest subcategory of $\cF$ (including $\cF$ itself) is left-cancellative.
\end{enumerate}
\end{observation}

\begin{example}\label{ex:LC}
\begin{enumerate}
\item A free forest-skein category is cancellative. 
\item 
The forest-skein category $\FC\langle a,b| a_1b_1=b_1a_1, \ a_1a_2=b_1b_2\rangle$ is not left-cancellative. Indeed, $Y_a(Y_a\ot Y_b)=Y_a(Y_b\ot Y_a)$ while $Y_a\ot Y_b\neq Y_b\ot Y_a.$
\item A monochromatic forest-skein category is left-cancellative if and only if it is free.
Here is a proof. 
Let $\chi:\cUF\to\cF$ be the quotient morphism of the monochromatic free forest-skein category onto a monochromatic forest-skein category.
Assume $\chi$ has a nontrivial kernel. Consider a pair of trees $(t,s)$ such that $t\neq s$ and $\chi(t)=\chi(s)$. 
Moreover, choose this pair $(t,s)$ so that the number of leaves of $t$ is minimal.
Such a pair exists by assumption.
Since there is only one monochromatic tree $Y$ with two leaves we necessarily have $Y\leq t, Y\leq s$ and they both decomposes as follows:
$$t=Y\circ(h\ot k) \text{ and } s=Y\circ(h'\ot k')$$
with $h,k,h',k'$ some monochromatic trees not all trivial.
Since $\chi(t)=\chi(s)$ and $\chi$ is monoidal we deduce that $\chi(h)=\chi(h')$ and $\chi(k)=\chi(k')$.
Therefore, $(h,h'),(k,k')$ are elements of the kernel of $\chi$ and are trees with stritcly less leaves than $t$.
This implies that $h=h'$ and $k=k'$ implying $t=s$, a contradiction.
\item The forest-skein category $\FC\la a,b| a_1^n=b_1b_2\cdots b_n\ra$ is left-cancellative for any $n\geq 1$.
In fact, we will see that all presented forest-skein categories with two colours and one relations of the form $a_1\cdots=b_1\cdots$ are left-cancellative.
\end{enumerate}
\end{example}

\subsubsection{Ore's property}
A category $\cC$ satisfies {\it Ore's property} if for all $f,g\in\cC$ with same target there exists $p,q\in\cC$ satisfying $f\circ p=g\circ q.$

\begin{remark}
Recall that a forest-skein category $\cF$ is equipped with a partial order $\leq$ defined as: $f\leq f'$ if there exists $p\in\cF$ satisfying $f'=f\circ p.$
Observe that $\cF$ satisfies Ore's property if and only if given $f,g\in\cF$ with the same number of roots there exists $h\in\cF$ satisfying that $f\leq h$ and $g\leq h$.
We may witness this property by only looking at the set $\cT$ of trees. We deduce that the poset of $(\cT,\leq)$ of trees is {\it directed} if and only if $\cF$ satisfies Ore's property.
\end{remark}

\begin{definition}
We say that a sequence of trees $(t_n)_{n\geq 1}$ is \textit{cofinal} if it is increasing (i.e.~$n\leq m$ implies $t_n\leq t_m$) and if for any $t\in\cT$ there exists $k\geq 1$ so that $t\leq t_k$. 
More generally, we consider cofinal {\it nets} of trees.
\end{definition}

Here are useful observations and characterisations of Ore's property. 

\begin{observation}\label{obs:Ore}
Let $\cF$ be a forest-skein category, $\cF_\infty$ its associated forest-skein monoid, and $\cT\subset\cF$ the subset of trees equipped with the usual partial order $\leq$.
The following assertions are equivalent.
\begin{enumerate}
\item The forest-skein category $\cF$ satisfies Ore's property;
\item The forest-skein monoid $\cF_\infty$ satisfies Ore's property;
\item The poset $(\cT,\leq)$ is directed;
\item The poset $(\cT,\leq)$ admits a cofinal net;
\item Every quotient of $\cF$ (including $\cF$ itself) satisfies Ore's property.
\end{enumerate}
\end{observation}
If $\cF(2,1)$ is countable (which implies that $\cF$ is countable), then the fourth item can be replaced by ``$(\cT,\leq)$ admits a cofinal {\it sequence}".

\begin{example}
\begin{enumerate}
\item Consider a nonempty set $S$ and the free forest-skein category $\cUF$ over $S$.
We have that $\cUF$ does not satisfies Ore's property unless $S$ has only one point.
Indeed, if $a,b\in S, a\neq b$, then there are no tree dominating both $Y_a$ and $Y_b$ inside $\cUF$.
\item The forest-skein category $\cF=\FC\la a,b| a_1b_1=b_1a_1\ra$ satisfies Ore's property and thus so does its quotient $\FC\la a,b| a_1b_1=b_1a_1 , \ a_1a_2=b_1b_2\ra$.
Note that the first is left-cancellative while the second is not.
\item The forest-skein category $\FC\la a,b| a_1^n=b_1b_2\cdots b_n\ra$ satisfies Ore's property for any $n\geq 1$.
\end{enumerate}
\end{example}

\begin{definition}
A small category $\cC$ that is left-cancellative and satisfies Ore's property is called a \textit{Ore category}.
In that case we say that $\cC$ admits a (left-cancellative) {\it calculus of fractions}.
\end{definition}

We will see in Section \ref{sec:forest-groups} how to construct (fraction) groupoids and groups from Ore forest-skein categories.

{\bf Warning.} In the literature, several different meanings are given to the term {\it Ore category}. It usually translates the existence of a calculus of fractions but requiring weaker or stronger conditions than ours.

\subsection{Dehornoy's criteria}

We present criteria due to Dehornoy which given a presentation provides tools to decide if the associated algebraic structure is left-cancellative or satisfies Ore's property.
We adapt and specialise these criteria to our situation of skein presentations.
Most details about these notions can be found in \cite{Dehornoy03}. Although, we will need several statements from other articles of Dehornoy that we will precisely cite.

{\bf Notations.} In all this section $\sigma$ is a nonempty set of \textit{letters} or \textit{generators}, $\word(\sigma)$ denotes the finite words in $\sigma$ including the trivial one denoted by $e$.
The \textit{length} $|u|=\length(u)$ of a word $u$ is the number of letters that composed $u$ taking the convention $|e|=0.$
We write $\rho$ for a set of pairs $(u,v)\in \word(\sigma)^2$ satisfying $|u|=|v|$.
Elements of $\rho$ are called \textit{relations}.
We consider the pair $\pi:=(\sigma,\rho)$ that we interpret as the presentation of a monoid $M$ that we denote by $\Mon\la \sigma|\rho\ra$ where $\Mon$ stands for {\it monoid}.
We continue to write $u\in\word(\sigma)$ for its image in the monoid $M$ and may write $u=_\rho v$ to express that two words $u,v\in \word(\sigma)$ are equal in the monoid $M$, i.e.~$(u,v)$ belongs to the smallest equivalence relation generated by $\rho.$

{\bf Terminology.}
We say that $\pi$ is a {\it homogeneous monoid presentation}. 
This means that all relations of $\rho$ are pairs of words $(u,u')$ in the letter set $\sigma$ such that $|u|=|u'|$. 
In particular, $u$ does not contain any inverse of letters of $\sigma$ (it is a monoid presentation). 
This last property is also called {\it positive}.

As usual we write $\sigma^{-1}$ for the set of formal inverses of elements of $\sigma$.
We consider the group $G=\Gr\la \sigma|\rho\ra$ with presentation $(\sigma,\rho)$ where $\Gr$ stands for {\it group}.

\subsubsection{Complete presentation}

{\bf Reversibility \cite[Definition 1.1]{Dehornoy03}.}

Consider words $w,w'\in \word(\sigma\cup \sigma^{-1})$ (hence composed of letters and formal inverse of letters) and say that $w\act^0 w'$ is true if $w'$ is obtained from $w$
\begin{enumerate}
\item by deleting an occurrence of $u^{-1}u$ for $u \in \word(\sigma)$ or;
\item by replacing an occurrence of $u^{-1} v$ by $v'u'^{-1}$ such that $(uv',vu')\in \rho$ and $u,v,u',v'\in \word(\sigma)$.
\end{enumerate}

We say that $w$ is (right-)reversible to $w''$ (or $w$ \textit{reverses} to $w''$) if we can go from $w$ to $w''$ with finitely many transformations as above.
We then write $w\act w''$ removing the superscript 0.
When there are several presentations involved we may add ``w.r.t.~the presentation $\pi$'' to express which kind of relations we are allowed to use in the reversing process.

{\bf Dehornoy's strong cube condition \cite[Definition 3.1]{Dehornoy03}.}

The presentation $(\sigma,\rho)$ satisfies the strong cube condition (in short SCC) at  $(u,v,w)\in \word(\sigma)^3$ if the following holds:

$$[u^{-1} w w^{-1} v \act v' u'^{-1} \text{ for some } u',v'\in \word(\sigma)]\Rightarrow (uv')^{-1}(vu')\act e.$$
Given $X\subset \word(\sigma)$ we say that $(\sigma,\rho)$ satisfies the SCC on $X$ if the cube condition holds for all $(u,v,w)\in X^3.$ 
Note that the order of $u,v,w$ matters this is why we use triples rather than sets.

\begin{observation}
If $(\sigma,\rho)$ is a presentation and $u\in \word(\sigma)$, then the SCC is satisfied at $(u,u,u).$
Moreover, note that if $u^{-1} w w^{-1} v \act v' u'^{-1}$, then necessarily $uv'=_\rho vu'$ (see \cite[Lemma 1.10]{Dehornoy03}). 
\end{observation}

{\bf Complete presentation.}
Dehornoy defined the very useful concept of \textit{completeness} for presentations. 
We only consider homogeneous monoid presentations. For those, there is a way to characterise completeness using the notion of SCC.
We use this characterisation as our definition.

\begin{definition}\cite[Proposition 4.4]{Dehornoy03}
A homogeneous monoid presentation $(\sigma,\rho)$ is called \textit{complete} if it satisfies the SCC at $\sigma$. 
\end{definition}

The following proposition is the main strength of the concept of completeness.

\begin{proposition}\label{prop:DehornoyRCF}\cite[Propositions 6.1 and 6.7]{Dehornoy03}
Let $(\sigma,\rho)$ be a complete presentation with associated monoid $M$.
The monoid $M$ 
\begin{enumerate}
\item is left-cancellative if and only if $u^{-1}v\act e$ for any relation $(au,av)\in \rho$ with $a\in \sigma, u,v\in \word(\sigma)$;
\item has Ore's property if and only if there exists $\sigma\subset \sigma'\subset \word(\sigma)$ so that for any $u,v\in \sigma'$ there exists $u',v'\in \sigma'$ satisfying $(uv')^{-1} (vu')\act e.$
\end{enumerate}
\end{proposition}

\subsubsection{The special case of complemented presentation}

We consider the specific case of {\it complemented} presentations introduced in \cite{Dehornoy97}.
Complemented presentations are easy to work with as there is at most one reversing process per pairs of words. 

{\bf Warning.} We warn the reader that {\it complete} and {\it complemented} are not the same notions. 
The first was defined above and stands for presentations that have no {\it hidden} relations in an intuitive sense while the second stands for presentations that have {\it few} relations as we are about to define.

\begin{definition}
A presentation $(\sigma,\rho)$ is {\it complemented} if for each $a,b\in \sigma$ there is at most one relation of the form $a\cdots=b\cdots$ in $\rho$.
When they exists we write $(a\bs b)$ and $(b\bs a)$ the unique elements of $\word(\sigma)$ satisfying $(a(a\bs b) \ , \ b(b\bs a))\in \rho.$ 
\end{definition}

If $(\sigma,\rho)$ is complemented, then for any pair of \textit{words} (not only letters) $(u,v)$ there is at most one pair of words $(u',v')$ satisfying $u^{-1} v\act v'u'^{-1}.$
Similarly, we write $u'=(v\bs u)$ and $v'=(u\bs v)$ when they exist so that $$u^{-1} v \act (u\bs v) (v\bs u)^{-1} \text{ and thus } u (u\bs v) =_\rho v(v\bs u).$$

The SCC translates as follows in this special case.

\begin{observation}
Let $(\sigma,\rho)$ be a complemented presentation.
The SCC at $(u,v,w)\in \word(\sigma)^3$ is satisfied if and only if either one of the $(x\bs y)$ is undefined for $x,y\in\{u,v,w\}$ or
$$[(u\bs v)\bs (u\bs w)] \bs [(v\bs u) \bs (v\bs w)]=e.$$
\end{observation}

We obtain the following characterisation of completeness for complemented presentations.

\begin{proposition}
A complemented presentation $(\sigma,\rho)$ is complete if and only if $$[(a\bs b)\bs (a\bs c)]\bs [(b\bs a)\bs (b\bs c)] \text{ is either undefined or equal to } e \text{ for each } a,b,c\in \sigma.$$
\end{proposition}

Here is a way to slightly reduce the number of SCC to check.
\begin{observation}\label{obs:complemented}
If $(\sigma,\rho)$ is a complemented presentation and $(u,v,w)\in\word(\sigma)^3$ are not all distinct, then the SCC is satisfied at $(u,v,w).$ 
\end{observation}

Note that the observation of above is no longer true in general for {\it non-complemented} presentations. 
The monoid presentation $\Mon\langle a,b| a^2=b^2, a^2=ba\rangle$ gives such an example at $(a,a,b)$.

\subsection{Applications of Dehornoy's techniques to forest-skein categories}
We now come back to our class of forest-skein categories. Our aim is to use Dehornoy's result for proving that certain forest-skein categories are Ore categories. We start by introducing and recalling notations and then outline the general strategy.

{\bf Notation.}
From now on $(S,R)$ is the skein presentation of a forest-skein category $\cF$ (implying by convention that $S\neq\emptyset$).
To $\cF$ is associated a monoid $\cF_\infty$ having the monoid presentation $P_\infty:=(S_\infty,R_\infty)$. 
Recall from Section \ref{sec:cat-pres} that 
$$S_\infty=\cE(S)_\infty=\{b_j:\ b\in S, j\geq 1\}$$ is the set of elementary forests of $\cF_\infty$
and 
\begin{align*}
& R_\infty := \cup_{n\in\N} R_n \cup \TR(S) \\
& R_n:=\{ (u_n,v_n):\ n\geq 1, (u,v)\in R\}\\
& \TR(S)=\TR(S)_\infty:= \{ y_q x_j = x_j y_{q+1}:\ x,y\in S, 1\leq j<q\}.
\end{align*}

By inspection we deduce the following.

\begin{observation}
The presentation $P_\infty$ is homogeneous and positive, i.e.~$R_\infty$ is a set of pairs of words (the letter set being $S_\infty$) of same length.
\end{observation}

Our general strategy for deducing properties of $\cF$ from Dehornoy's results on monoids is the following.

\begin{enumerate}
\item Check that $P_\infty$ is complete using the SCC at $S_\infty$;
\item Deduce properties of the monoid $\cF_\infty$ using Dehornoy's results;
\item Deduce properties for the category $\cF$ using Observations \ref{obs:LC} and \ref{obs:Ore}.
\end{enumerate}

We will often consider the case where $P_\infty$ is a {\it complemented} presentation. From there it is easy to check if $P_\infty$ is complete and left-cancellative. In some cases we can even establish Ore's property but this later property will be more likely proved by constructing a cofinal sequence of trees in $\cF$.

The first item is checking the SCC at each triple $(x_i,y_j,z_k)$ where $x,y,z\in S$ are colours and $i,j,k\geq 1$ indices.
It is not hard to see that the SCC is satisfied at such a triple if and only if it is satisfied at $(x_{i+1}, y_{j+1},z_{i+1})$. Indeed, if $(u,v)$ is a relation, then so is its shift ($I\ot u,I\ot v)$ and this shift corresponds in shifting by 1 all indices.
Therefore, when checking the SCC at $(x_i,y_j,z_k)$ we can always assume that $\min(i,j,k)=1$.
We start by considering an easier case where there are few relations.

\subsection{Complemented skein presentations}
We define complemented skein presentations and derive criteria of completeness for their associated monoid presentations.

\subsubsection{Definition and first observations}

\begin{definition}\label{def:complemented}
Let $P=(S,R)$ be a skein presentation.
We say that $P$ is {\it complemented} if for any two colours $a,b\in S$ there is at most one relation of $R$ of the form $a_1\cdots=b_1\cdots$ and moreover if $a=b$ then there are no such relations.
\end{definition}

Note that here we ask that there are no relations of $R$ where both words start by the same letter. This will give shorter statements and will cover all the cases we wish to study.
We have the following unsurprising fact.

\begin{proposition}
If $(S,R)$ is a complemented skein presentation, then the associated monoid presentation $(S_\infty,R_\infty)$ is complemented.
\end{proposition}

\begin{proof}
Consider a complemented skein presentation $(S,R)$ and two generators $x_i,y_j\in S_\infty$ where $i,j\geq 1$ and $x,y\in S$.
If $i\neq j$, then there is a unique Thompson-like relation of the form $x_iy_j=y_j x_{i+1}$ if $i>j$ and $y_jx_i=x_i y_{i+1}$ if $i<j$.
Moreover, there are no other relations in $R_\infty$ starting from these letters.
If $i=j$, then by assumption there is at most one relation of the form $x_i\cdots=y_i\cdots$ which belongs to $R_i$.
We conclude that $(S_\infty,R_\infty)$ is complemented.
\end{proof}

{\bf A few identities for complemented presentations.}
We recall and establish few elementary computations for complemented presentations in the special case of forest-skein categories.
These computations will allow us to easily check a number of SCC.
Let $(S,R)$ be a complemented skein presentation and $(S_\infty,R_\infty)$ the associated monoid presentation.
This later presentation is complemented in the usual sense and thus we have a partially defined map 
$$\word(S_\infty)^2\to \word(S_\infty), (u,v)\mapsto (u\bs v)$$
so that $$u^{-1} v\act (u\bs v)\cdot (v\bs u)^{-1}.$$
By plugging $e$ for $u$ we deduce
$$(e\bs v) = v \text{ and } (v\bs e) = e \text{ for all } v\in \word(S_\infty).$$

{\it Relations coming from the skein presentation.}
Assume that $(a_1u_1, b_1v_1)\in R$.
This produces a one parameter family of relations $((a_ju_j,b_jv_j) :\ j\geq 1)$ in $R_\infty$.
Here we have 
$$u_j=(a_j\bs b_j), v_j=(b_j\bs a_j) \text{ so that } (a_j\cdot (a_j\bs b_j), b_j\cdot (b_j\bs a_j))\in R_\infty.$$
Observe that both $u_j,v_j$ have same length $n$ (which does not depend on $j\geq 1$) and moreover $u_j$ is a word whose $k$th letter is a certain $c_m$ with $1\leq m\leq k+1$ for all $1\leq k\leq n.$

{\it Thompson-like relations.}
If $1\leq j<q$ and $a,b\in S$, then $a_q b_j = b_j a_{q+1}$ is a relation of $R_\infty$ implying that 
$$(a_q\bs b_j) = b_j \text{ and } (b_j\bs a_q)=a_{q+1} \text{ for all } 1\leq j<q, a,b\in S.$$
More generally, consider $w_1\in \word(S_\infty)$ which corresponds to a tree $t$ rooted at the first node of a certain length $n$ and write  $w_j$ for the word corresponding to $I^{\ot j-1}\ot t\ot I^{\ot \infty}$ obtained by shifting all indices by $j-1$.
We have the following identities:
\begin{align*}
& (a_q\bs w_j) = w_j \text{ and } (w_j\bs a_q) = a_{q+n}\\
& (a_j\bs w_q) = w_{q+1} \text{ and } (w_q\bs a_j) = a_j \text{ for all } a\in S, 1\leq j<q.
\end{align*}
This is obtained by observing that the first letter of $w_j$ has for index $j$, the second $j$ or $j+1$, etc., and by applying $n$ Thompson-like relations.
We obtain a similar statement for forests described by words $w$ rather than trees.

{\it Mix of both types of relations.}
Consider $x\in S$, some indices $1\leq j<q$, and a skein relation $(a_1 u_1,b_1 v_1)$.
This provides a relation $(a_ju_j,b_jv_j)$ where $u_j=(a_j\bs b_j)$ and $v_j=(b_j\bs a_j)$.
Now, $u_j$ is a word of a certain length $n\geq 0$ with $k$th letter of the form $y_{i_k}$ with $y\in S$ and $j\leq i_k\leq j+k+1.$
By applying $n$ Thompson-like relations we deduce:
$$x_{q+1} u_j = u_j x_{q+n+1}$$ giving the identity
$$[x_{q+1}\bs ( a_j\bs b_j)] =(a_j\bs b_j) \text{ and } [(a_j\bs b_j)\bs x_{q+1}] = x_{q+n+1}.$$
Now, if we exchange the indices we obtain that:
$$x_j u_{q} = u_{q+1} x_j$$ giving:
$$[x_j\bs (a_q\bs b_q)] = (a_{q+1}\bs b_{q+1}) \text{ and } [(a_q\bs b_q)\bs x_j] = x_j.$$

\subsubsection{Characterisation of completeness}

Here is a very useful characterisation of completeness which automatically provides left-cancellativity.

\begin{proposition}\label{prop:LC-complemented}
Let $(S,R)$ be a complemented skein presentation.
The associated monoid presentation $P_\infty=(S_\infty,R_\infty)$ is complete if and only if 
for all triple of {\it distinct} colours $(x,y,z)$ of $S$ we have that
\begin{equation*}[(x_1\bs y_1)\bs (x_1\bs z_1)]\bs [(y_1\bs x_1)\bs (y_1\bs z_1)]\end{equation*}
is either undefined or equal to $e$.
In that case the forest-skein category $\cF=\FC\la S|R\ra$ is left-cancellative.
\end{proposition}

Note that if $S$ has one or two colours and $(S,R)$ is complemented, then automatically $P_\infty$ is complete and $\cF$ is left-cancellative.

\begin{proof}
Since $(S,R)$ is complemented, then so is $(S_\infty,R_\infty)$ and the function $(u,v)\mapsto (u\bs v)$ is well-defined on a certain domain of $\word(S_\infty)^2.$
Now, $P_\infty$ is complete if and only if it satisfies the SCC at each triple of generators $(x_i,y_j,z_k)$ with $x,y,z\in S$ and $i,j,k\geq 1$.
Since the presentation is complemented the SCC at $(x_i,j_j,z_k)$ is satisfied if and only if the expression
\begin{equation}\label{eq:SCCprop}E:=[(x_i\bs y_j)\bs (x_i\bs z_k)]\bs [(y_j\bs x_i)\bs (y_j\bs z_k)]\end{equation}
is either undefined or equal to $e$.
Fix three colours $x,y,z\in S$ and some indices $i,j,k\geq 1$.

{\bf Claim: If $i,j,k$ are not all equal, then \eqref{eq:SCCprop} is undefined or equal to $e$.}

We inspect several SCC using identities established before the proposition.
\begin{enumerate}
\item Case 1: $i>j,k$.
We obtain
$$E=[ y_j\bs z_k]\bs [x_{i+1} \bs (y_j\bs z_k)].$$
Assume $(y_j\bs z_k)$ is defined since if not we are done.
Now, $i+1$ is strictly larger than the index of the first letter of the word $(y_j\bs z_k)$.
We deduce that $[x_{i+1} \bs (y_j\bs z_k)]= (y_j\bs z_k)$ and thus $E=e$.
\item Case 2: $i<j,k$.
We have
$$E=[y_{j+1}\bs z_{k+1}]\bs [ x_i \bs (y_j\bs z_k)].$$
Assume $(y_j\bs z_k)$ exists.
Since all the letters of $(y_j\bs z_k)$ have indices strictly larger than $i$ we deduce that $x_i \bs (y_j\bs z_k)= (y_{j+1}\bs z_{k+1})$ implying that $E=e$.
\item Case 3: $k>i,j$.
We have 
$$E=[(x_i\bs y_j)\bs z_{k+1}]\bs [(y_j\bs x_i) \bs z_{k+1}].$$
If $(x_i\bs y_j)$ is well-defined then so is $(y_j\bs x_i).$
Moreover, these two words have same length, say $m$.
We deduce that $(x_i\bs y_j)\bs z_{k+1}=z_{k+m+1}$ and $(y_j\bs x_i) \bs z_{k+1}=z_{k+m+1}$ implying $E=e$.
\end{enumerate}
The remaining three cases can be treated likewise using similar identities and are left to the reader. This establishes the claim.

Consider now that $i=j=k$. By symmetry of the problem we may assume $i=1.$
Moreover, by Observation \ref{obs:complemented} the SCC is satisfied whenever $x,y,z$ are not all distinct.
By characterisation of completeness we deduce that $P_\infty$ is complete if and only if the SCC is satisfied at $(x_1,y_1,z_1)$ for each triple $(x,y,z)$ of distinct colours.
If this is the case, we have that $\cF_\infty$ is left-cancellative if there are no relations in $R_\infty$ with both words starting with the same letter. 
This is indeed the case by the definitions of $R_\infty$ and of complemented skein presentation.
This implies that $\cF$ is left-cancellative since $\cF_\infty$ is.
\end{proof}

We deduce a criteria for having Ore forest-skein categories when there are few relations of small length.

\begin{corollary}\label{cor:Ore-FC}
Let $(S,R)$ be a skein presentation, $\cF$ its associated forest-skein category, and assume that all relations of $R$ are of the form $a_1x=b_1y$ where $a,b\in S$, $a\neq b$, and $x,y$ are letters. Moreover, if $a,b\in S, a\neq b$, then there exists one relation $a_1x=b_1y$. 

If $$[(x_1\bs y_1)\bs (x_1\bs z_1)]\bs [(y_1\bs x_1)\bs (y_1\bs z_1)]=e$$ for all triple $(x,y,z)$ of distinct colours, then $\cF$ is a Ore forest-skein category.
\end{corollary}

\begin{proof}
Consider $S,R,\cF$ as above.
By assumption $(S,R)$ is complemented. 
The associated monoid presentation $P_\infty=(S_\infty,R_\infty)$ is complete by the last proposition.
Moreover, there are no relations in $R_\infty$ starting by the same letter implying that $\cF_\infty$ is left-cancellative by Proposition \ref{prop:DehornoyRCF}.
Finally, all relations of $R_\infty$ are pairs of words of length two. This implies that the set of letters $S_\infty$ is closed under reversing and thus satisfies the second item of Proposition \ref{prop:DehornoyRCF}. 
Therefore, $\cF_\infty$ is a Ore monoid and thus $\cF$ is a Ore category.
\end{proof}

We deduce that skein presentations with zero or one relation are complemented and complete.

\begin{corollary}\label{cor:presentation-free}
Consider the skein presentation $(S,R)$ where $R$ is either empty or equal to a single relation of the form $a_1\cdots=b_1\cdots$ with $a\neq b$. 
Then the associated monoid presentation $(S_\infty,R_\infty)$ is complemented,  complete, and the forest-skein category $\FC\la S|R\ra$ is left-cancellative.
\end{corollary}

\begin{proof}
By definition $(S,R)$ is complemented and so is $(S_\infty,R_\infty)$
We can thus use the last proposition to check if $(S_\infty,R_\infty)$ is complete which will automatically imply that $\FC\la S|R\ra$ left-cancellative.
Consider a triple of distinct letters $(x,y,z)$ and let us prove that the SCC is satisfied at $(x_1,y_1,z_1)$.
This is the case when 
$$E:=[(x_1\bs y_1)\bs (x_1\bs z_1)]\bs [(y_1\bs x_1)\bs (y_1\bs z_1)]$$
is either undefined or equal to $e$.
If $E$ is well-defined, then there is one relation of the form $x_1\cdots=y_1\cdots$ and another of the form $x_1\cdots=z_1\cdots$ contradicting our assumption.
\end{proof}

\subsection{Beyond the complemented case}
We may consider skein presentations $(S,R)$ that are not necessarily complemented (and thus may have more relations than unordered pairs of generators).
We make the assumption that all the relations of $R$ are of the form $a_{1}\cdots=b_{1}\cdots$ where $a,b\in S$ are {\it distinct} colours.
We provide a simplified criterion for checking completeness by removing a number of SCC cases.

\begin{proposition}
Consider a skein presentation $(S,R)$ of a forest-skein category $\cF$ with no relations starting from the same letter.
As usual we write $\cF_\infty$ for the associated forest-skein monoid with associated monoid presentation $P_\infty=(S_\infty,R_\infty)$.

The presentation $P_\infty$ is complete if and only if the SCC is satisfied at $(x_1,y_1,z_1)$ for all triple of colours $(x,y,z)$ where $x\neq z\neq y$.
In that case, $\cF$ is left-cancellative.
\end{proposition}

\begin{proof}
Fix a skein presentation $(S,R)$ as above.
We have that $P_\infty$ is complete if and only if the SCC is satisfied at each triple $(x_i,y_j,z_k)$.
We are going to show that many such SCC are automatically satisfied whatever $(S,R)$ is.
Note, that when $P_\infty$ is complete, then there are no relations in $R_\infty$ starting from the same letter. 
Hence, we deduce that $\cF_\infty$ (and thus $\cF$) is left-cancellative.
Therefore, we only need to prove the first statement.

The strategy of the proof is to use completeness of certain sub-presentations of $P_\infty$ of the form $(S_\infty, Q_\infty)$ with $Q_\infty\subset R_\infty$ satisfying the assumption of Corollary \ref{cor:presentation-free}.
More specifically, $Q_\infty$ will be either equal to $\TR(S)$ or to $\TR(S)\cup \bigcup_n Q_n$ that is deduced from a {\it single} skein relation of the type $a_{1}u=b_{1} v$ (so that $Q_n=(a_nu_n,b_nv_n)$).
We will then use that the SCC is satisfied at any triple of generators for this sub-presentation deducing SCC inside the larger presentation $P_\infty$ in some cases.

Consider a triple of generators $(x_i,y_j,z_k)$.

{\bf Claim 1: The SCC is satisfied at $(x_i,y_j,z_k)$ when $i,j,k$ are all distinct.}

Assume that $i,j,k$ are all distinct.
Now, if $$x_i^{-1} z_k z_k^{-1} y_j\act fg^{-1}$$ with respect to the presentation $P_\infty$, then observe that only Thompson-like relations have been used for performing this reversing.
Indeed, $x_i^{-1}z_k\act z_{k'} x_{i'}^{-1}$ and $z_k^{-1} y_j\act y_{j'} z_{k'}^{-1}$ where the dash indices are equal to the original one plus 0 or 1, i.e.~$j'\in\{j,j+1\}.$
Moreover, no other reversings are allowed since there is only one relation of the form $x_i\cdots=z_k\cdots$ when $i\neq k$.
Hence, $x_i^{-1} z_k z_k^{-1} y_j\act z_{k'} x_{i'}^{-1} y_{j'} z_{k'}^{-1}.$
By inspection of the various cases ($i>j>k$ or $j>k>i$, etc) we observe that $i'\neq j'$. 
Hence, $x_{i'}^{-1} y_{j'}\act y_{j''} x_{i''}^{-1}$. 
We have only used relations of $\TR(S)$ for performing the reversings.
Since the presentation $(S_\infty,\TR(S))$ is complete (by Corollary \ref{cor:presentation-free}) the SCC is satisfied at $(x_i,y_j,z_k)$ and thus $(x_i g)^{-1} (y_j f)$ reverses to $e$ with respect to the presentation $(S_\infty,\TR(S)).$
Hence, $(x_i g)^{-1} (y_j f)$ reverses to $e$ for the larger presentation $P_\infty$ and thus the SCC is satisfied at $(x_i,y_j,z_k)$ for $P_\infty$ proving the claim.

{\bf Claim 2: The SCC is satisfied at $(x_i,y_j,z_k)$ when $i,j,k$ are not all equal.}

Assume now that $|\{i,j,k\}|=2$ and $x_i^{-1} z_k z_k^{-1} y_j\act fg^{-1}$ w.r.t.~$P_\infty$.
We proceed as in the previous claim but possibly allowing one extra relation not in $\TR(S)$.
Indeed, at most one skein relation is used in the reversing process.
For instance, if $i=k<j$, then
$$x_i^{-1} z_i z_i^{-1} y_j\act x_i^{-1} z_i y_{j+1} z_i^{-1}.$$
Now, $x_i^{-1}z_i$ may reverse to $u_iv_i^{-1}$ using a relation $(x_iu_i,z_iv_i)$ where $u_i,v_i$ are words with first letter having index $i$ or $i+1$.
Hence, $x_i^{-1} z_i z_i^{-1} y_j\act u_iv_i^{-1}y_{j+1} z_i^{-1}.$
Now, $j+1$ is strictly larger than the index of the first letter of $v_i$.
Hence, we only have relations from $\TR(S)$ for the remaining reversing operations.
We deduce that $$x_i^{-1} z_i z_i^{-1} y_j\act u_iy_{j+1+m} v_i^{-1} z_i^{-1}=:fg^{-1}$$
where $m$ is the length of the word $v_i$.
Observe that we only used one single relation $(x_iu_i,z_iv_i)$ that was not in $\TR(S)$.
Now, the sub-presentation 
$$(S_\infty, \TR(S)\cup \{ (x_nu_n,z_nv_n):\ n\geq 1\})$$
of $P_\infty$ is complemented and complete by Corollary \ref{cor:presentation-free}.
Therefore, the SCC is satisfied at $(x_i,y_i,z_k)$ for this presentation and thus $(x_ig)^{-1}(y_i f)$ reverses to $e$ for this presentation.
This implies that $(x_ig)^{-1}(y_i f)$ reverses to $e$ for $P_\infty$ and thus the SCC is satisfied at $(x_i,y_i,z_k)$.
The other cases where the three indices are not all equal can be treated similarly.

We have shown that the SCC is satisfied for all triple $(x_i,y_j,z_k)$ where $i,j,k$ are not all equal.
Assume now that $i=j=k$. From our previous observation we know that the SCC is satisfied at the triple $(x_i,y_i,z_i)$ if and only if it is satisfied for the index $i=1$ so we may assume $i=1$.
To conclude the proof it is sufficient to prove that the SCC is satisfied at $(x_1,y_1,x_1)$ and at $(x_1,y_1,y_1)$.
For the first note that:
$$x_1^{-1} x_1 x_1^{-1}y_1 \act x_1^{-1} y_1$$
uniquely and from there we use one single relation that is not in $\TR(S)$.
We can then conclude as in the previous claims.
Similarly, in the second case we have
$$x_1^{-1}y_1 y_1^{-1}y_1\act x_1^{-1} y_1$$
uniquely and can conclude similarly as above.
\end{proof}

\section{Forest-skein groups}\label{sec:forest-groups}

Fix a Ore forest-skein category $\cF$.
We will present the \textit{calculus of fractions} of $\cF$. 
This will allow us to define a groupoid $\Frac(\cF)$ and in particular its isotropy groups $\Frac(\cF,n)$ that are all (equivalence classes of) pairs of forests with the same number of leaves and a fixed number of roots $n$. 
We will be particularly interested in $\Frac(\cF,1)$ defined from pairs of {\it trees}.

\subsection{Fraction groupoids of forest-skein categories}

We recall the construction of the \textit{fraction groupoid} of $\cF$. 
Extensive details can be found in \cite{Gabriel-Zisman67}. All statements below are standard and easy exercises.

{\bf Pairs of forests.}
Let $\cP_\cF$ be the set of pairs $(f,g)\in\cF\times \cF$ where $f$ and $g$ have the same number of leaves.
Consider the smallest equivalence relation $\sim$ of $\cP_\cF$ generated by 
$$(f,g)\sim (f\circ h, g\circ h) \text{ for all } (f,g)\in \cP_\cF , \ h\in\cF$$ and write $\Frac(\cF)$ for the quotient $\cP_\cF/\sim$.
We write $[f,g]$ for the equivalence class of $(f,g)$ inside $\Frac(\cF)$.

{\bf A binary operation: fraction groupoids.}
We define a partially define binary operation on $\Frac(\cF)$.
Consider $(f,g), (f',g')\in \cP_\cF$ such that $|\Root(g)|=|\Root(f')|$. 
By Ore's property there exists $p,p'\in\cF$ satisfying $gp=f'p'$.
We define a binary operation as follows:
$$[f,g]\circ [f',g']:= [fp, g'p'].$$
We recall well-known facts on calculus of fractions that we specialize to our context.

\begin{proposition}
The following assertions are true.
\begin{enumerate}
\item The binary operation $\circ$ of above is indeed well-defined and associative (it does not depend on the choices of $p,p'$ nor on the representatives $(f,g), (f',g')$);
\item The algebraic structure $(\Frac(\cF),\circ)$ is a groupoid;
\item The  formula $f\mapsto [f,I^{\ot n}]$, where $n=|\Leaf(f)|$, defines a covariant functor from $\cF$ to $\Frac(\cF)$ (it preserves compositions).
\item We have the identities $[f,g]\circ [g,h] = [f,h]$ and $[f,g]^{-1}=[g,f]$ for all $[f,g],[g,h]\in\Frac(\cF).$
\end{enumerate}
\end{proposition}

\begin{definition}
We call $\Frac(\cF)$ the \textit{fraction groupoid} of $\cF$ and its elements {\it fractions}.
A {\it forest-skein groupoid} is a groupoid $\cG$ for which there exists a forest-skein category $\cF$ satisfying that $\cG\simeq \Frac(\cF)$.
\end{definition}

\begin{remark}
Note that the fraction groupoid $\Frac(\cF)$ is obtained by formally inverting elements of $\cF$ and is sometimes denoted by $\cF[\cF^{-1}].$ 
Formally, $[f,g]$ corresponds to $f \circ g^{-1}$ and we shall often use this notation which makes item 3 and 4 obvious.
It is common to denote $[f,g]$ as a fraction $\frac{f}{g}$ by analogy with fractions of real numbers or more generally elements of localized rings. This explains the terminology.
Every small category admits an enveloping groupoid made of signed-path of morphisms. What is exceptional here is that being a Ore category assures that every element (or morphism) of this enveloping groupoid (which is the fraction groupoid) is a word of at most length two: $f\circ g^{-1}.$

Note that the map $\cF\to \Frac(\cF),f\mapsto [f,I^{\ot n}]$ where $n=|\Leaf(f)|$ is not injective in general. Indeed, for a Ore forest-skein category it is injective if and only if $\cF$ is also right-cancellative (that is $\cF$ is left and right-cancellatives not only left-cancellative), see Lemma 3.8 and Proposition 3.11 in \cite{Dehornoy-book}.
Even weaker conditions can be considered for having a right-calculus of fractions (for instance by requiring a weaker form of left-cancellativity) but we will not consider them in our study, see \cite{Gabriel-Zisman67}.
\end{remark}

{\bf Diagrammatic description of elements of $\Frac(\cF)$.}
We may extend the diagrammatic description of $\cF$ to its fraction groupoid $\Frac(\cF)$.
Elements of $\cF$ view inside $\Frac(\cF)$ are represented by the same diagrams than before.
Now, an element of the form $[I^{\ot n},g]=g^{-1}$, where $n= |\Leaf(g)|$, is represented by the diagram of $g$ but drawn upside down with now leaves on the bottom and roots on top.
The multiplication is still the vertical stacking.
Hence, we represent an element $[f,g]=f\circ g^{-1}$ as the diagram of $f$ and stack on top of it the upside down diagram of $g$ as in the following example:
\[\includegraphics{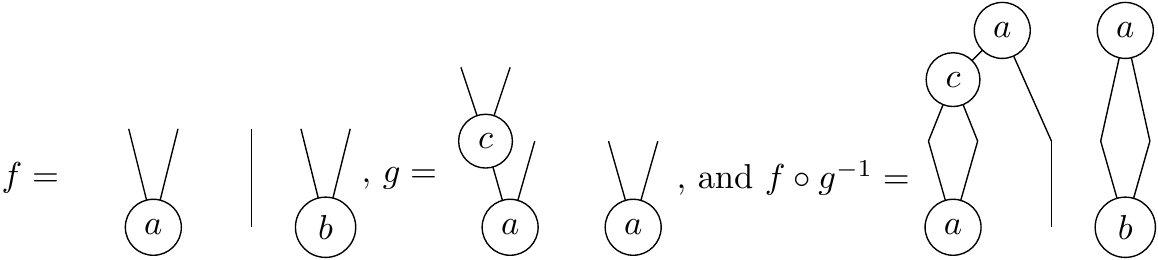}\]
Now, we know how to diagrammatically multiply $f$ with $g^{-1}$ but also $f$ with $f'$ and $g^{-1}$ with $(g')^{-1}$.
It remains to understand the diagrammatic multiplication of $g^{-1}$ with a forest $h$.
Choose some forests $p,q$ satisfying $gp=hq$. 
We obtain that $$g^{-1}\circ h = p\circ (gp)^{-1}\circ hq\circ q^{-1} = p\circ q^{-1}.$$
This can be expressed diagrammatically by stacking $h$ on top of $g^{-1}$ and performing various {\it universal} skein relations of the form 
\[\includegraphics{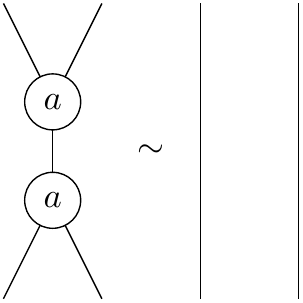}\]
 inherent to all Ore forest-skein categories combined with specific skein relations until obtaining the diagram of $p\circ q^{-1}$. 
We will not use substantially diagrammatic multiplication in $\Frac(\cF)$.
We thus skip a detailed description of it that can be derived without much trouble from the classical monochromatic case of the Thompson groups.
For this later we refer the reader to the Chapter 7 of the PhD thesis of Belk and the article of Jones where many such (monochromatic) diagrammatic computations are performed \cite{Belk-PhD,Jones21}.

\subsection{Fraction groups associated to forest-skein categories.}

We are interested in constructing and studying groups rather than groupoids.
Given $\cF$ as above and its fraction groupoid $\Frac(\cF)$ we can consider for each $r\geq 1$ the set of pairs $(f,g)$ where $f,g$ have $r$ roots and have the same number of leaves.
The set of such classes $[f,g]$ is a group that we write $\Frac(\cF,r)$.
If we consider $\Frac(\cF)$ using categorical terminology, then $\Frac(\cF,r)$ is nothing else than the automorphism group of the object $r$ inside the groupoid $\Frac(\cF)$ also called the \textit{isotropy group} of the object $r$.

\begin{definition}
Let $\cF$ be a Ore forest-skein category.
The \textit{fraction group} of $\cF$ is $\Frac(\cF,1)$.
It is the set of fractions $[f,g]=f\circ g^{-1}$ where $f,g$ are \textit{trees} with the same number of leaves equipped with the restriction of the multiplication of $\Frac(\cF)$.
A {\it forest-skein group} is a group $G$ isomorphic to the fraction group of a forest-skein category. We may also say that $\Frac(\cF,1)$ is {\it the} forest-skein group associated to $\cF$.
\end{definition}

\begin{example}
Let $\cF$ be the monochromatic free forest-skein category.
It is left-cancellative and satisfies Ore's property.
Its fraction group $\Frac(\cF,1)$ is (isomorphic to) Thompson's group $F$ via the Brown diagrammatic description of \cite{Brown87}.
Moreover, $\Frac(\cF,r)$ for $r\geq 1$ is the Higman-Thompson's group usually denoted by $F_{2,r}$ in the literature. Note that $F_{2,r}$ is the ``$F$-version'' of the Higman-Thompson's group $V_{2,r}$ that was first defined by Brown in \cite{Brown87}.
More examples can be found in Section \ref{sec:example} and Section \ref{sec:class-example}.
\end{example}

\begin{remark}
It is worth noting that if $\cF$ is a Ore forest-skein category, then $\Frac(\cF,r)$ is isomorphic to $\Frac(\cF,1)$ for all $r\geq 1$. 
Indeed, since we are considering \textit{binary} forests there exists a tree $t$ with $r$ leaves.
Now, the conjugation by $t$ provides a group isomorphism from $\Frac(\cF,r)$ to $\Frac(\cF,1)$. 
In categorical language, the groupoid $\Frac(\cF)$ is path-connected implying that all its isotropy groups are isomorphic. 
(If we add the empty diagram and the object $0$ to $\cF$ in the purpose of having a tensor unit and fulfilling scrupulously all the axioms of a monoidal category, we obtain one additional isolated path-connected component associated to $0$ giving the trivial isotropy group $\Frac(\cF,0)$.)
If we were considering {\it ternary} forests rather than {\it binary}, then there would be no morphisms between $1$ and $2$ for instance. Although, for the classical case of ternary forests with one colour and no quotients the Higman-Thompson's groups $F_{3,r},r\geq 1$ are all isomorphic but for a different reason, see \cite[Proposition 4.1]{Brown87}.
\end{remark}

\subsection{Functoriality}

The collection of all Ore forest-skein categories together with morphisms between them forms a subcategory of $\Forest$ that we write $\Forest_{Ore}$ and call the \textit{category of Ore forest-skein categories}.
We collect straightforward facts concerning morphisms between forest-skein categories, fraction groupoids, and fraction groups. It shows that all our constructions are functorial.
They are all easy to prove and well-known in greater generality.

\begin{proposition}\label{prop:morphisms}
Let $\cF,\cG$ be Ore forest-skein categories.
The following assertions are true.
\begin{enumerate}
\item If $\theta\in\Hom(\cF,\cG)$, then $\theta$ uniquely extends into groupoid and group morphisms
$$\Frac(\theta):\Frac(\cF)\to\Frac(\cG) \text{ and } \Frac(\theta,1):\Frac(\cF,1)\to\Frac(\cG,1)$$ 
via the formula
$$[f,g]\mapsto [\theta(f),\theta(g)].$$
We often write $\theta$ for $\Frac(\theta)$ or $\Frac(\theta,1)$ if the context is clear.
\item The assignments $\cF\mapsto \Frac(\cF), \theta\mapsto\Frac(\theta)$ and $\cF\mapsto \Frac(\cF,1), \theta\mapsto\Frac(\theta,1)$ define covariant functors that we write:
$$\Frac:\Forest_{Ore}\to \Groupoid \text{ and } \Frac(-,1):\Forest_{Ore}\to \Gr$$
where $\Groupoid$ is the category of groupoids and $\Gr$ the category of groups.
\item If a morphism $\theta:\cF\to\cG$ is injective (resp.~surjective), then $\Frac(\theta)$ and $\Frac(\theta,1)$ are both injective (resp.~surjective).
\end{enumerate}
\end{proposition}

\begin{remark}
The converse of item 3 of the last proposition is false in general: there exist morphisms of Ore categories that are not injective or surjective but extend to injective or surjective morphisms between the corresponding groupoids. 
Such examples appear for Ore categories that are not right-cancellative.
Consider the presented monoid
$$M:=\Mon\la a,b| aa=ba\ra$$
and the morphisms
$$M\to \N, a,b\mapsto 1 ; \ \N\to M, 1\mapsto a.$$
The first is not injective and the second not surjective. However, they extend into isomorphism of groups between $\Frac(M)$ and $\Frac(\N)=\Z.$
We can adapt these examples into the context of forest-skein categories by considering the presented forest-skein category $$\cF:=\FC\la a,b| a_1a_1=b_1a_1\ra$$ and free monochromatic forest-skein category $\cG$. They are both Ore categories admitting fraction groupoids and groups.
The morphisms of above define maps between the colour sets of $\cF$ and $\cG$ which define morphisms of forest-skein categories. 
These two morphisms are not isomorphisms of forest-skein categories but extend into isomorphisms of groupoids and groups.
\end{remark}

From the observations of above we deduce an interesting corollary. 

\begin{corollary}\label{cor:Finside}
Any forest-skein group $G$ contains a copy of Thompson's group $F$.
More precisely, if $G$ is the fraction group of $\cG=\FC\la S|R\ra$ and $\cUF$ is the monochromatic free forest-skein category, then each colour $a\in S$ defines an injective morphism $C_a:\cUF\into\cG,Y\mapsto Y_a$ called the colouring map that extends into an injective group morphism $F\into G$.
\end{corollary}

\begin{proof}
Let $G$ be a forest-skein group. 
There exists a presented Ore forest-skein category $\cG=\FC\langle S|R\rangle$ satisfying that $G$ is isomorphic to $\Frac(\cG,1)$. 
Fix a colour $a\in S$ and define the colouring map
$$C_a:\cUF\to\cG, Y\mapsto Y_a$$
from the monochromatic free forest-skein category $\cUF$ to $\cG$.
Such a morphism exists by universal property of $\cUF$.

Let us show that $C_a$ is injective. We assume the contrary and follow a similar argument given in Example \ref{ex:LC} item 3.
Assume $C_a$ has a nontrivial kernel and consider a pair of trees $(t,s)\in\ker(C_a)$ satisfying $t\neq s$ with minimal number of leaves.
Decompose $t$ and $s$ as $$t=Y \circ (h\ot k),\ s= Y \circ (h'\ot k')$$
where $h,k,h',k'$ are trees.
Since $C_a$ preserves composition and tensor product we have that $(h,h')$ and $(k,k')$ are in $\ker(C_a)$.
By minimality of the number of leaves we conclude that $h=h', k=k'$ implying $t=s$, a contradiction.
Therefore, $C_a$ is injective and extends into an injective {\it group} morphism $\Frac(C_a,1): \Frac(\cUF,1)\to \Frac(\cG,1)$ by the last proposition.
\end{proof}

\begin{remark}
Note that the proof of the last proposition shows that if $\cG$ is a left-cancellative forest-skein category, then it contains a copy of the monochromatic free forest-skein category.
\end{remark}

We deduce that all forest-skein groups share common properties since they all contain the group $F$. We list few of those in the corollary below. Although, this list is far from being exhaustive.

\begin{corollary}\label{cor:geometric-dimension}
A forest-skein group is infinite, contains free Abelian groups of arbitrary rank, has infinite geometric dimension (see Section \ref{sec:def-finiteness}), is not elementary amenable, and has exponential growth.
\end{corollary}

\begin{remark}
A number of properties of $F$ are not shared by all $F$-forest-skein groups.
For instance, there exist forest-skein groups with torsion, with their abelianisation with torsion, some forest-skein groups contain a copy of the free group of rank two, some forest-skein groups decompose as a nontrivial direct product, can have nontrivial centre, etc. 
\end{remark}

\subsection{Fraction groups of forest-skein monoids and the CGP}\label{sec:CGP}

Let $\cF$ be a Ore forest-skein category which implies that its forest-skein monoid $\cF_\infty$ is a Ore monoid.
We can define a calculus of fractions on $\cF_\infty$ just as we did for $\cF$. 
We briefly recall the construction that is rather identical to the one performed for $\cF$.
Although, the construction is slightly simpler since we do not need to care about numbers of leaves and roots.

Consider the product $\cF_\infty\times\cF_\infty$ and define the equivalence relation $\sim$ on it that is generated by $(f,g)\sim (f\circ h, g\circ h)$ with $f,g,h\in\cF_\infty.$
Write $\Frac(\cF_\infty)$ for the quotient of $\cF_\infty\times\cF_\infty$ by $\sim$ and write $[f,g]=\frac{f}{g}=f\circ g^{-1}$ for the equivalence class of $(f,g)$ inside $\Frac(\cF_\infty).$
Given $[f,g],[f',g']\in \Frac(\cF_\infty)$ we define
$$[f,g]\circ [f',g']:=[fp,g'p'] \text{ for some } p,p' \text{ satisfying } gp=f'p'.$$
This is a well-defined associative binary operation on $\Frac(\cF_\infty)$.

\begin{definition}
The algebraic structure $(\Frac(\cF_\infty),\circ)$ is a group called the \textit{fraction group of the forest-skein monoid $\cF_\infty$}. 
\end{definition}

We prove that the two groups $\Frac(\cF,1)$ and $\Frac(\cF_\infty)$ embed in each other. 

\begin{proposition}\label{prop:GHinclusion}
Let $\cF$ be a Ore forest-skein category and consider the associated groups $G:=\Frac(\cF,1)$ and $H:=\Frac(\cF_\infty).$
For each colour $a$ of $\cF$ there exists an injective group morphism $\ga_a:H\to G$.
For each $k\geq 1$ there exists an injective group morphism $\eta_k:G\to H$.

In particular, any property of groups that is closed under taking subgroup is either satisfied for $G$ and $H$ or for none of them. This applies for instance for analytic properties like Haagerup property, Cowling-Haagerup weak amenability, or growth properties, etc.
\end{proposition}

\begin{proof}
Consider $\cF,G,H$ as above.
Fix a colour $a$ of $\cF$ (that is an element of $S$ if $(S,R)$ is a skein presentation of $\cF$ or a tree with two leaves) and a natural number $k\geq 1.$
The map $$\cF\to\cF_\infty, f\mapsto I^{\ot k-1} \ot f \ot I^{\ot \infty}$$ is a functor from the category $\cF$ to the monoid $\cF_\infty$.
It extends into a groupoid morphism from $\Frac(\cF)$ to $\Frac(\cF_\infty)$ and restricts into a group morphism $\eta_k:\Frac(\cF,1)\to \Frac(\cF_\infty)$ that is injective since the original map $\cF\to\cF_\infty$ is.
The range of $\eta_k$ corresponds to pairs of trees with roots at the $k$th spot.
Here is an example of the mapping $\eta_3$:
\[\includegraphics{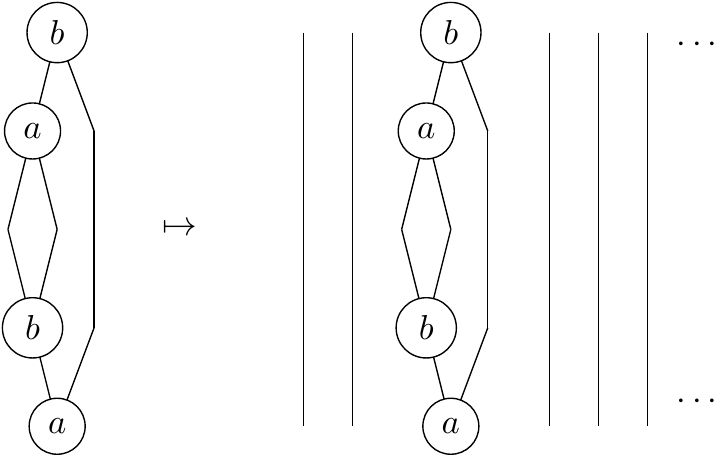}\]

Define the following sequence of trees of $\cF$:
$$t_{n+1}:= a_{1,1}a_{2,2}\cdots a_{n,n} \text{ for all } n\geq 1$$
so that $t_{n+1}$ is a monochromatic tree of colour $a$ with $n+1$ leaves and a long right-branch. 
We say that $t_{n+1}$ is a {\it right $a$-vine} (with $n$ interior vertices).
Here is the diagram of $t_4$:
\[\includegraphics{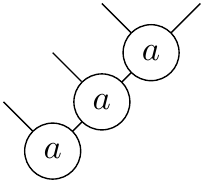}\]
Observe that for any forest $f\in\cF_\infty$ the sequence 
$$n\mapsto [t_n\circ f|_n, t_{n+m}] = t_n \circ f|_n \circ t_{n+m}^{-1}$$ is eventually constant where $f|_n$ is obtained by truncating the infinite forest $f$ to its $n$th first trees and $m$ is the number of carets of $f|_n$ (which is eventually constant too).
This defines a map from $\cF_\infty$ to $G$ which extends into a group embedding $\ga_a:H\to G$.
\end{proof}

\begin{remark}
We observed in Remark \ref{rk:forest-monoid} that $\cF_\infty$ is obtained as a direct limit of set of forests with finitely many roots. 
Similarly, $\Frac(\cF_\infty)$ is the direct limit of the system of groups $(\Frac(\cF,r):\ r\geq 1)$ for the (injective) connecting morphisms
$$\iota_r^{r+n}:\Frac(\cF,r)\to \Frac(\cF,r+n), \ [f,g]\mapsto [f\ot I^{\ot n}, g\ot I^{\ot n}].$$
This explains the notation where $\Frac(\cF_\infty)$ is interpreted as $\varinjlim_{r} \Frac(\cF,r)$: the limit of the fraction groups $\Frac(\cF,r)$ for $r$ tending to infinity.
Now, the morphism $\eta_1:\Frac(\cF,1)\to \Frac(\cF_\infty)$ of the last proposition corresponds to the limit morphism $\iota_1^\infty:=\varinjlim_{n}\iota_1^{1+n}$.
\\
The group $\Frac(\cF,1)$ is often interpreted as a subgroup of $\Frac(\cF_\infty)$ in the literature. Although, we do the opposite and consider $\Frac(\cF_\infty)$ as a subgroup of $\Frac(\cF,1)$ using the morphism $\gamma_a$ of the last proposition. 
With this viewpoint we will provide in Section \ref{sec:presentation} presentations of $\Frac(\cF,1)$ and $\Frac(\cF_\infty)$ where one is the truncation of the other. 
\end{remark}

\begin{definition}\label{def:CGP}
We write $G(a)$ for the range of $\ga_a$ which is a subgroup of $G$ isomorphic to $H$.
If $G=G(a)$, then we say that $G$ has the {\it colouring generating property} at $a$ (in short the CGP at $a$ or simply the CGP).
In that case, $\ga_a$ defines an isomorphism between $H$ and $G$.
\end{definition}

Note that elements of $G(a)$ consists of elements $g=t\circ s^{-1}$ which admits representatives $t,s$ so that all vertices between the root and the right-most leaf of $t$ and $s$ are coloured by $a$.
We will see that in some cases $G(a)=G$. This happens when $\cF$ is the monochromatic free forest-skein category and $G=F$ as pointed out by Brown but there are many other examples, see Section \ref{sec:example}.

\begin{remark}
The isomorphism $$\ga_a:H\to G(a),f\mapsto [t_n\circ f|_n, t_{n+m}]$$  consists in taking a finitely supported forest $f$ and to add on top and bottom a right $a$-vine.
This cannot be generalised in the obvious way to right-vines that are not {\it monochromatic}.
Indeed, let $(a^{(n)},n\geq 1)$ be a sequence of colours and put $\ti t_{n+1}:=a^{(1)}_{1,1}\circ\cdots\circ a^{(n)}_{n,n}$ the right-vine with $j$th vertex coloured by $a^{(j)}$ for $1\leq j\leq n.$
Now, given $f\in\cF_\infty$ we consider 
$$\ti\ga(f,n):=[\ti t_{n}\circ f|_{n}, \ti t_{n+m}]= \ti t_n\circ f|_{n} \circ \ti t_{n+m}^{-1}$$ where $m$ is the number of vertices of $f|_{n}$.
In order to have a monoid morphism we need to have that this fraction is eventually constant in $n$.
But $\ti\ga(f,n+1)$ corresponds to $\ti\ga(f,n)$ to which we add a $a^{(n+1)}$-vertex and a $a^{(n+p+1)}$-vertex. 
Hence, $\ti\ga(f,n)=\ti\ga(f,n+1)$ implies that $a^{(n+1)}=a^{(n+p+1)}.$
\end{remark}

\subsection{Forest-skein groups similar to $T,V,BV$}\label{sec:X-version}

Consider a Ore forest-skein category $\cF$ and let $\cF^X, \cF^Y_\infty$ be the associated $X$-forest-skein category and $Y$-forest-skein monoid, respectively, where $X=F,T,V,BV$ and $Y=F,V,BV$.
It is rather obvious that if $\cF$ is a Ore category, then so are $\cF^X$ and $\cF^Y_\infty$.
Hence, we can define the fraction groupoids  $\Frac(\cF^X)$ and the fraction groups $\Frac(\cF^X,1)$, $\Frac(\cF_\infty^Y).$
We often denote the groups as follows:
$$G^X=\Frac(\cF^X,1) \text{ and } H^Y=\Frac(\cF_\infty^Y)$$
removing the superscript when $X=F$ or $Y=F$.
A group isomorphic to $G^X$ is called a {\it $X$-forest-skein group}.
We also say that $G^X$ and $H^Y$ are the {\it $X$-version} of $G$ and {\it $Y$-version} of $H$, respectively.
An element of $G$ is of the form $t\circ s^{-1}$ with $t,s$ trees.
Now, an element of $\cF^V$ is of the form $f\circ \pi$ where $f$ is a forest and $\pi$ a permutation.
Hence, an element of $G^V$ is of the form $t\circ \pi \circ \sigma^{-1}\circ s$ with $t,s$ trees and $\pi,\sigma$ permutations.
Since $\tau:=\pi\circ\sigma^{-1}$ is itself a permutation we deduce that all elements of $G^V$ are described by {\it triples} rather than {\it quadruples}.
We may write $[t,\tau,s]$ or $[t\circ \tau,s]$ for $t\circ \tau\circ s^{-1}$.
Similarly, elements of $G^T$ and $G^{BV}$ can be described by an equivalence class of a triple $(t,\tau,s)$ with $\tau$ a cyclic permutation or a braid, respectively.
An easy adaptation of the proofs of Corollary \ref{cor:Finside} and Propostion \ref{prop:GHinclusion} provides the following.

\begin{corollary}
Consider $X\in\{F,T,V,BV\}, Y\in\{F,V,BV\},$ a Ore forest-skein category $\cF$ and the associated groups $G^X,H^Y$.
The group $G^X$ contains a copy of the group $X$.
Moreover, $G^Y$ and $H^Y$ embed in each other.
\end{corollary}

\subsection{Examples of forest-skein groups}\label{sec:example}

We provide examples of skein presentations that define Ore forest-skein categories and thus forest-skein groups. 
We anticipate definitions and results from the next sections and the next article. 
In particular, explicit presentations of forest-skein groups given in Theorem \ref{theo:groupG-presentation}. 
If $\cF$ is a Ore forest-skein category, then we use the notations $G^X:=\Frac(\cF^X,1)$ and $H^Y:=\Frac(\cF_\infty^Y)$ for $X=F,T,V,BV, Y=F,V,BV$.
We write $G^{ab}$ for the abelianisation $G/[G,G]=G/G'$ of the group $G$.

\subsubsection{Monochromatic forest-skein categories}
Consider a monochromatic forest-skein category $\cF$. It admits a skein presentation of the form $(S,R)$ where $S=\{x\}$ is a singleton and $R$ is a set of pairs of trees.
Now, $\cF$ is left-cancellative if and only if $R$ is empty. 
In that case $\cF$ is the monochromatic free forest-skein category that is a Ore category.
We obtain that $G\simeq H\simeq F$ and moreover $G^X\simeq X$ for $X=T,V,BV$.
In particular, $G^{BV}$ is Brin's braided Thompson group $BV$ and $H^{BV}$ corresponds to the group $\wh{BV}$ considered by Brin in \cite{Brin-BV1}.

\subsubsection{Ternary Thompson group $F_{3,1}$}
Consider the presented forest-skein category
$$\cF=\FC\la a,b| a_1b_2=b_1a_1\ra.$$
Graphically we may use the following diagrams for the $a$-caret and $b$-caret which provides an intuitive skein relation:
\[\includegraphics{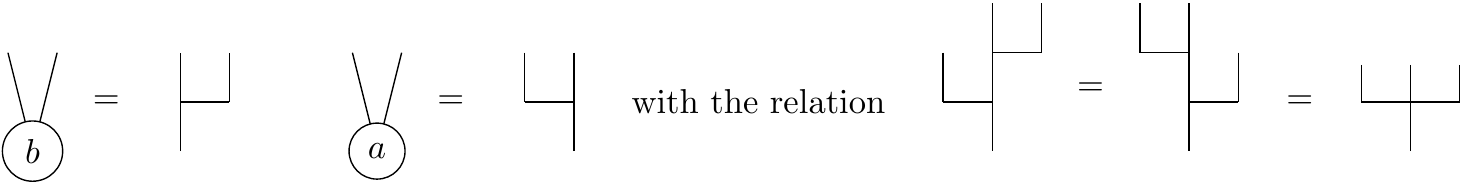}\]
Since there are two colour and a single relation with words of length two we automatically obtain that it is left-cancellative by Corollary \ref{cor:presentation-free} and satisfies Ore's property by Corollary \ref{cor:Ore-FC}. Moreover, one can show it satisfies the CGP.
Hence, $G$ and $H$ exist and are isomorphic.
Moreover, they admit the following infinite group presentation with generators:
$$\{a_j,b_j:\ j\geq 1\}$$
and relations
\begin{itemize}
\item $x_q y_j = y_j x_{q+1}$ for all $x,y=a,b$ and $1\leq j<q$;
\item $a_j b_{j+1} = b_j a_j$ for all $j\geq 1.$
\end{itemize}
We are going to show that $H$ and $G$ are in fact isomorphic to the ternary Higman-Thompson group $F_3=F_{3,1}$ obtained from the monochromatic free {\it ternary} forest-skein category $\cF_3$. 
We do it by renaming the generators $a_j,b_j$ of $H$ and observing that they satisfy the relation of a well-known presentation of $F_{3,\infty}:=\Frac((\cF_3)_\infty)$ itself isomorphic to $F_3:=\Frac(\cF_3,1)$.
Indeed, define 
$$\begin{cases} z_{2n}=b_n\\
z_{2n-1} = a_n
\end{cases} \text{ for all } n\geq 1.$$
Now, substitute $z$ in the presentation of $H$. 
For instance, the Thompson-like relation $a_q b_j=b_j a_{q+1}$ becomes $z_{2q-1} z_{2j} = z_{2j} z_{2q+1}$ and the translated skein relation $b_j a_j=a_j b_{j+1}$ becomes $z_{2j} z_{2j-1}=z_{2j-1} z_{2j+2}$ for $1\leq j<q.$
We deduce the following new presentation of $H$:
$$\Gr\langle z_n, n\geq 1 | z_q z_j = z_j z_{q+2}, \ 1\leq j<q\rangle.$$
Consider now the forest-skein monoid $F_{3,\infty}^+:=(\cF_3)_\infty$ of ternary forests. 
Observe that if $t_j$ is the elementary forest of $F_{3,\infty}^+$ that has a single ternary-caret at the $j$th root, then we have the relations $t_qt_j=t_j t_{q+2}$ for all $1\leq j<q$ and in fact this provides a monoid presentation of $F_{3,\infty}^+$ and thus a group presentation of its fraction group denoted $F_{3,\infty}$ by Brown.
Now, sending $z_j$ to $t_j$ provides a monoid isomorphism from $\cF_\infty$ to $F_{3,\infty}^+$ inducing a group isomorphism from $H=\Frac(\cF_\infty)$ to $F_{3,\infty}.$

Note that $G\simeq H\simeq F_{3}$ admits two embedding of $F$ given by the two colours $a$ and $b$, see Corollary \ref{cor:Finside}.
Moreover, by sending a monochromatic ternary caret to $a_1b_2$ we obtain an embedding of $F_3$ in itself which is not the identity.
Finally, the $T,V,$ and $BV$-versions of $G$ are not (or at least not isomorphic in an obvious way) the usual Higman-Thompson groups $T_{3,1}, V_{3,1}$  nor the ternary Brin braided Thompson group $BV_{3}$.

\subsubsection{Cleary irrational-slope Thompson group}
Consider the presented forest-skein category
$$\cF:=\FC\la a,b| a_1a_1=b_1b_2\ra$$
with two generators $a,b$ and one relation with words of length two.
Using Corollaries \ref{cor:presentation-free} and \ref{cor:Ore-FC} or Theorem \ref{theo:class-example} we have that $\cF$ is a Ore category.
Moreover, it satisfies the CGP (see Section \ref{sec:CGP}).
Hence, $G\simeq H$ and is isomorphic to the Cleary irrational-slope Thompson group further studied by Burillo, Nucinkis, and Reeves \cite{Cleary00,Burillo-Nucinkis-Reeves21}.
It is the group of piecewise affine homeomorphisms of the unit interval with finitely many breakpoints and all slopes powers of the golden number $(\sqrt 5 -1 )/2.$
By Theorem \ref{theo:groupG-presentation} it admits the infinite group presentation
$$\Gr\la a_j,b_j, j\geq 1| x_qy_j=y_j x_{q+1}, \ a_ia_i = b_ib_{i+1}, 1\leq j<q, i\geq 1\ra$$
and the finite group presentation with generators
$$\{a_1,b_1,a_2,b_2\}$$
and relations
\begin{itemize}
\item $[a_1^{-1}x_i, a_1^{-j} y_2 a_1^j]=e$ for $x,y\in\{a,b\}, i,j\in\{1,2\}$ with $(x,i)\neq (a,1)$;
\item  $a_1a_1=b_1b_2$;
\item $a_2a_2=b_2a_1^{-1}b_2a_1$.
\end{itemize}
From there we deduce that $G^{ab}\simeq \Z\oplus\Z\oplus\Z/2\Z.$
Similarly, the larger groups $G^T$ and $G^V$ are the $T$ and $V$ versions of Cleary's group introduced by Burillo-Nucinkis-Reeves in \cite{Burillo-Nucinkis-Reeves22}.
Cleary proved that $G$ is of type $F_\infty$ and his proof can be extended without troubles to prove that $G^T$ and $G^V$ are of type $F_\infty$. 
The forest-skein category $\cF$ is an example of a Ore forest-skein category with a finite spine equal to $\{a_1,b_1, a_1a_1\}$, see Definition \ref{def:spine}.
Hence, by Theorem \ref{theo:Finfty} the group $G$ is of type $F_\infty$ and so are $G^T,G^V$ but also $G^{BV}.$
Hence, we obtain the result that the braided version of $G$ is of type $F_\infty$ as well.

The nice thing of our formalism is that we can easily generalise these results to similar forest-skein groups.
For instance, consider 
$$\cF_n=\FC\la a,b| a_1^n=b_1b_2\cdots b_n\ra \text{ for } n\geq 1$$
obtained by considering a left $a$-vine and right $b$-vine both having $n$ interior vertices.
Here is the skein relation of $\cF_3$:
\[\includegraphics{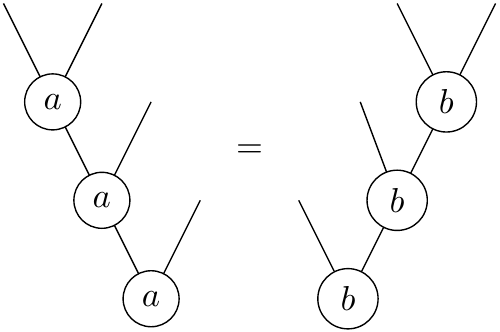}\]
By Theorems \ref{theo:class-example} this is a Ore forest-skein category producing the forest-skein group $G_n:=\Frac(\cF_n,1)$ that is of type $F_\infty$ (the spine is equal to $\{a_1,b_1, a_1^n\}$).
Moreover, it satisfies the CGP and thus $G_n\simeq H_n:=\Frac((\cF_n)_\infty).$
Similar results hold for its $T,V,BV$ versions.
The group $G_n$ admits the finite group presentation  with generators
$$\{a_1,b_1,a_2,b_2\}$$
and relations
\begin{itemize}
\item $[a_1^{-1}x_i, a_1^{-j} y_2 a_1^j]=e$ for $x,y\in\{a,b\}, i,j\in\{1,2\}$ with $(x,i)\neq (a,1)$;
\item  $a_1^n=b_1b_2(a_1^{-1}b_2a_1) (a_1^{-2} b_2 a_1^2)\cdots (a_1^{3-n} b_2 a_1^{n-3})$;
\item $a_2^n=b_2(a_1^{-1}b_2a_1) (a_1^{-2} b_2 a_1^2)\cdots (a_1^{2-n} b_2 a_1^{n-2})$.
\end{itemize}
Moreover, $$G_n^{ab}\simeq \Z\oplus\Z\oplus\Z/n\Z$$ (with generators the images of $a_1,b_1, a_2b_2^{-1}$)
implying in particular that the $G_n$ are pairwise non-isomorphic.

One can generalise in many ways this construction.
Using again Theorem \ref{theo:class-example} we can replace the two vines $a_1^n$ and $b_1\cdots b_n$ by {\it any} pair of {\it monochromatic} trees obtaining the group $G_{(t,s)}$ considered in Section \ref{sec:two-colours}.
Furthermore, we can consider any number of colours and even infinitely many.

\subsubsection{Brin's higher dimensional Thompson groups}
Brin's groups $dV$ with $d\geq 1$ are higher dimensional generalisation of Thompson's group $V$ so that $1V=V$ and $dV$ acts by piecewise affine maps on the hypercube $[0,1]^d$.
An element of $dV$ is characterised by two standard dyadic partitions of $[0,1]^d$: $P=\{R_1,\cdots,R_k\}$ and $P'=\{R_1',\cdots,R_k'\}$ having the same number of pieces and one bijection $\beta$ between the elements of $P$ and $P'$.
The pieces $R_j$ are higher dimensional rectangle obtained by iteratively cutting the unit hypercube into two equal pieces along one axis.
We now describe a description of elements of $dV$ using trees given by Brin and further exploit by Burillo and Cleary \cite{Brin-dV1,Burillo-Cleary10}.
We can encode the data of $P,P'$ as a pair of trees with $d$ possible colours of carets, each colour corresponding to one axis. 
Now, if $a,b$ are two colours, then the two operations $Y_a(Y_b\ot Y_b)$ and $Y_b(Y_a\ot Y_a)$ provide the same result of cutting the hypercube into four using two different axis.
We define the forest-skein category $\cF$ with $d$ colours and relations $a_1(b_1b_3)=b_1(a_1a_3)$ for all pairs of colours $a,b$.
We obtain a Ore forest-skein category. Although, its fraction group is not Brin's group $dV$ but share similar metric properties. 
What makes it different is that the skein relations are not compatible with bijections $\beta$, see \cite[Section 1]{Burillo-Cleary10} for details.

Note that the {\it decolouring map} $$D:\cF\to \cUF, Y_a\mapsto Y$$ consisting in removing all colours is a well-defined and surjective morphism from $\cF$ to the monochromatic free forest-skein category $\cUF$.
It induces a surjective group morphism $G\onto F$ which admits a section. 
Hence, the group $G$ is a nontrivial split extension of $F$ when $d\geq 2$ (and similarly for the $T,V,BV$-versions).
Since all Brin's groups $dV$ are simple (as proved by Brin in \cite{Brin-dV1} for $d=2$ and in the general case in \cite{Brin-dV3}) we deduce that they are not isomorphic to our forest-skein groups.

\subsubsection{An example not satisfying the CGP}
Consider the presented forest-skein category
$$\cF=\FC\la a,b| b_1a_1b_3=a_1a_2a_3\ra.$$
The skein presentation has two colours and one relation implying it is complemented with associated monoid presentation complete.
In particular, $\cF$ is left-cancellative.
Moreover, one can prove it satisfies Ore's property. It is proved by showing that every tree can be grown into a monochromatic tree of colour $a$.
Hence, $\cF$ is a Ore category producing the groups $G$ and $H$.
The group $G$ admits the following infinite presentation: generating set $\{a_j,b_j,\hat b_j:\ j\geq 1\}$ and relations:
$$\begin{cases}
 x_{q} y_{j} = y_{j} x_{q+1} \\
\hat b_{q} y_{j} = y_{j} \hat b_{q+1} \\
b_j b_{j+1} a_j = a_j a_{j+1} a_{j+2}\\
\hat b_j \hat b_{j+1} a_j = e
\end{cases}$$
for all $1\leq j<q, x,y\in\{a,b\}$.
Note that this presentation is not homogeneous.
Using the two first relations we can observe that $G$ is generated by $\{a_1,a_2, b_1,b_2, \hat b_1,\hat b_2\}$.
Moreover, the last relation shows that we may remove $\hat b_1$ for the set of above and still generates $G$.

A presentation of $H$ is deduced by removing all generators with a hat and relations containing them.
We obtain the presentation of $H$ with generating set 
$\{a_j,b_j:\ j\geq 1\}$ and relations:
$$\begin{cases}
x_{q} y_{j} = y_{j} x_{q+1} \text{ for all } 1\leq j<q , x,y\in\{a,b\} \\
b_j b_{j+1} a_j = a_j a_{j+1} a_{j+2}
\end{cases}.$$
This forest-skein category does not satisfy the CGP and thus the embedding $\ga_a:H\to G$ of Proposition \ref{prop:GHinclusion} is not surjective.

\subsubsection{A second example not satisfying the CGP}\label{sec:example-no-CGP}
We consider one example that is a member of the family of forest-skein categories studied in Section \ref{sec:two-colours}.
The presented forest-skein category is:
$$\cF:=\FC\la a,b | a_1a_3a_1 = b_1b_3b_1\ra.$$
By Theorem \ref{theo:class-example} it is a Ore category and thus admits a forest-skein group $G=\Frac(\cF,1)$.
If we follow the notation of Section \ref{sec:two-colours} we may write it $G_{(t,t)}$ where $t$ is the complete (monochromatic) binary tree with 4 leaves all at distance 2 from the root.
The group $G$ admits the following infinite presentation with generator set $\{a_j,b_j,\hat b_j:\ j\geq 1\}$ and relations:
$$\begin{cases}
 x_{q} y_{j} = y_{j} x_{q+1} \\
\hat b_{q} y_{j} = y_{j} \hat b_{q+1} \\
b_j b_{j+2} b_j = a_ja_{j+2} a_j\\
\hat b_j \hat b_{j+2} b_j = a_j
\end{cases}$$
for all $1\leq j<q, x,y\in\{a,b\}$.
Once again this forest-skein category does not satisfy the CGP. 
We observe that $G$ is generated by the set $\{a_1,a_2,b_1,b_2, \hat b_3\}$.

\subsubsection{Monoids give forest-skein groups}
In the future article \cite{Brothier22-HPM} we will explain how certain monoid gives forest-skein groups and study them.
We provide here two examples of those that have interesting finiteness properties.

Consider 
$$\cF=\FC\la a,b| a_1a_1=b_1b_1, a_1b_1=b_1a_1\ra.$$
This is a Ore forest-skein category that is an extension of $F$. 
It has infinite spine but nevertheless is of type $F_\infty$. 
Moreover, one can prove that $G$ is isomorphic to the wreath product $\Z/2\Z\wr_Q F=\oplus_{Q}\Z/2\Z\rtimes F$ associated to the action $F\act Q$ that is the classical action of $F$ on the set of dyadic rationals of the unit torus.
This group and its abelianisation have torsion and moreover have a nontrivial center isomorphic to $\Z/2\Z.$
Similarly, $G^T$ and $G^V$ are isomorphic to $\Z/2\Z\wr_Q T$ and $\Z/2\Z\wr_Q V$ that are not simple nor perfect.

Consider the skein presentation with colours $\{x_i:\ i\geq 1\}$ and relations
$$(x_k)_j (x_i)_j = (x_i)_j(x_{k+1})_j \text{ for all } 1\leq i<k \text{ and } i\geq 1.$$
It is constructed from the Thompson monoid.
Its associated forest-skein category $\cF$ is a Ore category. Since $\cF(2,1)$ is infinite we have that $\cF$ does not admit any skein presentation with finitely many colours. 
Nevertheless, its forest-skein group is finitely presented and in fact of type $F_\infty$ as we will prove it in \cite{Brothier22-HPM}.

\section{Presentations of forest-skein groups}\label{sec:presentation}

In all this section we consider $P:=(S,R)$ a skein presentation of a forest-skein category $\cF$.
Let $\cF_\infty$ be the associated forest-skein monoid.
Assume that $\cF$ is left-cancellative and satisfies Ore's property. 
This implies that both $\cF$ and $\cF_\infty$ admit a right-calculus of fractions giving a fraction groupoid $\Frac(\cF)$ and a fraction group $\Frac(\cF_\infty)$.
Consider the fraction groups $G:=\Frac(\cF,1)$ and $H:=\Frac(\cF_\infty)$.

In this section we introduce a contractible simplicial complex $E\cF$ admitting a free simplicial action $G\act E\cF$.
In particular, the space of orbits $B\cF:=G\bs E\cF$ is a classifying space of $G$.
From $B\cF$ we deduce an explicit infinite group presentation of $G$ in terms $P$ and described a reduce one
$\Gr\la \sigma|\rho\ra$ of $G$.
By identifying $H$ as a subgroup of $G$, we obtain a presentation $\Gr\la \sigma_0|\rho_0\ra$ of $H$ by taking specific subsets $\sigma_0\subset \sigma$ and $\rho_0\subset \rho.$
We prove that when $\cF$ admits a skein relation with finitely many colours (resp.~finitely many colours and skein relations), then the groups $G$ and $H$ are finitely generated (resp.~presented) and we give upper bounds for the number of their generators and relations.

\subsection{A complex associated to each forest-skein category}\label{sec:complex}

Recall that $\cF$ is a category with object $\N$ and a forest $f$ is a morphism with origin $\omega(f)=|\Leaf(f)|$ and target $\tau(f)=|\Root(f)|$.
We extend these origin and target maps to the groupoid $\Frac(\cF)$.
Hence, $[f,g]=f\circ g^{-1}$ has for origin the number of roots of $g$ and target the number of roots of $f$.

{\bf A $G$-space.}
Since $G\subset \Frac(\cF)$ we can consider the restriction of the multiplication of $\Frac(\cF)$ to composable pairs of $G\times \Frac(\cF)$.
This provides an action 
$$G\times Q\to Q, (g,x)\mapsto g\circ x$$
where $$Q=\{[t,f]:\ t \text{ tree } f \text{ forest }, \ |\Leaf(t)|=|\Leaf(f)|\} = \{ x\in \Frac(\cF):\ \tau(x)=1\}.$$

{\bf A $G$-directed set.}
Now, an element of $Q$ can be multiplied to the right by an element of $\Frac(\cF)$ but if we want to preserve $Q$ we must multiply  by elements of $\cF$.
Hence, we consider
$$\Frac(\cF,1) \act Q \curvearrowleft \cF.$$
The right action $Q\curvearrowleft \cF$ defines a partial order:
$$x\leq x\circ f,\ x\in Q, f\in\cF.$$
We obtain a partially ordered set (in short poset) $(Q,\leq)$. 
Since the left and right actions of any subsets of $\Frac(\cF)$ mutually commute we have that $\leq$ is invariant under the action of $G$. 
Moreover, the poset $(Q,\leq)$ is directed (i.e.~two elements of $Q$ admits an upper bound) since $\cF$ satisfies Ore's property.

{\bf A contractible $G$-free simplicial complex.}
Let $\Delta(Q)$ be the ordered complex deduced from $(Q,\leq)$.
It is the abstract simplicial complex with vertices $Q$ and $k$-simplices (for $k\geq 1$) the strict chain $x_0<\cdots<x_k$ whose faces are the sub-chains.
Write $E\cF:=|\Delta(Q)|$ for the geometric realisation of $\Delta(Q)$ which is a simplicial complex with $k$-simplices $|x_0<\cdots<x_k|$ that we may identify with $\Delta(Q)$ if the context is clear.

Since $G\act Q$ is free and order-preserving it induces a free simplicial action $G\act E\cF$.
Moreover, since $Q$ is directed we deduce that $E\cF$ is contractible.
Therefore, the quotient $B\cF:= G\backslash E\cF$ is a classifying space for $G$, i.e.~$B\cF$ is a path-connected CW-complex, $\pi_1(B\cF,p)\simeq G$, and $\pi_n(B\cF,p)=\{e\}$ for all $n\geq 2$ where $p$ is a point of $B\cF$.
The projection map $E\cF\to B\cF$ is a fibre bundle and since $E\cF$ is simply connected (since it is contractible) $E\cF$ is a universal cover of $B\cF$.
A $k$-cell of $B\cF$ is an orbit of the form $G\cdot |x_0<\cdots<x_k|.$

\begin{remark}
Note that the constructions of $E\cF$ and $B\cF$ are canonical and thus do not depend on any choice of skein presentation of $\cF$. 
Moreover, these constructions are functorial.
\end{remark}

\subsection{An infinite group presentation}

Since $G\act E\cF$ is a free action on a contractible space we have that the Poincar\'e group $\pi_1(B\cF,p)$ at any point of $B\cF:=G\bs E\cF$ is isomorphic to $G$.
Moreover, $G\simeq \pi_1(B\cF,p)=\pi_1(B\cF^{(2)},p)=\pi_1(G\bs E\cF^{(2)},p)$ where $B\cF^{(2)}$ stands for the 2-skeleton of $B\cF$ (the subcomplex of $B\cF$ obtained by removing all $k$-cells for $k\geq 3$).
From there we are able to extract informations on group presentations of $G$ in terms of the simplicial structure of $E\cF^{(2)}.$
We start by providing a group presentation of $G$ in terms of the skein presentation $P=(S,R)$ of $\cF$.

{\bf Orientation.} 
Equip the simplicial complex $E\cF=|\Delta(Q)|$ with the orientation deduced from the partial order $\leq$ of $Q$.
Hence, a 1-simplex $|x<y|$ of $E\cF$ is now interpreted as a directed edge starting at $x$ and ending at $y$.

{\bf A subtree of $E\cF$.}
We now construct a subtree of $E\cF$ whose image in $B\cF$ is a maximal subtree.
Here, a tree in $E\cF$ means a path-connect subcomplex contained in the 1-skeleton of $E\cF$ that does not contain any cycle.
The reader should not confuse such a tree in $E\cF$ and a tree of the forest-skein category $\cF$.

{\it Vertices.}
Fix a colour $a\in S$ and define by induction a sequence of trees $t_n\in\cF$ with $n$ leaves and monochromatic in $a$ such that 
$$t_1=I \text{ and } t_{n+1}= t_n\circ a_{n,n} \text{ for } n\geq 1$$
so that $t_{n+1}=a_{1,1}a_{2,2}\cdots a_{n,n}.$
The tree $t_n$ corresponds graphically to a long right branch that we call a {\it right vine of colour $a$} or a {\it right $a$-vine} following the terminology of Belk \cite[Section 1.3]{Belk-PhD}.
They previously appear in the proof of Proposition \ref{prop:GHinclusion}.
Note that $t_n$ is an element of $Q$ and thus a vertex of the complex $E\cF:=|\Delta(Q)|.$

{\it Edges.}
For each $n\geq 1$ note that $t_n<t_{n+1}$ and thus we have a directed edge $|t_n<t_{n+1}|$ of $E\cF$ going from $t_n$ to $t_{n+1}$. 
Let $T_a$ be the directed graph inside $E\cF$ with vertex set 
$$V(T_a):=\{t_n:\ n\geq 1\}$$ and edge set $$E(T_a)=\{ |t_n<t_{n+1}|:\ n\geq 1\}.$$
Graphically $T_a$ is a single infinite ray.

{\bf A maximal subtree of $B\cF$.}
Consider the quotient map $$q:E\cF\to B\cF,\ \sigma\mapsto G\cdot\sigma.$$
Recall that we have the origin-map 
$$\omega:\Frac(\cF)\to\N, \ [f,g]\mapsto |\Root(g)|$$ 
and observe that $\omega$ is $G$-invariant.
Moreover, if $x,y\in Q$ and $\omega(x)=\omega(y)$, then there exists $g\in G$ satisfying $y=g\cdot x.$
Indeed, one can simply take $g:= y\circ x^{-1}$ which is in $\Frac(\cF,1)=G$.
This two facts implies that $\omega$ factorises into a bijection:
$$\omega:G\bs P\to \N_{>0}.$$
Since $\omega(t_n)=n$ for all $n\geq 1$ we deduce that $\{t_n:\ n\geq 1\}$ is a set of representatives of the vertices of $B\cF$, i.e.~the quotient map $E\cF\to B\cF$ restricts into a bijection from $V(T_a)$ to $B\cF^{(0)}$.
This implies that $q(T_a)$ is a maximal subtree of $B\cF$ containing all vertices of $B\cF$.

\begin{remark}
Note that $(t_n:\ n\geq 1)$ is an increasing sequence of trees such that $t_n$ has $n$ leaves.
We could make the same construction with any such sequence which will provide a different presentation of the group $G$.
However, the specific choice of $t_{n+1}=a_{1,1}\cdots a_{n,n}$ seems to be the most practical to work with. 
\end{remark}

{\bf Group presentation of $G$.}
We have found a classifying space $B\cF$ of $G$ and thus its Poincar\'e group is isomorphic to $G$. 
Moreover, $q(T_a)$ is a maximal subtree of $B\cF$. 
By classical theory (see for instance \cite[Chapter 4]{Stillwell-book} or \cite[Theorem 3.1.16]{Geoghegan-book}) a presentation of $G$ is obtained by taking the free group over the oriented edges of $B\cF$ that we mod out by the boundary of each 2-cell of $B\cF$ and by the edges of $q(T_a)$.
Let us describe precisely this presentation in our context.

{\bf Naive description.}
Consider a nontrivial forest $f\in\cF$ and write $\ell,r$ for its number of leaves and roots, respectively.
We have an edge $|t_r<t_rf|$ in $E\cF$ whose image by the quotient map $q:E\cF\to B\cF$ connects $q(t_r)$ and $q(t_\ell)$.
Therefore, any nontrivial forest $f$ defines an edge in $B\cF$ and conversely any edge of $B\cF$ is of this form.
We obtain a surjective map
$$\cF^*\onto \text{Edge}(B\cF),\ f\mapsto q(|t_r<t_rf|)$$
where $\cF^*$ is the set of nontrivial forests.

Let $\Gr\langle\cF\rangle$ be the free group over the set of forests $\cF$ writing $\bar f$ the generator of $\Gr\la \cF\ra$ associated to $f\in\cF$.
Let $\pi$ be the quotient of $\Gr\la\cF\ra$ by the relations: 
\begin{enumerate}
\item $\bar f =e$ if $f$ is a trivial forest (i.e.~$f$ has the same number of leaves and roots);\\
\item $\bar f =e$ if the edge $q(|t_r<t_rf|)$ associated to $f$ is in $q(T_a)$;\\
\item $\bar f \cdot \bar f'=\bar f''$ if $f,f',f''$ correspond to the edges equal to the boundary of a 2-cell of $B\cF$ so that $\bar f,\bar f''$ have same start and $\bar f',\bar f''$ have same end.
\end{enumerate}
From the discussion of above we obtain that $\pi\simeq \pi_1(B\cF,p)\simeq G$.

{\bf More refined description.}
Let us better interpret the relations of above using the structure of $\cF$. 
The second item is equivalent to $G\cdot |t_r\leq t_r f| = G\cdot | t_r\leq t_\ell |$. 
Which means $f=t_r^{-1}\circ t_\ell$ that is a product of some $a_{m,m}$.
Item two is thus equivalent to
$$\bar a_{n,n}=e \text{ for all } n\geq 1.$$
For the third item: consider a 2-cell of $B\cF$. It is of the form $G\cdot |t_r<t_r f<t_rff'|$ for some $r\geq 1$ and forests $f,f'.$
Now, $G\cdot |t_r<t_rf|$ corresponds to $\bar f$ and $G\cdot |t_rf< t_rff'|$ to $\bar{f'}.$
The third item is equivalent to
$$\bar f \cdot \bar f' = \ov{f\circ f'} \text{ for $(f,f')$ a pair of composable nontrivial forests.}$$
Finally, if we allow $f$ or $f'$ to be the trivial forests in the expression of above we deduce the first item.
We deduce a group presentation of $G$ and relate it with the morphisms $\ga_a:H\to G(a)\subset G$ of Proposition \ref{prop:GHinclusion}.

\begin{proposition}\label{prop:presentation-G}
Let $\cF$ be a Ore forest-skein category, $a$ in the colour set of $\cF$, $G:=\Frac(\cF,1)$ the fraction group of $\cF$, $B\cF$ the classifying space of $\cF$ constructed as above, and $p$ a point of $B\cF$.
We have the following isomorphisms:
\begin{equation}\label{eq:presentation-pi}G\simeq \pi:=\pi_1(B\cF,p)\simeq \Gr\langle \cF | \ \bar{a_{n,n}}=e,\ \bar f \cdot \bar f' = \ov{f\circ f'}\rangle\end{equation}
where $f\mapsto \bar f$ denote the canonical embedding from $\cF$ into the free group over the set $\cF$, and where the relations of above stand for each $n\geq 1$, and $(f,f')$ composable pairs of forests.
The group morphism
$$\theta_0:\Gr\la\cF\ra\to G,\ \ov f\mapsto t_r \circ f \circ t_\ell^{-1} \text{ where } f\in\cF(r,\ell)$$
factorises into a group isomorphism $\theta$ from the presented group of above and $G$. 
Moreover, the set $\theta(\cF\ot I)$ of all $\theta(\bar f)$ with $f$ a forest with last tree trivial generates the group $G(a)$.
\end{proposition}

\subsection{Practical infinite and reduced group presentations}\label{sec:practical-presentation}
We now fix a skein presentation $(S,R)$ of $\cF$ and express presentations of $G$ and $H$ using the elementary forests.
Note that we freely identify $G$ and $\pi$ via the isomorphism $\theta$ of the last proposition.
Using Propositions \ref{prop:universal-category} and \ref{prop:presentation-G} we deduce that $G$ admits the presentation with generating set 
$$\{ \ov{b_{j,n}}:\ b\in S, 1\leq j\leq n\}$$
and set of relations:
\begin{enumerate}
\item $\ov{a_{n,n}}=e$ for all $n\geq 1$;
\item $\ov{b_{q,n}}\circ \ov{a_{j,n+1}} =  \ov{a_{j,n}}\circ \ov{b_{q+1,n+1}}$ for all  $a,b\in S$ and $1\leq j <q\leq n$;
\item $\ov{U_{j,n}}=\ov{U'_{j,n}}$ for all skein relations $(u,u')\in R$ and $1\leq j\leq n$
\end{enumerate}
where $U_{j,n}$ is a word in the elementary forest expressing the forest $u_{j,n}:= I^{\ot j-1}\ot u\ot I^{\ot n-j}$ and where $\ov{U_{j,n}}$ denotes its image in the free group $\Gr\la \cF\ra$.
The third kind of relation is written $R(u,u',j,n)$.

First, observe that $\ov{b_{j,n}} = \ov{b_{j,j+1}}$ for all $n>j$. 
Hence, the generator set of $G$ can be written as
$$\{\ov{b_{j,j+1}}, \ \ov{b_{j,j}} :\ j\geq 1\}.$$
Second, the generators of the first kind generates the subgroup $G(a)$ which is isomorphic to the fraction group $\Frac(\cF_\infty).$
An isomorphism is given by $$\ov{b_{j,j+1}}=[t_n \circ b_{j,n}, t_{n+1}]\mapsto b_j \text{ for all } b\in S, 1\leq j <n.$$
Third, the conjugation by $\ov{a_{1,2}}$ corresponds in shifting both indices, i.e.~
$$\ov{a_{1,2}}^{-1} \circ \ov{b_{k,r}} \circ \ov{a_{1,2}} = \ov{b_{k+1,r+1}}.$$
This implies that we only need to retain the relations $R(u,u',1,1)$, $R(u,u',1,2)$, $R(u,u',2,2)$, and  
$R(u,u',2,3)$.
Indeed, $R(u,u',j,n)=R(u,u',j,j+1)$ for all $j<n$ by the first observation and $\ad(\ov{a_{1,2}}^{-1})(R(u,u',i,m))=R(u,u',i+1,m+1)$ by the third for all $i\leq m$.
Similarly, we may only retain the indices $j=1,2$ for the generating set. 
Moreover, using the relation $\ov{a_{j,j}}=e$ for all $j$ we may remove the generators $\ov{a_{1,1}},\ov{a_{2,2}}.$
We deduce the following.

\begin{observation} 
The set
$$\sigma:=\{ \ov{b_{1,2}}, \ov{b_{2,3}}, \ov{c_{1,1}}, \ov{c_{2,2}}:\ b\in S, c\in S\setminus\{a\}\}$$ generates the group $G.$
The subset $\sigma_0:=\{\ov{b_{1,2}}, \ov{b_{2,3}}: \ b\in S\}$ generates the subgroup $G(a)$ isomorphic to $H$.
In particular, if $\cF$ has finitely many colours, then $G$ and $H$ are finitely generated.
\end{observation}

If $U_{j,n}$ has letters $\ov{b_{k,r}}$ with $k$ larger than $2$, then we replace it by $\ov{a_{1,2}}^{2-k} \circ \ov{b_{2,r-k+2}}\circ \ov{a_{1,2}}^{k-2}$ so that the relations $R(u,u',j,n)$ are all expressed using the letter of the smaller generating set $\sigma$.
An easy adaptation of the proof of \cite[Section 3]{Cannon-Floyd-Parry96} permits to reduce the Thompson-like relations to a few commutators equal to the identity.
To have more compact notation we give a presentation with the following correspondence of symbols:
$$b_j\leftrightarrow \ov{b_{j,j+1}},\ \wh b_j\leftrightarrow \ov{b_{j,j}}
, R(u,u',j)\leftrightarrow (\ov{U_{j,j+1} },\ov{U'_{j,j+1} }), \text{ and } \wh R(u,u',j)\leftrightarrow (\ov{U_{j,j} },\ov{U'_{j,j} }).$$

\begin{theorem}\label{theo:groupG-presentation}
Let $\cF=\FC\langle S|R\rangle$ be a presented Ore forest-skein category with associated forest-skein group $G:=\Frac(\cF,1)$.
Fix a colour $a\in S$.
The group $G$ admits the group presentation with generating set
$$\{ \wh b_j, b_j:\ b\in S, j\in\{1,2\}\}$$
and set of relations
\begin{enumerate}
\item $[ a_1^{-1} x_i \ ,\ a_1^{-j}\wh y_2a_1^j]=e$ for all $i,j\in\{1,2\}$ and $x,y\in S$ with $(x,i)\neq (a,1)$;
\item $[ a_1^{-1} x_i \ ,\ a_1^{-j}y_2a_1^j]=e$ for all $i,j\in\{1,2\}$ and $x,y\in S$ with $(x,i)\neq (a,1)$;
\item $\wh a_1= \wh a_2 =e;$
\item $\widehat{R}(u,u',i)=R(u,u',i)=e \text{ for all } (u,u')\in R, i=1,2.$
\end{enumerate}
In particular, if $S$ is finite, then $G$ is finitely generated admitting a generating set of cardinal $4|S|-2$ and if both $S,R$ are finite, then $G$ is finitely presented admitting a presentation with $4|R|+8|S|^2-4|S|+2$ relations.

Moreover, the subgroup of $G$ generated by $\{b_j:\ b\in S, j\in\{1,2\}\}$ is isomorphic to $H:=\Frac(\cF_\infty).$
A presentation of $H$ is obtained by considering the relations (2) and (4) excluding $\wh R(u,u',i)$.

An infinite presentation of $G$ is given by the generating set
$$\{ \wh b_j, b_j:\ b\in S, j\geq 1\}$$
and set of relations
\begin{enumerate}
\item $x_q\circ y_j = y_j \circ x_{q+1} \text{ for all } 1\leq j <q \leq m, x,y\in S;$ 
\item $\wh x_q\circ y_j = y_j \circ \wh x_{q+1} \text{ for all } 1\leq j <q \leq m, x,y\in S;$
\item $\wh a_n=e$ for all $n\geq 1$;
\item $\wh{R}(u,u',i)=R(u,u',i)=e \text{ for all } (u,u')\in R, i=1,2.$
\end{enumerate}
An infinite presentation of $H$ is obtained by considering the generating set $\{b_j:\ b\in S, j\geq 1 \}$ and the relations (2), (4) excluding $\wh R(u,u',i)$.
\end{theorem}

\begin{remark}
\begin{enumerate}
\item The last theorem proved that if a Ore forest-skein category is finitely generated (resp.~presented) as a forest-skein category, then its fraction group is finitely generated (resp.~presented) as a group. 
However, the converse is not true. A counterexample is given in Section \ref{sec:example}.
\item Let $\cF$ be a Ore forest-skein category. 
We have constructed for the forest-skein monoid $\cF_\infty$ a homogeneous monoid presentation providing a positive homogeneous group presentation of $\Frac(\cF_\infty)$. 
In particular, if $\cF$ satisfies the CGP at a certain colour, then $\Frac(\cF,1)$ admits a positive homogeneous group presentation.
Although, the presentation of $\Frac(\cF,1)$ provided in the last theorem is not homogeneous in general, see Section \ref{sec:example}.
\end{enumerate}
\end{remark}

\section{Betti numbers}\label{sec:Betti}

In this short section we show that every countable forest-skein groups have trivial first $\ell^2$-Betti number.
This investigation comes after a discussion with Sri (Srivatsav Kunnawalkam Elayavalli) in Bonn at the Hausdorff Institute of Mathematics in September 2022. 
Sri strongly suspected that forest-skein groups have trivial first $\ell^2$-Betti number and moreover suggested a proof based on the notion of "good list of generators" as we name it below. We warmly thank Sri!

{\bf Convention.}
As in the rest of this article all groups appearing in this section are discrete but moreover are assumed to be {\bf countable}.

The notion of $\ell^2$-Betti numbers for groups is due to Atiyah, later generalised by Cheeger and Gromov, and extensively studied by L\"uck and Gaboriau \cite{Atiyah76,Cheeger-Gromov86,Luck02,Gaboriau02}.
Using L\"uck's work and in particular his generalisation of Murray-von Neumann's dimension we can define the $\ell^2$-Betti numbers of a countable discrete group $G$ as:
$$\beta_n^{(2)}(G) := \dim_{LG} H_n(G; LG) \text{ for } n\geq 0$$
where $LG$ stands for the group von Neumann algebra of $G$, $\dim_{LG}$ for the L\"uck dimension and $H_n(G;LG)$ for the $n$-th homology group over the $G$-module $LG$.
Note that $\beta_n^{(2)}(G)$ is a positive {\it real} number (so possibly not a natural number). 
Indeed, non-integers may appear as well as irrational and even transcendental, see \cite{Pichot-Schick-Zuk15} and the discussion on the Atiyah problem, see also \cite{Lehner-Wagner13,Grabowski14}.

Our strategy is based on the more recent work of Peterson and Thom where they rephrase this numbers to dimension of {\it cohomology} groups: 
$$\beta_n^{(2)}(G)=\dim_{LG} H^n(G,M)$$
where $M$ is the $G$-module either equal to the densely defined closed operators $\cU G$ affiliated to $LG$, the group von Neumann algebra $LG$, or the Hilbert space $\ell^2(G)$, see \cite[Section 2]{Peterson-Thom11}.
In particular, the first number $\betti(G)$ is the L\"uck dimension of the group of 1-cocycles valued in $M=\cU G, LG$, or $\ell^2(G)$ mod out by the inner ones.
It is known that this dimension is zero if and only if $H^1(G,\cU G)$ is trivial
, see \cite{Thom08} and \cite[Section 2]{Peterson-Thom11}.

Note that L\"uck proved that all the $\ell^2$-Betti numbers of Thompson's group $F$ are trivial and via a different proof Bader, Furman, and Sauer recovered the result of L\"uck that they extended to $T$, see \cite[Theorem 7.10]{Luck02} and \cite[Theorem 1.8]{Bader-Furman-Sauer14}.

\subsection{Setting and notations}
In this part we consider a Ore forest-skein category $\cF$ with a {\bf countable} set $S$ of colours. 
We fix one colour $a\in S$ and as usual write $H:=\Frac(\cF_\infty)$ for the fraction group associated to the monoid $\cF_\infty$, $G:=\Frac(\cF,1)$ for the forest-skein group, and $G^X$ for the $X$-version of $G$ where $X=T,V,BV$, see Sections \ref{sec:CGP} and \ref{sec:X-version}.
We identify $H$ as a subgroup of $G$ via the morphism $\gamma_a$ of Proposition \ref{prop:GHinclusion} and consider the presentations of $H$ and $G$ associated to the fixed colour $a$ of Theorem \ref{theo:groupG-presentation}.
We have that 
$$Y:=\{b_1,b_2:\ b\in S\}$$
is a generating set for $H$ and 
$$Z:=Y\cup \{ \hat c_1,\hat c_2:\ c\in S\setminus\{a\}\}$$
is a generating set for $G$.

Elements of $G$ can be alternatively be denoted as fractions of the form $[t,s]:=t\circ s^{-1}$ with $t,s$ trees. 
For $G^X$, elements are of the form $[t,\pi,s]:= t\circ \pi \circ s^{-1}$ with $t,s$ trees, and $\pi$ being a cyclic permutation, a permutation, or a braid depending if $X=T,V,BV$, respectively.
A key fact is that $G^X$ is generated by $G$ and finitely many elements of the form $[t,\pi,t]$. Moreover, $t$ can be chosen to be a fixed monochromatic tree with three leaves such as $t=a_{1,1}\circ a_{2,2}.$
This can be observed from the presentations given for $T$ and $V$ in \cite{Cannon-Floyd-Parry96} and the presentation of $BV$ given in \cite{Brin-BV2}.

\subsubsection{A key result}
We recall the definition of q-normal and wq-normal subgroups introduced in \cite[Section 5.1]{Peterson-Thom11}.
Note that the second first appeared in \cite{Popa06}.

\begin{definition}
A subgroup $\ti H\subset \ti G$ is 
\begin{itemize}
\item q-normal if there exists a generator set $\ti X$ of $\ti G$ satisfying that $g\ti Hg^{-1} \cap \ti H$ is infinite for all $g\in \ti X$;
\item wq-normal if there exists an increasing chain of subgroups $(\ti H_i:\ i\in I)$ of $\ti G$ satisfying that $\cap_i \ti H_i=\ti H, \cup_i \ti H_i=\ti G$, and 
$\cup_{i\in I: i<i_0} \ti H_i \subset \ti H_{i_0}$ is q-normal for each $i_0\in I.$
\end{itemize}
\end{definition}

Here is the key result of Peterson and Thom that we will be extensively using.
\begin{theorem}\label{theo:PT11}
If $H\subset G$ is an infinite wq-normal subgroup of (countable discrete) group, then $\beta^{(2)}_1(G)\leq \beta^{(2)}_1(H).$
\end{theorem}

To show that $\betti(\ti G)=0$ for a group $\ti G$, it is then sufficient to show that $\ti H\subset \ti G$ is wq-normal and $\betti(\ti H)=0$.

\subsection{A good list of generators}
We introduce the notion of {\it good list of generators} for a group, show that $H$ admits such a list, and deduce that $\betti(H)=0$.
We write $\ti G,\ti H$ for some arbitrary groups to not interfere with the notation of our forest-skein groups $G,H$ of above.

For a group $\ti G$ we define a {\it good list of generator} $L$ to be either a nonempty finite list $(g_1,\cdots,g_k)$ with $k\geq 2$ or an infinite list $(g_j:\ j\geq 1)$ of elements of $\ti G$ satisfying
\begin{itemize}
\item the elements of $L$ generate $\ti G$;
\item two consecutive elements of $L$ commutes;
\item each element of the list (except possibly the last one) is required to have infinite order.
\end{itemize}

\begin{proposition}
If $\ti G$ admits a good list of generators, then its first $\ell^2$-Betti number is equal to zero.
\end{proposition}

\begin{proof}
Consider a group $\ti G$ with a good list of generators $L=(g_i:\ i\in J)$ where $J$ is either $\{1,\cdots,k\}$ for some $k\geq 1$ or is the set of nonzero natural numbers.
For each $j\in J$ satisfying that $j+1\in J$ we define $H_j:=\la g_1,\cdots,g_j\ra \subset \ti G$ the subgroup generated by the first $j$ elements of $L$.
Note that $H_j$ is always infinite since it contains at least one element of infinite order. 
We claim that $H_j\subset H_{j+1}$ is a $q$-normal subgroup.
Indeed, consider the generating set $\ti X:=\{g_1,\cdots,g_{j+1}\}$ of $H_{j+1}$.
Since the first $j$ elements of $\ti X$ are in $H_j$ we only need to consider the element $g_{j+1}$. 
Now, $g_{j+1}$ commutes with $g_j$ implying that $g_{j+1} H_j g_{j+1}^{-1}\cap H_j$ contains the subgroup generated by $g_j$.
Since $g_j$ has infinite order we deduce that this intersection is infinite.
This proves the claim.

We deduce that the inclusion $H_1\subset \ti G$ is wq-normal since $\ti G=\cup_{j\in J} H_j$.

Using Theorem \ref{theo:PT11} we have 
$$\betti(\ti G) \leq \betti(H_1) =\betti(\Z)=0$$
and thus $\betti(\ti G)=0.$
\end{proof}

We now provide a good list of generator for $H$.
Given any colour $x\in S$ we consider the following finite list 
$$L_x:=(x_3x_4^{-1}, x_1x_2^{-1}, x_4x_5^{-1}, x_2x_3^{-1}, x_5).$$
List the colours of $S$ as $(a,b,c,\cdots)$ and form the (possibly infinite) list $L$ that is $L_a$, followed by $L_b$, followed by $L_c$, etc.

\begin{proposition}
The family $L$ is a good list of generators of $H$. In particular, $\betti(H)=0$.
\end{proposition}

\begin{proof}
Using the Thompson-like relation 
$$x_q y_j= y_j x_{q+1} \text{ for all } 1\leq j<q \text{ and all colours } x,y \in S$$
we deduce that two consecutive elements of $L$ commute.
Now, an element of $L_x$ is contained in the subgroup of $H$ generated by $\{x_1,x_2\}$ which is isomorphic to $F$. Since $F$ is torsion-free we deduce that any element of $L_x$ is of infinite order.
Finally, $x_1$ and $x_2$ are contained in the subgroup of $G$ generated by $L_x$.
Since $X:=\{x_1,x_2:\ x\in S\}$ generates $H$ we deduce that $L$ is a good list of generators for $H$.
Using the previous proposition we conclude that $\betti(H)=0$.
\end{proof}

\subsection{All countable forest-skein group have trivial first $\ell^2$-Betti number}
We now show that $G,G^T,G^V$, and $G^{BV}$ have trivial first $\ell^2$-Betti number.
Our strategy is to show that $H\subset G^X$ is a q-normal subgroup and to use Theorem \ref{theo:PT11} whatever $X=F,T,V,$ or $BV$.
In order to do so we complete the generating set $Y$ of $H$ into a larger set so that all additional element $g$ satisfies that $gHg^{-1}\cap H$ is infinite.
We start with the $F$-case.

{\bf Claim: The set 
\begin{equation}\label{eq:generating-set}
Z:=Y\cup \{\hat c_3:\ c\in S\setminus\{a\}\}
\end{equation}
generates the group $G$.}

Write $K$ for the group generated by $Z$.
First note that $a_1\hat c_3 a_1^{-1} = \hat c_2$ for all $c\in S$ implying that $K$ contains all $\hat c_2.$
Second, $a_1^{-j} \hat c_3 a_1^j = \hat c_{3+j}$ for all $j\geq 0$ and $c\in S$ implying that $K$ contains 
$$\wh Z:= Y\cup \{\hat c_j:\ j\geq 2, c\in S\setminus\{a\}\}.$$
To conclude, it is sufficient to show that $\hat b_1\in K$ for all $b\in S$.
Fix a colour $b\in S$.
Since $\cF$ is a Ore category there exists two forest $f,h$ satisfying $b_{1,1}\circ f = a_{1,1}\circ h$.
The forest $f$ can be written with letters of the form $c_{i,j}$ with $c\in S,i\leq j$ and $j\geq 2.$
This produces a relation in $G$ of the form:
$$\hat b_1 w= u$$
where $w,u$ are words in $Y\cup\{ \hat c_j:\ j\geq 2, c\in S, c\neq a\}$.
This proves that $\hat b_1\in K$ and thus $K=G$.
The claim is proven.

Consider the generating set $\wh Z$ of $G$ given in \eqref{eq:generating-set}.
To show that $H\subset G$ is q-normal it is sufficient to show that $\hat b_3 H (\hat b_3)^{-1} \cap H$ is infinite for all $b\in S$ with $b\neq a$.
Observe that $a_2a_1^{-1}$ and $\hat b_3$ commute implying that the group $\hat b_3 H (\hat b_3)^{-1} \cap H$ contains the element $a_2a_1^{-1}$. 
Since this element has infinite order we deduce that $H\subset G$ is q-normal and thus $\betti(G)\leq \betti(H)=0$ giving $\betti(G)=0$.

We now prove the remaining cases.
Consider either $G^T,G^V$ or $G^{BV}$ and a cyclic permutation, a permutation, or a braid $\alpha$ over $n$ strands.
Let $t$ be a tree with $n$ leaves and the element 
$$\alpha_t:=[t,\alpha,t] = t\circ \alpha\circ t^{-1}$$ 
living in $G^T,G^V$ or $G^{BV}$ depending on the nature of $\alpha$.
Consider now some trees $p_i,q_i$ for $1\leq i\leq n$ requiring that $p_i$ has the same number of leaves than $q_i$ for each $i$. 
We associate the following element 
$$g_{p,q}:= [ t\circ (p_1\ot \cdots \ot p_n), t\circ (q_1\ot \cdots\ot q_n)]$$ which is in $G$.
We write $K_t$ for the subset of $G$ equal to all such $g_{p,q}$.

\begin{proposition}
The set $K_t$ is a subgroup of $G$ isomorphic to $G^n$.
An isomorphism is given by:
$$G^n\to G, ([p_i,q_i]:\ 1\leq i\leq n)\mapsto g_{p,q}.$$
Moreover, we have that $\alpha_t K_t \alpha_t^{-1} = K_t$.
\end{proposition}

The proof of this proposition is an easy computation.
In particular, $K_t\subset G$ is an infinite subgroup which permits to deduce the following.

\begin{corollary}
The inclusions $G\subset G^T$, $G\subset G^V$, and $G\subset G^{BV}$ are all q-normal subgroups.
The first $\ell^2$-Betti number of each of the groups $G^T,G^V,G^{BV}$ is equal to zero.
\end{corollary}

\begin{proof}
We use the classical notation $(123)$ and $(12)$ to express permutations acting on $\{1,2,3\}$ where the first is the 3-cycle $i\mapsto i+1 \mod 3$ and the second exchanges $1$ with $2$.
Observe that 
\begin{itemize}
\item $G^T$ is generated by $G$ and $[a_{1,1}a_{2,2}, (123), a_{1,1}a_{2,2}]$;\\
\item $G^V$ is generated by $G$, $[a_{1,1}a_{2,2}, (123), a_{1,1}a_{2,2}]$, and $[a_{1,1}a_{2,2}, (23), a_{1,1}a_{2,2}]$;\\
\item $G^{BV}$ is generated by $G$, $[a_{1,1}a_{2,2}, \beta, a_{1,1}a_{2,2}]$, and $[a_{1,1}a_{2,2},\beta',a_{1,1}a_{2,2}]$ where $\beta,\beta'$ are braids over three strands: $\beta$ is the composition of the undercrossing of 1 and 2 with the undercrossing of 2 and 3 and $\beta'$ is the undercrossing of 2 and 3.
\end{itemize}
In particular, in each case the group $G^X$ is generated by $G$ and finitely many elements of the form $\alpha_t$ (where here $t=a_{1,1}a_{2,2}$) and $\alpha$ ranges over: either one permutation, two permutations, or two braids.
The previous proposition applied to any such $\alpha$ implies that $\alpha_t G \alpha_t^{-1}\cap G$ contains $K_t\simeq G^3$.
Hence, $\alpha_t G \alpha_t^{-1}\cap G$ is infinite implying that $G\subset G^X$ is q-normal for $X=T,V,BV$.
We have previously proven that $\betti(G)=0$.
Using Theorem \ref{theo:PT11} we deduce that $0\leq \betti(G^X)\leq \betti(G)=0$ for $X=T,V,BV$.
\end{proof}

Observe that a forest-skein group $\Frac(\cF,1)$ is countable if and only if $\cF$ has countably many colours (i.e.~$\cF(2,1)$ is countable).
We can now conclude this section with the following general result.

\begin{theorem}\label{theo:Betti}
If $G$ is a countable forest-skein group, then $\betti(G)=0$. 
Moreover, $\betti(G^X)=0$ for $X=T,V,BV$.
\end{theorem}

\section{A replacement of the dyadic rationals for forest-skein groups}\label{sec:Q-space}

Consider a forest-skein category $\cF$ and let $G^X=\Frac(\cF^X,1)$ be its $X$-forest-skein group for $X=F,T,V$.
The classical Thompson groups $F,T,V$ act on the set of dyadic rationals $\Q_2:=\Z[1/2]/\Z$ of the torus in an obvious manner by restricting the classical action on the unit torus. 
This section is about constructing an analogous space $Q$ for the forest-skein groups $G,G^T,G^V$. 
We start by building the action $G^T\act Q$ that we describe in three different ways. 
The last of it uses Jones' technology and permit to be extended to $G^V$ in a canonical way.
We derive key transitive properties of it. 
From there we deduce a finiteness theorem by adapting an elegant argument of Brown and Geoghegan.

\subsection{Three descriptions of a $G^T$-space}

{\it First description.} We consider the homogeneous space action $G^T\act G^T/G$ given by $$g\cdot hG:=ghG,\ g\in G^T, hG\in G^T/G.$$

{\it Second description.}
Consider the two commuting actions $$G^T\act \Frac(\cF^T)\curvearrowleft \cF$$ on the fraction groupoid $\Frac(\cF^T)$ obtained by restriction of the composition.
Note, this is similar to the way we constructed a classifying space of $G$ in Section \ref{sec:complex}.
This gives $G^T\act Q^T$ where $$Q^T=\{ t\circ \pi\circ  f^{-1}:\ t, \pi ,f \},$$
where $t$ is a tree with $n$ leaves, $\pi\in\Z/n\Z$ a cyclic permutation, and $f$ is a forest with $n$ leaves.
In other words, $Q^T$ is the subset of $\Frac(\cF^T)$ of elements of target $1$.
Consider now the equivalence relation generated by the right action of $\cF$ on $Q^T$ that is
$$x\sim y \text{ if } x\cF\cap y\cF\neq\emptyset$$
and denote by $Q^T/\cF$ the space $Q^T$ quotiented by this equivalence relation.
The action $G^T\act Q^T$ factorises into an action $G^T\act Q^T/\cF$.
Note that any element of $Q^T/\cF$ is obtained as a class of the form $t\pi\cF$ with $t$ a tree and $\pi$ a cyclic permutation.
In fact a set of representatives of $Q^T/\cF$ is given by all the $t\pi$ satisfying the property:
$$[t\pi\cF\cap t'\pi'\cF\neq\emptyset] \Rightarrow [t\pi\cF\supset t'\pi'\cF].$$
In that case, $t'\pi'=t\pi h = (th^\pi) \pi^h$ for a certain forest $h$ and where $(\pi,h)\mapsto (\pi^h, h^\pi)$ denotes the Brin-Zappa-Sz\'ep product (or bicrossed product) of the groupoid of cyclic permutations and the forests category $\cF$.
Here is a diagrammatic example when on the left we have a product $\pi h$ and on the right a product $h^\pi \pi^h$. 
\[\includegraphics{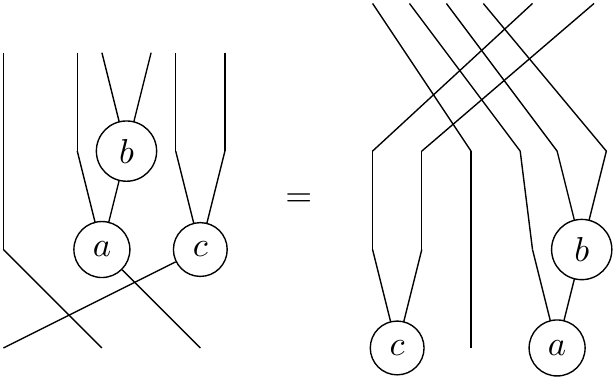}\]
Note that $\pi$ and $\pi^h$ are both cyclic and $h, h^\pi$ have the same number of roots and leaves.
The second action is conjugate to the first via the map
$$Q^T/\cF\to G^T/G, \ t \pi\cF\mapsto t\pi t^{-1}G.$$
To see that it is well-defined note that if $f$ is a forest then $t\pi \cdot f = t f^\pi \pi^f$. 
Moreover, a key point is that $f$ and $f^\pi$ have the same number of leaves (and roots) implying that $t f^\pi (tf)^{-1}\in G$.
We deduce: 
$$t f^\pi \pi^f (tf^\pi)^{-1} G = tf^\pi \pi^f (tf)^{-1} G = t\pi t^{-1} G.$$
The $G^T$-equivariance of the map is also obvious.

{\it Third description.}
We adapt to {\it any} Ore forest-skein category $\cF$ an example of actions constructed in \cite{Brothier-Jones19bis} just before Exercise 2.2.
This is done using Jones' technology.
Let $(\Set,\sqcup)$ be the monoidal category of sets equipped with disjoint union for the monoidal structure.
Let $\Phi:\cF\to \Set$ be the unique monoidal contravariant functor satisfying $\Phi(n)=\{1,\cdots,n\}$ and $\Phi(t)(1)=1$ for all tree $t$.
Hence, if $f=(f_1,\cdots,f_r)$ is a forest with $r$ roots and $n$ leaves, then $\Phi(f)$ is a map from $\{1,\cdots,r\}$ to $\{1,\cdots,n\}$ sending $j$ to 
$$j^f:=|\Leaf(f_1,\cdots,f_{j-1})|+1.$$
The number $j^f$ corresponds to the first leaf of the $j$th tree of $f$.
It is remarkable that $\Phi$ indeed exists whatever the forest-skein category $\cF$ is. 
It is a consequence of the universal property of a forest-skein category among categories given in Section \ref{sec:universal-property}.
Following Jones' technology we define the set $X_0$ equal to all pairs $(t,j)$ where $t$ is a tree and $j$ a leaf of $t$ (identified with an element of $\{1,\cdots,r\}$ where $r=|\Leaf(t)|$). 
The space $X$ is obtained by quotienting $X_0$ by the equivalence relation generated by
\begin{equation}\label{eq:bicross}(t,j)\sim (tf, j^f) \text{ where } j^f = |\Leaf(f_1,\cdots,f_{j-1})|+1.\end{equation}
Here is an example of two elements of $X_0$ in the same class where the distinguish leaf is a black dot:
\[\includegraphics{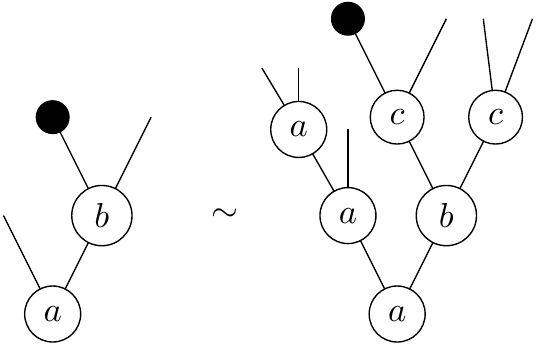}\]
We write $[t,j]$ for the class of $(t,j)$ inside $X:=X_0/\sim$.
Hence, elements of $X$ can be thought as a tree $t$ with a distinguish leaf $\ell$ of it.
Now, we can grow this tree into $tf$ getting for distinguish leaf $\ell^f$: the unique leaf of $tf$ which is a descendant of $\ell$ satisfying that the geodesic from $\ell$ to $\ell^f$ only contains left edges (i.e.~is a left vine).

A key point is that Jones' technology provides an action of the larger group $G^V$.
We have the Jones action 
$$G^V\act X,\ g \cdot [t,j]:= [sp' , \pi(j^p)]$$
where $G^V\ni g=s\circ \pi \circ t'^{-1}$ with $s,t',t$ trees, $\pi$ a permutation, and $p,p'$ some forests satisfying $t'p'=tp$.
In particular, $$s\circ \pi\circ t^{-1} \cdot [t,j]=[s, \pi(j)]$$
when we have equality of $t$ with $t'$.
The restriction of this action to $G^T$ is conjugate to the two actions described above via the map
$$Q^T/\cF\to X,\ t\pi \cF \mapsto [t, \pi^{-1}(1)].$$

{\bf Classical case.}
When $\cF$ is the monochromatic free forest-skein category, then $G=F, G^T=T,$ and $G^V=V$.
Moreover, if $\Q_2=\Z[1/2]/\Z$ is the set of dyadic rationals of the unit torus, then the evaluation map
$$T\to \Q_2, g\mapsto g(0)$$
factorises into a bijection
$$T/F\to\Q_2.$$
The usual action $T\act \Q_2$ and the homogeneous action $T\act T/F$ are conjugated by this bijection. 
The extension to an action of $V$ is obtained by restricting the classical action of $V$ on infinite binary sequences and by identifying $\Q_2$ with finitely supported ones.

\subsection{Total order}
We now mostly work with the third description of the action that we extend to $G^V$.
This is the Jones action $G^V\act X$ where $X$ is the set of classes $[t,j]$ where $t$ is a tree of $\cF$ and $j$ a leaf of $t$.
Moreover, if $a_k$ is the elementary forest with one $a$-caret at the $k$th root, then $[t,j]=[ta_k,j']$ where $j'=j+1$ if $k<j$ and $j'=j$ if $k\geq j.$

We define a total order $\preceq$ on $X$ (a partial order for which any two elements are comparable).
Consider two elements $[t,j]$ and $[t',j']$ of $X$. 
Here, $t,t'$ are trees with $n$ and $n'$ leaves, respectively, and $1\leq j\leq n, 1\leq j'\leq n'$.
Since $\cF$ satisfies Ore's property there exists some forests $f,f'$ so that $tf=t'f'=:z$.
Hence, $[t,j]=[tf,j^f]=[z,j^f]$ and $[t',j']=[t'f',j'^{f'}]=[z,j'^{f'}]$.
We define the order $\preceq$ by setting
$$[t,j]\preceq [t',j'] \text{ if } j^f\leq j'^{f'}.$$
We claim that the formula of above defines a partial order.

\begin{proof}[Proof of the claim]
Consider $[t,j],[t',j']\in X$ and assume there exists a pair of forests $(f,f')$ satisfying 
$$tf=t'f' \text{ and } j^f\leq j'^{f'}.$$
The claim resides in proving that if $(h,h')$ is another pair of forests satisfying $th=t'h'$, then we must have $j^h\leq j'^{h'}.$

First, note that given any forest $h$ with $n$ roots and $m$ leaves we have a map 
$$\{1,\cdots,n\}\to \{1,\cdots,m\}, \ j\mapsto j^h$$ 
(defined in \eqref{eq:bicross}) that is increasing in $j$.
Second, observe that $j^{hk}=(j^h)^k\geq j^h$ and in particular $h \mapsto j^h$ is increasing in $h$ for the usual partial order on the set of forests.

From these observations it is easy to deduce that $\preceq$ is a well-defined partial order.
Indeed, choose $(h,h')$ so that $th=t'h'$.
By Ore's property there exists a pair of forests $(u,v)$ so that $tfu=thv=:z.$
Now, 
\begin{align*}
& [t,j]=[tf,j^f]=[z,j^{fu}] =[z,j^{hv}];\\
& [t',j']=[t'f',j'^{f'}]= [tf, j'^{f'}] = [tfu,j'^{f'u}] = [z, j'^{f'u}]  \text{ and similarly } \\
& [t',j'] = [t'h',j'^{h'}]= [th, j'^{h'}]= [thv, j'^{h'v}] = [z,j'^{h'v}].
\end{align*}

We obtain
$$j^{hv}=j^{fu} = (j^{f})^u \leq (j'^{f'})^u = j'^{f'u} = j'^{h'v}.$$
Hence, $(j^h)^v\leq (j'^{h'})^v$ implying that $j^h\leq j'^{h'}.$
\end{proof}

We have proven that $\preceq$ is indeed a partial order. 
Moreover, Ore's property assures that any two elements of $X$ are comparable.
Therefore, $\preceq$ is a total order on $X$.
We write $\prec$ for the strict order associated to $\preceq$.
The following proposition proves that $G$ and $G^T$ can be defined using $G^V\act X$ and the partial order $\preceq$.

\begin{proposition}\label{prop:order-preserving}
Consider the action $\al:G^V\act X$ and the total order $\preceq$ on $X$.
We have that $G$ is equal to the set of elements of $G^V$ that is preserving the order $\preceq$, i.e.~
$$G=\{ g\in G^V:\ \al_g(x)\preceq \al_g(y) \text{ for all } x,y\in X, x\preceq y\}.$$
The group $G^T$ is the subset of $G^V$ of elements preserving the order $\preceq$ up to cyclic permutation.
That is $G^T$ is the set of $g\in G^V$ satisfying that for all chain $x_1\preceq x_2\preceq \cdots \preceq x_k$ in $X$ there exists a cyclic permutation $\pi\in\Z/k\Z$ so that
$$\al_g(x_{\pi(1)}) \preceq \al_g(x_{\pi(2)})\preceq \cdots \preceq \al_g(x_{\pi(k)}).$$
\end{proposition}

\begin{proof}
Consider $k\geq 1$ and a chain $x_1\preceq \cdots\preceq x_k$ in $X$.
We have that $x_i=[t_i,j_i]$ for some tree $t_i$ and natural number $j_i$.
Up to using Ore's property and taking larger tree representatives we may assume that $t=t_1=\cdots=t_k$.
Hence, $x_i=[t,j_i]$ with $t$ a tree with $n$ leaves and $i\mapsto j_i$ is increasing from $\{1,\cdots,k\}$ to $\{1,\cdots,n\}.$
Every chain of $X$ is thus of this form for suitable $t$ and $j_i$.
Consider $g=[s\circ \pi ,t']\in G^V$.
Up to growing $t$ and $t'$ we may assume $t=t'$.
We obtain that $g\cdot x_i= [s, \pi(j_i)]$ for all $i$.
It is now clear that elements of $G$ preserves the order and elements of $G^T$ preserves the order up to cyclic permutations.

Conversely, if $\pi$ is nontrivial, then $g\notin G$ and there exists $i<j$ such that $\pi(j)<\pi(i)$. This implies that $\al_g([t,j])=[s,\pi(j)]\prec [s,\pi(i)]=\al_g([t,i])$ while $[t,i]\prec [t,j]$.
Hence, $g$ does not preserve the order $\preceq.$
Similarly, we prove that if $\pi$ is not cyclic, then $g$ does not preserve at least one chain up to cyclic permutations.
\end{proof}

{\bf Classical case.}
Assume that $\cF$ is the monochromatic free forest-skein category and thus $G,G^T,G^V$ are equal to $F,T,V$, respectively.
We have a bijection between $X$ and finitely supported sequences $\{0,1\}^{(\N)}$ but also with the dyadic rationals $[0,1)\cap \Z[1/2]$ of the half-open interval $[0,1)$.
A bijection from the second to the third space is given by the map
$$\{0,1\}^{(\N)}\to [0,1], x\mapsto \sum_{n=1}^\infty \frac{x_n}{2^n}.$$
Equip $\{0,1\}^{(\N)}$ with the lexicographical order and $[0,1)\cap \Z[1/2]$ with the usual order of the real line.
We obtain that $(X,\preceq)$ is isomorphic to these two totally ordered sets. Moreover, the bijections considered are isomorphisms of posets.
We recover that $F$ (resp.~$T$) corresponds to order-preserving (resp.~up to cyclic permutations) transformations.

\subsection{Transitivity properties}

\begin{proposition}\label{prop:Q-space}
Let $\cF,G$ as above and write $\al:G^T\act X$ the Jones action described above. 
Define $X_k$ to be all subsets of $X$ of cardinal $k$ and $\al_k:G^T\act X_k$ the action deduced from $\al$ for $k\geq 1.$
The following assertions are true.
\begin{enumerate}
\item For each $A\in X_k$ there exists a tree $t$ with $n$ leaves and some distinct natural numbers $1\leq j_1,\cdots,j_k\leq n$ satisfying $$A=\{ [t,j_i]:\ 1\leq i\leq k\}.$$
\item The action $\al_k$ is transitive for all $k\geq 1.$
\item If $A\in X_k$, then the fixed-point subgroup 
$$\Fix(A):=\{g\in G^T:\ \al_g(x)=x, \text{ for all } x\in A\}$$ 
is isomorphic to $G^k$ and the stabiliser subgroup 
$$\Stab(A)=\{g\in G^T:\ \al_g(A)=A\}$$ 
is isomorphic to $G^k\rtimes \Z/k\Z$ where $\Z/k\Z$ shifts indices modulo $k$.
\end{enumerate}
\end{proposition}
 
\begin{proof}
The first statement was shown in the proof of Proposition \ref{prop:order-preserving}.

Proof of (2).
Consider $k\geq 1$ and $A,B\in X_k$.
From the first item of the proposition we may write $A$ and $B$ as sets of the form $\{[t,j_1],\cdots,[t,j_k]\}$ and $\{[s,i_1],\cdots, [s,i_k]\}$ where both $t$ and $s$ are trees with the same number of leaves, say $n$.
Moreover, we may assume $1\leq j_1<j_2<\cdots<j_k\leq n$, $1\leq i_1<i_2<\cdots<i_k\leq n$ up to re-indexing the elements of $A$ and $B$.
Up to applying $g=t\circ \pi \circ t^{-1}\in G^T$ to $A$ where $\pi(j_1)=1$ is a cyclic permutation we may assume $j_1=1$ and similarly we may assume $i_1=1$.

We claim that there exists trees $t'$ and $s'$ with $n'$ leaves and a sequence $1=l_1<l_2<\cdots<l_k \leq n'$ satisfying that 
$$(t,j_p)\sim (t', l_p) \text{ and } (s,i_p)\sim (s',l_p) \text{ for all } 1\leq p\leq k.$$
We proceed by induction on $k$.
If $k=1$, then we have nothing to do since $j_1=i_1=1.$
Assume $k\geq 2.$
Consider $j_2$ and $i_2$ and choose two trees $f_1$ and $h_1$ satisfying
$$|\Leaf(f_1)| + j_2 = |\Leaf(h_1)|+i_2=:l_2.$$
Now, choose again two trees $f_n$ and $h_n$ satisfying
$$|\Leaf(f_1)|+|\Leaf(f_n)|=|\Leaf(h_1)| + |\Leaf(h_n)|.$$
Finally, define the forests with $n$ roots $f=(f_1,\cdots,f_n)$ and $h=(h_1,\cdots,h_n)$ such that $f_j=h_j$ is trivial if $1\neq j\neq n$.
Observe that 
$$(t,j_2)\sim (tf, l_2) \text{ and } (s,i_2)\sim (sf,l_2).$$
Moreover, $(t,j_1)\sim (tf,j_1)$ and $(s,i_1)\sim (s,l_1)$ since $j_1=i_1=1$.
Finally, $tf$ and $sh$ are two trees with the same number of leaves.
We continue this process in adding trees on top of $t$ and $s$ until all the indices match.

Using the claim we may assume that 
$$A=\{ [t',l_p]:\ 1\leq p\leq k\} \text{ and } B=\{[s',l_p]:\ 1\leq p\leq k\}$$
where $t',s'$ are trees with the same number of leaves.
Taking $g:=s'\circ t'^{-1}\in G\subset G^T$ we deduce that $\al_k(g)(A)=B.$
We have proven that $\al_k$ is transitive.

Proof of (3).
Consider $k\geq 1$ and $A\in X_k$.
By (2) we may assume $A=\{ [t, p]:\ 1\leq p\leq k\}$ where $t$ is a tree with $k$ leaves.
Consider $g\in G^T$ and observe that it can always be decomposed as
$$g= tf \circ \pi \circ (th)^{-1}$$
with $f,h$ forests with $k$ roots and $\pi$ a cyclic permutation.
Assume that $g$ fixes each point of $A$.

We are going to show that the $j$th tree of $f$ and $h$ have the same number of leaves for each $j$.
Note that 
$$g\cdot [t,p] = g\cdot [th,p^h] = [tf, \pi(p^h)],$$
see \eqref{eq:bicross} for the notation $p^h$.
Since $g\cdot [t,p]=[t,p]$ and $[t,p]=[tf,p^f]$ we deduce that
$$\pi(p^h)=p^f \text{ for all } 1\leq p\leq k.$$
Note that $1^h=1=1^f$ implying that $\pi(1)=1$ and thus $\pi$ is the trivial permutation (since it is cyclic and fixes one point).
Hence, $p^h=p^f$ for all $1\leq p\leq k$.
This implies that $|\Leaf(f_j)|=|\Leaf(h_j)|$ for all $1\leq j\leq p$.
Conversely, this condition assures that $tf\circ (th)^{-1}$ fixes each point of $A$.
We deduce
$$\Fix(A)=\{ tf\circ (th)^{-1}:f,h\}$$
where $f=(f_1,\cdots,f_k), h=(h_1,\cdots,h_k)$ and $|\Leaf(f_j)|=|\Leaf(h_j)|$ for all $1\leq j\leq k.$ 
We obtain the first statement of (3) by observing that the map
$$G^k\to \Fix(A), (f_j\circ h_j^{-1},\ 1\leq j\leq k)\mapsto tf\circ (th)^{-1}$$
realises a group isomorphism.
Now $\Stab(A)$ is generated by $\Fix(A)$ and a group $C$ of permutations of the elements of $A$. 
Moreover, $\Fix(A)$ and $C$ forms a semidirect product $\Fix(A)\rtimes C$ where the action is deduced from $C\act A$.
The only possible permutations are necessarily of the form $[t,p]\mapsto [t,p+n]$ for a fixed $n$ and where the index $p$ is considered modulo $k$.
These permutation are realised by $t\circ \pi\circ t^{-1}$ where $\pi$ a cyclic permutation of the leaves.
After identification of $\Stab(A)$ with $G^k$ we deduce the last statement.
\end{proof}

\subsection{A finiteness result}\label{sec:def-finiteness}
We recall few definitions of topological finiteness properties of group introduced by Wall \cite{Wall65}. There exist homological analogous that we don't present here as we did not find any applications with them. We refer the reader to the book of Geoghegan \cite[Sections 7 and 8]{Geoghegan-book} for details.

{\bf Classifying space and universal cover.}
Let $G$ be a group. 
A classifying space $BG$ of $G$ (or $K(G,1)$-complex) is a (pointed) path-connected CW complex such that all its homotopy groups are trivial except the first one $\pi_1(BG,p)$ (the Poincar\'e group) which is isomorphic to $G$.
A universal cover $EG$ of $BG$ is necessarily contractible admitting a free (cellular) action $G\act EG$ whose quotient $G\bs EG$ is isomorphic to $BG$.
Conversely, every classifying space arise in that way.
All classifying spaces of a fixed group $G$ are pairwise homotopically equivalent but they may not be isomorphic as {\it complexes}.
Hence, we may define invariants of the group using properties of the family of its classifying spaces.

{\bf Topological finiteness properties.}
If $G$ is a group and $n\geq 0$, then we say that $G$ satisfies the {\it topological finiteness property} $F_n$ (or simply is of {\it type} $F_n$) if there exists a classifying space $BG$ with finite $n$-skeleton (i.e.~$BG$ has finitely many $k$-cells for all $k\leq n$).
Equivalently, $G$ is of type $F_n$ if there exists a contractible CW complex $EG$ on which $G$ acts freely and such that there are only finitely many $k$-cells modulo $G$ for each $k\leq n$.
The group $G$ is of type $F_\infty$ if it is of type $F_n$ for all $n\geq 1.$
Note that all groups are of type $F_0$ admitting a classifying space with one vertex.
Moreover, being of type $F_1$ (resp.~$F_2$) is equivalent to be finitely generated (resp.~finitely presented).

{\bf Geometric dimension.}
The {\it geometric dimension} $\dim_g(G)$ of $G$ is the minimal dimension (as a complex) of a classifying space of $G$ which is either a natural number or infinity $\infty$. 
It is increasing with respect to inclusion of groups and satisfies that $\dim_g(\Z^r)=r$ and $\dim_g(\Z/k\Z)=\infty$ for any $k\geq 2.$

{\bf Exceptional properties of Thompson-like groups.}
The group $F$ was the first example of a torsion-free group with infinite geometric dimension and of type $F_\infty$ as proved by Brown and Geoghegan \cite{Brown-Geoghegan84}.
Since the result of Brown-Geoghegan there have been many groups similar to $F$ that have been proved to be of type $F_\infty$. 
Here is a list that is very far from being exhaustive which includes the family of groups: $T,V$, Higman-Thompson's groups, Stein's groups, Brin's higher dimensional Thompson's groups $dV$, Brin and Dehornoy's braided Thompson's group $BV$, and large subclasses of Hughes' FSS groups, Guba-Sapir's diagram groups, Rover-Nekrashevych's groups, cloning system of groups, operad groups, and a number of other generalisations \cite{Brown87, Stein92,Kochloukova-Martinez-Perez-Nucinkis13, Fluch-Marschler-Witzel-Zaremsky13, Bux-Fluch-Marschler-Witzel-Zaremsky16,Farley-Hughes15,Witzel-Zaremsky18, Thumann17}.
We refer to a recent preprint of Skipper and Wu for additional examples and references \cite{Skipper-Wu21}.

For all these families of groups the scheme of the proof is based on two technical results: Brown's criterion and Bestvina-Brady's discrete Morse theory \cite{Brown87,Bestvina-Brady97}.
It relies on finding a suitable filtration of a CW complex (often a simplicial one) and computing connectivity of links rather than finding a CW complex with finite skeleton (which is rare to find in practice).
We refer the reader to the article of Zaremsky for a very clear and pedagogical explanation on this strategy \cite{Zaremsky21}.

Even though many families of Thompson-like groups have been proved to be either not finitely presented or of type $F_\infty$ there exist groups satisfying properties in between.
Indeed, for any $n\geq 1$ there exists a Thompson-like groups that is of type $F_n$ but not of type $F_{n+1}$ using a construction and result of Tanushevski \cite{Tanushevski16}.
This is obtained by considering diagrams of monochromatic binary forests and by decorating their leaves with elements of a group with the desirable finiteness property.
Although, all these examples are nontrivial split extension of Thompson group ($F,T$, or $V$) and are thus not simple.
Skipper, Witzel, and Zaremsky have constructed a sequence of Rover-Nekrashevych groups that are of type $F_n$ but not of type $F_{n+1}$ and moreover are simple (giving the first family of this kind), see \cite{Skipper-Witzel-Zaremsky19} and also the family of groups constructed by Belk and Zaremsky \cite{Belk-Zaremsky22}.

Note that all forest-skein groups have infinite geometric dimension since they contain a copy of $F$, see Corollary \ref{cor:Finside}. 
Although, it is a delicate question to decide if they have finiteness properties.
Indeed, finiteness properties (topological or homological) fail to be close under most permanence properties. 

{\bf A result.}
We adapt a clever strategy due to Brown and Geoghegan explained in \cite[Section 4B, Remark 2]{Brown87} and in the book \cite[Theorem 9.4.2]{Geoghegan-book} to deduce that $G^T$ is of type $F_n$ when $G$ is. 
To do this we use a criteria stated by Brown \cite[Proposition 1.1]{Brown87} on homological finiteness property which we simplify and translate to our need.

\begin{proposition}\label{prop:stab-F-infty}
Consider $n\geq 1$, a group $G$, and $C$ a contractible CW-complex.
If there exists a cellular action $G\act C$ such that the stabiliser of each cell is of type $F_n$ and $C$ has finitely many $k$-cells modulo $G$ for each $k\geq 0$, then $G$ is a group of type $F_n$.
\end{proposition}

\begin{theorem}\label{theo:T-F-infty}
Consider $n\geq 1$, a Ore forest-skein category $\cF$, and $G:=\Frac(\cF,1), G^T:=\Frac(\cF^T,1)$ the associated $F$ and $T$-forest-skein groups, respectively.
If $G$ is of type $F_n$, then so is $G^T.$
\end{theorem}

\begin{proof}
Consider $n,G,G^T$ as above and assume $G$ is of type $F_n$.
Let $\ell^2(G^T/G)$ be the real Hilbert space with standard orthonormal basis $B$ indexed by the set $G^T/G$. 
Define $C\subset \ell^2(G^T/G)$ to be the simplicial complex whose $k$-simplices are the convex hull of $k+1$ distinct elements of the basis $B$. 
The complex $C$ is contractible and moreover the action $G^T\act G^T/G$ induces a simplicial action $G^T\act C$.
Using item 2 of Proposition \ref{prop:Q-space}, we deduce that $C$ has one $k$-simplex modulo $G^T$ for each $k\geq 0$. 
Moreover, the stabiliser of a $k$-simplex is isomorphic to $G^{k+1}\rtimes \Z/(k+1)\Z$ by item 3 of the same proposition.
This later group is of type $F_n$ since $G$ is.
We deduce from Proposition \ref{prop:stab-F-infty} that $G^T$ is of type $F_n.$
\end{proof}

\begin{remark}\label{rem:FPm}
A {\it homological} version of this last theorem can be given by replacing $F_n$ by $FP_n$ where $FP_n$ stands for having a projective $\Z G$-resolution of $\Z$ that is finitely generated in dimension smaller than $n$, see \cite[Section 8]{Geoghegan-book} for details.
Note that $F_n$ always implies $FP_n$ for a group $G$ (take the augmented cellular chain complex of the universal cover of a classifying space of $G$ which is a {\it free} $\Z G$-resolution of $\Z$).
The converse holds for $n=1$ and for all $n\geq 2$ providing $G$ is finitely presented.
Hence, new applications in the homological setting would only happen for forest-skein groups $G$ that are not finitely presented but are of type $FP_n$ for a certain $n\geq 2$ (thus necessarily finitely generated). 
We have not witnessed any forest-skein groups with these properties so far.
\end{remark}

\section{Thumann's operad groups and a finiteness theorem}\label{sec:Thumann}
In his PhD thesis Thumann introduced the class of {\it operad groups} and proved that a large collection of them are of type $F_\infty$.
Thumann's theorem recovers at once a number of result of the literature by adapting and extending their techniques of proof in a unified categorical framework \cite{Thumann17}.

\subsection{Operad groups and forest-skein groups}

In this section we present Thumann's operad groups and explain that forest-skein groups are particular case of those. 
We do it for the ordinary $F$-case and briefly explain why the $V$ and $BV$-versions are operad groups as well. 
We use Thumann's notations and terminology and rephrase them in our formalism. 
This allows the reader to easily find additional details in Thumann's article and to present notions and results in the two languages (which are significantly different).

Consider a forest-skein category $\cF$ which we recall is a monoidal category $(\cF,\circ,\ot)$ with set of objects $\N$ and morphisms the forests that we compose by vertical stacking. 
A forest with $\ell$ leaves and $r$ roots is a morphism from $\ell$ to $r$ or an {\it arrow} $\ell\to r$.

\subsubsection{Operad, colours, and operations.}
In Thumann's framework \textit{colour} means something else than colours for vertices of forests in our framework.
Hence, we write \textit{Op-colour} when we refer to the colour in the operad group context. 

An {\it operad} $\cO$ is a set of \textit{Op-colours} together with sets of \textit{operations} $\cO(a_1,\cdots,a_n;b)$ with $a_i,b$ Op-colours so that $a_i$ are the input and $b$ the output Op-colours.
Such an operation is represented (like string diagrams in Jones' planar algebras or in tensor categories) by a triangle with $n$ horizontal lines on the left and one on the right corresponding to the $n$ inputs and the one output.
Here is an example when $n=3$:
\[\includegraphics{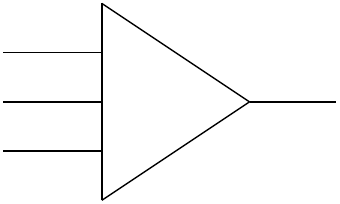}\]
They can be composed associatively by concatenating diagrams horizontally.
Moreover, there is an identity for each Op-colour $1_a\in \cO(a;a).$

{\bf Forest case.}
Define $\cO_\cF$ to be the operad having a single Op-colour $*$ (i.e.~it is {\it Op-monochromatic}) so that the set of operations $\cO_\cF(a_1,\cdots,a_n;b)$ with $a_1=\cdots=a_n=b=*$ is equal to the set of trees with $n$ leaves. 
Our composition of forests correspond to the composition of the operads.
The identity $1_*\in\cO_\cF(*,*)$ is the trivial tree $I$.
Therefore, $\cO_\cF$ corresponds to the set of trees of $\cF$ equipped with the composition of trees with elementary forests.

\subsubsection{Colour-tame.}
Thumann requires that operads are {\it colour-tame} in order to apply his machinery. 
He notes that this may not be necessary but for technical reasons we must assume it.
We will not define it but simply notice that if an operad is Op-monochromatic), then it is automatically colour-tame.
In particular, the operad $\cO_\cF$ associated to any forest-skein category is colour-tame.

\subsubsection{Monoidal category associated to an operad}
Given an operad $\cO$ we consider $\cS(\cO)$ the monoidal category with object finite lists of Op-colours and morphisms finite lists of operations. The tensor product is the juxtaposition.

{\bf Forest case.}
A finite list of Op-colours in $\{*\}$ is a natural number and a finite list of trees is a forest. Juxtaposition corresponds to the tensor product of forests.
We deduce that $\cS(\cO_\cF)$ and $\cF$ are isomorphic as monoidal categories. 

\subsubsection{Adding permutations or braids}
Thumann defines three different types of operads: {\it planar, symmetric, and braided} operads.
The first is the one we just defined, the second consists in adding permutation, and the third braids.
Graphically, we now consider the diagram of a braid drawn horizontally followed to the right by an operation where the strands of the braids are connected to the inputs of the operation. 
For permutations, we do the same except that under and over crossing are replace by a single type of crossing (corresponding to the usual maps $B_n\onto \Sigma_n$).
Formally, this can be achieved using a Brin-Zappa-Sz\'ep product of the category $\cS(\cO)$ with the groupoid of symmetries or braids. 
Hence, any planar operad can produce three different categories: planar, symmetric, and braided.

{\bf Forest case.} 
For forest-skein categories we have proceed in a very similar way. Given a forest-skein category $\cF$ we can add symmetries or braids obtaining $V$ and $BV$-forest-skein categories.
Hence, $F,V,BV$-forest-skein categories correspond to planar, symmetric, and braided operads. 
Note that Thumann did not consider the $T$-case corresponding in incorporating only cyclic permutations.
We are mainly focusing on the $F$-case in this article. Although, we will translate results for the $F,V,$ and $BV$-cases since it is easy to do from Thumann's formalism. 
Moreover, using Theorem \ref{theo:T-F-infty} we will deduce results for the remaining $T$-case.

\subsubsection{Operad groups}
Given a list of Op-colours $Y$ we define the {\it operad group}
$$\pi_1(\cO,Y):=\pi_1(\cS(\cO),Y)),$$
equal to all loops in $\cS(\cO)$ starting (and thus ending) at $Y$ that we equip with the composition of path.
This means formal compositions of morphisms and their inverses in $\cS(\cO)$. 
If $\cS(\cO)$ admits a {\it cancellative calculus of fractions}, then elements of $\pi_1(\cS(\cO),Y)$ are nothing else than $[f,g]=f\circ g^{-1}$ with $f,g$ list of operations ending at $Y$. 
The multiplication and inverse are the obvious ones: $[f,g]\circ [g,h]=[f,h]$ and $[f,g]^{-1}=[g,f].$

{\bf Forest case.} 
Cancellative calculus of fractions is the same notion of being a Ore category in our sense (i.e.~left-cancellative and satisfying Ore's property).
Assume we are in this case.
A list of Op-colours $Y$ is a natural number $n$ and the operad group $\pi_1(\cO_\cF,Y)$ corresponds to the group of fractions $[f,g]$ of forests $f,g$ having both $n$ roots.
We deduce that $\pi_1(\cO_\cF,Y)$ is the forest-skein group group $\Frac(\cF,n)$.

If we consider the symmetric or braided version of $\cO_\cF$ we obtain the $V$ and $BV$-forest-skein groups, respectively.

\subsubsection{Degree and transformation}

{\bf Degree}
An operation of {\it degree} $d$ is an operation with $d$ inputs.
The category $\cI(\cO)$ with object the Op-colours and morphism all the degree 1 operations is considered. 
For technical reasons it is assumed that $\cI(\cO)$ is a groupoid and later we will ask that it is of type $F_\infty^+$ (a groupoid where any subgroup of any of its isotropy group is of type $F_\infty$).

{\bf Transformation.}
A {\it transformation} of $\cS=\cS(\cO)$ is a list of morphisms of $\cI=\cI(\cO)$ forming a subcategory $\cT=\cT(\cO)$ of $\cS$. 
Since $\cI$ is assumed to be a groupoid we deduce that $\cT$ is a groupoid as well.
We consider $\cT\cC=\cT\cC(\cO)$ equal to the set of all operations of $\cO$ mod out by $\cT.$
Elements of $\cT\cC$ are thus classes of operations modulo the transformations.
Compositions defines a partial order for operations that passes through the quotient by transformations.
Moreover, the degree map $\deg$ defined as the number of inputs of an operation passes through the quotient too.
We obtain that $\deg:(\cT\cC,\leq)\to (\N,\leq)$ is strictly increasing, i.e.~$x<y$ implies $\deg(x)<\deg(y).$
We say that $(\cT\cC,\leq,\deg)$ is a {\it graded poset}.

{\bf Forest case.}
Now, a degree $d$ operation of $\cO_\cF$ is a tree with $d$ leaves.
Hence, $\cI=\cI(\cO_\cF)$ is the category with one object $*$ and one morphism: the trivial tree $I$. It is the trivial group and in particular is a groupoid of type $F_\infty^+$ whatever $\cF$ is and whatever we work with the $F,V$, or $BV$-version.
The groupoid $\cT$ is then equal to the collection of trivial forests $\{I^{\ot n}, n\geq 1\}$ that is the groupoid having for set of object $\N$, no morphisms between distinct natural numbers, and one single endomorphism for each $n\in \N$.
A transformation is thus a trivial forest.
Once again this does not change for the other versions $V$ and $BV$ (hence permutations and braids are {\it not} in $\cT$).
An operation of $\cO_\cF$ is a tree. Multiplying by a trivial forest (to the right for our conventions and to the left for the conventions of Thumann) does not change it.
Hence, $\cT\cC$ is in fact equal to the set of operations of $\cO_\cF$ that is the set of trees. 
The partial order is the usual one: $t\leq t\circ f$ that we have considered.
The degree map is the initial map or number of leaves map also considered as the origin map $\omega$:
$$\deg(t):=|\Leaf(t)|=\omega(t).$$
The graded poset $(\cT\cC,\leq,\deg)$ is the set of trees equipped with the same partial order $t\leq t\circ f$ we have defined and the number of leaves map $t\mapsto |\Leaf(t)|.$
For the other $V$ and $BV$-cases one should interpret tree by a tree with a permutation or a braid on top, respectively.

\subsubsection{Spine, elementary transformation classes, finitely generated, and finite type}

{\bf Spine of a graded poset.}
Consider a graded poset $(P,\leq,\deg)$ and its subset $M$ of minimal elements.
We say that a subset $Sp\subset P$ is the {\it spine} of $P$ if it is the smallest subset $Sp\subset P$ satisfying that $M\subset Sp$ and for all $v\in P$ the set $\{s\in Sp:\ s<v\}$ contains a greatest element, i.e.~comparable and larger with all other element in the set of above.

Note that the spine is unique if it exists.
Thumann proved that every graded poset admits a spine, see \cite[Section 3.4.1]{Thumann17}.
To prove it he constructs explicitly the spine providing a convenient equivalent definition.
The construction is as follows: let $Sp_0$ be the minimal elements of $P$.
Given $x,y\in P$ distinct define the set $M_0(x,y)$ of $z\in P$ larger than both $x$ and $y$. 
Now, define $M(x,y)$ to be the subset of all minimal elements of $M_0(x,y)$.
Note that $M(x,y)$ could be empty if the poset is not directed.
Let $Sp_1:= \cup_{(x,y)} M(x,y)$ be the union of all $M(x,y)$ where $x,y$ are in $Sp_0$ and $x\neq y$.
Define inductively $Sp_{n+1}:=\cup_{(x,y)} M(x,y)$ for $n\geq 1$ where the union is over the pairs $(x,y)\in Sp_n\times Sp_n$ so that $x\neq y$.
Finally, the spine of $P$ is equal to the union $\cup_{n\geq 0} Sp_n.$

{\bf Spine of the graded poset associated to an operad.}
Define $\cT\cC^*(\cO)$ to be the full subposet of $\cT\cC(\cO)$ spanned by the degree classes of at least 2.
Minimal elements of $(\cT\cC^*(\cO),\leq)$ form the collection of {\it very elementary transformation classes} denoted $VE$. 
The spine of the poset $(\cT\cC^*(\cO), \leq,\deg)$ forms the collection of {\it elementary transformation classes} denoted $E$.

{\bf Finitely generated operads and operads of finite type.}
The operad is \textit{finitely generated} if $VE$ is finite and of \textit{finite type} if $E$ is finite.

{\bf Forest case.}
{\it Spine.} 
Consider again a forest-skein category $\cF$ and the graded poset $(\cT\cC(\cO_\cF),\leq,\deg)$ of above.
We have that $\cT\cC^*(\cO_\cF)$ is the set of all nontrivial trees: the set of trees minus $I$.
Moreover, observe that if $x,y$ are trees, then $M(x,y)$ corresponds to the minimal common right-multiples of $x$ and $y$.
The subset of minimal elements of $\cT\cC^*(\cO_\cF)$ is the set $\cF(2,1)$ of all trees with two leaves that we call the set of abstract colours of $\cF$. 
Hence, the spine of the graded poset is obtained by taking minimial common right-multiples of trees starting from $\cF(2,1)$.
We rephrase this construction using composition rather than the poset structure and call {\it spine of $\cF$} the set equal to the spine of the graded poset $(\cT\cC(\cO_\cF),\leq,\deg)$.
Write $\cm(x,y)$ for the common (right-)multiples of both $x$ and $y$ and $\mcm(x,y)$ for the minimal elements of $\cm(x,y)$.

\begin{definition}\label{def:spine}
Let $\cF$ be a forest-skein category.
Define the following sequence of sets:
$$\Sp_0(\cF) =\cF(2,1) ,\ \Sp_{n+1}(\cF) := \bigcup_{f,g\in \Sp_n(\cF), f\neq g}\mcm(f,g) \text{ for all } n\geq 0.$$
The \textit{spine} $\Sp(\cF)$ of $\cF$ is the union 
$$\Sp(\cF):=\bigcup_{n\geq 0} \Sp_n(\cF).$$
\end{definition}

{\it Elementary transformation.}
The very elementary transformations are the set of trees with two leaves: the abstract colour set.
The elementary transformations are the elements of $\Sp(\cF)$.

{\it Finitely generated.}
The operad $\cO_\cF$ is finitely generated if and only if $\cF$ has finitely many abstract colours.

{\it Finite type.}
The operad $\cO_\cF$ is of finite type if and only if the spine of $\cF$ is finite.
This last property will be crucial in applying Thumann's theorem.

\subsection{Topological finiteness properties}

We state Thumann's theorem regarding finiteness properties of operad groups that we translate and specialise to forest-skein categories.
We explain the limit of this theorem by exhibiting one striking example.

\subsubsection{Thumann's theorem}
Here is Thumann's main theorem of \cite{Thumann17}.

\begin{theorem}\label{theo:Thumann}
Consider an operad $\cO$ (that is either planar, symmetric, or braided).
If the following five conditions are satisfied:
\begin{enumerate}
\item it has finitely many Op-colours;
\item it is colour-tame;
\item it admits a cancellative calculus of fractions;
\item it is of finite type;
\item the category $\cI(\cO)$ is a groupoid of type $F_\infty^+$,
\end{enumerate}
then for any object $Y$ of $\cS(\cO)$ the operad group $\pi_1(\cO,Y)$ is of type $F_\infty$.
\end{theorem}

Consider now a Ore forest-skein category $\cF$ and its $X$-forest-skein group $\Frac(\cF^X,1)$ with $X$ being either $F,V,$ or $BV$.
Its associated operad $\cO_\cF$ is by definition Op-monochromatic and thus satisfies the first two items.
By assumption it is a Ore category and thus satisfies the third item.
The category $\cI(\cO_\cF)$ is equal to the trivial group and thus satisfies the last item.
Note that the forest-skein group $\Frac(\cF^X,1)$ is isomorphic to $\pi_1(\cO_\cF, Y)$ where $Y$ is the list with one element which is the unique Op-colour of $\cO_\cF$ and $X=F,V,BV$ when $\cO_\cF$ is planar, symmetric, or braided, respectively.
Finally, recall that $\cO_\cF$ is of finite type if and only if the spine of $\cF$ is finite. 
We obtain the following corollary for forest-skein groups.

\begin{corollary}\label{theo:Finfty-withoutT}
Let $\cF$ be a Ore forest-skein category.
If the spine of $\cF$ is finite, then the forest-skein group $\Frac(\cF^X,1)$ is of type $F_\infty$ for $X\in\{F,V,BV\}.$
\end{corollary}

Using Theorem \ref{theo:T-F-infty} we deduce the following result encompassing the $T$-case.

\begin{theorem}\label{theo:Finfty}
Let $\cF$ be a Ore forest-skein category and let $X$ be $F,T,V,$ or $BV$.
If the spine of $\cF$ is finite, then the forest-skein group $\Frac(\cF^X,1)$ is of type $F_\infty$.
\end{theorem}

We now deduce that a skein presentation with two colours and one relation provides a forest-skein group (when it exists) of type $F_\infty$. This produces a very large class of examples for which the theorem above applies. However, it has the weakness to not prove {\it existence} of the forest-skein group. 
In the next section we will present a class of skein presentations for which forest-skein groups always exist.

\begin{corollary}
Consider a presented forest-skein category 
$$\cF=\FC\la a,b| t=s\ra$$ 
with two colours and one relation $t=s$ so that $t$ and $s$ have their roots coloured by $a$ and $b$, respectively.
If $\cF$ satisfies Ore's property, then it is a Ore category whose fraction group $\Frac(\cF,1)$ is of type $F_\infty$ and so are its $T,V$ and $BV$-versions.
\end{corollary}

\begin{proof}
This is a combination of Corollary \ref{cor:presentation-free} and Theorem \ref{theo:Finfty}.
Consider $\cF$ as above and moreover assume it satisfies Ore's property.
Having $t$ and $s$ with roots of different colours implies that the skein presentation $(a,b|t=s)$ is complemented and complete. We deduce that $\cF$ is left-cancellative.
Therefore, $\cF$ is a Ore category and thus the fraction group $\Frac(\cF,1)$ can be constructed.
Now, we show that the spine of $\cF$ has at most three elements.
Indeed, by definition $\Sp_0(\cF)=\{Y_a,Y_b\}$.
Now, $\Sp_1(\cF)=\mcm(Y_a,Y_b)=\{t\}$ is a singleton. 
Since $\Sp_1(\cF)$ is a single point we obtain that $\Sp_2(\cF)$ is empty and thus so are all the $\Sp_n(\cF)$ with $n\geq 3.$
Therefore, $\Sp(\cF)=\{Y_a,Y_b,t\}$ which is finite. 
We can now apply the last theorem which implies that $\Frac(\cF,1)$ is of type $F_\infty$ and so are its $T,V$, and $BV$-versions.
\end{proof}

\begin{remark}\label{rem:Stein-Finfty}
\begin{enumerate}
\item Since we work over binary forests there is always a morphism between $n$ and $m$ inside the fraction groupoid $\Frac(\cF)$.
This implies that all the forest-skein groups $(\Frac(\cF,n),\ n\geq 1)$ are pairwise isomorphic. 
This is why we state our results only for the specific forest-skein group $\Frac(\cF,1)$ with $n=1$.
\item
Observe that if $\cF=\FC\la a,b| t=s\ra$ with $t,s$ trees with {\it three} leaves and roots of different colours, then automatically $\cF$ is a Ore category.
By Corollary \ref{cor:presentation-free} it is left-cancellative and by Corollary \ref{cor:Ore-FC} it satisfies Ore's property.
This only provides very few examples (in fact four of them with two of them isomorphic to $F$) but has the merit to cover the Cleary irrational-slope Thompson group and the {\it bicoloured} description of the ternary (Brown-)Higman-Thompson $F_{3,1}$, see Section \ref{sec:example}.
\item
The discussion of this section can be generalised to forest-skein categories $\cF$ with interior vertices having various valencies (corresponding to the extension of Stein of the groups of Higman and Brown).
The only difference resides in the definition of the spine of $\cF$. 
Indeed, for this more general case consider $\Sp_0(\cF)$ to be the set of all  {\it minimal trees with  at least} two leaves rather than trees with {\it exactly} two leaves.
Then $\Sp_{n+1}(\cF)$ is defined as before as the union of the $\mcm(x,y)$ with $x\neq y$ in $\Sp_n(\cF)$ and $\Sp(\cF)=\cup_{n\geq 1} \Sp_n(\cF).$
If this spine is finite and $\cF$ is a Ore category, then $\Frac(\cF^X,n)$ is a group of type $F_\infty$ for $X=F,T,V,BV$ and any $n\geq 1$.
In this case it does make sense to consider various $n$ in contrary to the binary case as the groups indexed by $n$ may produce several non-isomorphic groups.
\end{enumerate}
\end{remark}

\subsubsection{Strategy and limits of Thumann's theorem}\label{sec:rebel}

{\bf Strategy of Thumann.}
The strategy of Thumann resides in considering the exact same complex than our: $E\cF$ the geometric realisation of the ordered complex of the directed poset $(Q,\leq)$ where $Q$ is the set of $[t,f]$ with $t$ a tree and $f$ a forest.
He considers the origin map $\omega:Q\to\N, [t,f]\mapsto |\Root(f)|$ that is extended as a Morse function $\omega:E\cF\to\R_+$ by making it affine on each simplex. 
This map provides a filtration $E\cF_j:=\omega^{-1}([0,j])$ that is compatible with the (left) action of $G:=\Frac(\cF,1)$.
Now, $G\bs E\cF_j$ is a finite complex and thus by Brown's criterion we have that $G$ is of type $F_\infty$ if the connectivity of the pair $(E\cF_{j+1},E\cF_j)$ tends to infinity in $j$.
Using Bestvina-Brady discrete Morse theory it is then sufficient to prove that the descending links $\linkd(v)$ in $E\cF$ of vertices $v\in E\cF^{(0)}=Q$ has its connectivity tending to infinity when $\omega(v)$ tends to infinity.
This last part is the technical core of the proof.
One can show that $\linkd(v)$ only depends on $\omega(v)=j$ and is isomorphic to some complex $L_j$. 
Now, one can decompose $L_j$ with respect to subcomplexes related to $L_k$ for $k<j$. 
The Hurewicz theorem and the Mayer-Vietoris sequence give some lower bound for the connectivity of $L_j$ with respect to the connectivity of some $L_k$ with $k<j$.
When good hypothesis on $\cF$ are added (in this case: its spine is finite) one can deduce that the connectivity of $L_j$ tends to infinity and thus $G$ is of type $F_\infty$.
Note, this proof is performed separately for the three cases of planar, symmetric, and braided operad groups corresponding to $\cF$ and its $V$ and $BV$-versions, respectively.

{\bf Garside family and Witzel's theorem.}
There is another important strategy that can be used for this kind of groups.
The idea is to find a smaller subcomplex of $X\subset E\cF$ that is stabilised by $G$ and still contractible (so a universal cover of a classifying space of $G$).
We then apply a similar strategy to $X$ using Brown's criterion and Bestvina-Brady discrete Morse theory.
If $X$ is significantly smaller than $E\cF$, then proving that the connectivity of descending links are getting large may be easier to do.

Dehornoy defined Garside family for monoids and more generally for certain small categories \cite{Dehornoy-book}.
Witzel proved (in greater generality) the key result that if we keep the same vertex set $Q$ but allow only simplices of the form $x<xf_1<\cdots<xf_n$ where all $f_j$ are in a fixed factor-closed Garside family (or equivalently $f_n$ is in this specified Garside family), then we obtain a subcomplex $X\subset E\cF$ that is $G$-closed and still {\it contractible} \cite{Witzel19}.
One can then apply the strategy of above to $X$ to deduce finiteness properties of $\cF$.
Of course, the smaller the Garside family is, the smaller $X$ is, and the better is the theorem of Witzel for doing computations.

Using Dehornoy's work we have that every Ore forest-skein category $\cF$ admits a {\it smaller} Garside family $\Gar(\cF)$: it is the smallest set containing the abstract colours of $\cF$ that is closed under taking mcm and right-factors. 
To fulfil Witzel's assumption we close it for left-factors as well and continue to write it $\Gar(\cF).$
If $\Ga(\cF)$ denotes the trees of $\Gar(\cF)$, then note that $\Gar(\cF)$ is the set of forest $f=(f_1,\cdots,f_n)$ such that $f_j\in \Ga(\cF)$ for each $1\leq j\leq n.$
By definition, we observe that $\Sp(\cF)\subset \Ga(\cF)$. 
Hence, if the spine is infinite, then we cannot apply Thumann's result but we can still use Witzel strategy by taking the subcomplex $X$.
The complex $X$ will be still quite large since $\Ga(\cF)$ is infinite but we may still hope to simplify a number of computations.

{\bf One rebel example.}
We present one example of forest-skein category that has a small skein presentation with two colours and two relations of length two, provides a forest-skein group very similar to $F$, but does not fulfil the assumption of Thumann's theorem and for which Witzel's theorem does not help.
Consider the forest-skein category $$\cF=\FC\langle a,b| a_{1,1}b_{1,2} = b_{1,1} a_{1,2}, \ a_{1,1}a_{1,2}=b_{1,1}b_{1,2}\rangle.$$
This example is inspired by a thin monoid provided by Dehornoy in \cite[Example 4.4]{Dehornoy02}.
It is not hard to prove that $\cF$ is left-cancellative and satisfies Ore's property.
A quick computation shows that $\Sp_n(\cF)$ is never empty for every $n$ implying that $\Sp(\cF)$ is infinite.
Hence, we cannot use Thumann's theorem.
Moreover, following an argument of Dehornoy we have that $\Gar(\cF)$ is as large as possible and equal to the whole category $\cF$. Hence, the subcomplex $X$ is equal to $E\cF$ in this case.

However, we will see in a future article that its forest-skein group $G:=\Frac(\cF,1)$ is of type $F_\infty$ using a different complex that is inspired by the work of Tanushevski and Witzel-Zaremsky \cite{Tanushevski16,Witzel-Zaremsky18}.
Moreover, we will show that $G$ decomposes as a wreath-product of the form $\Z_2\wr_X F$ where $F\act X$ is the usual action of $F$ on the dyadic rationals of $[0,1)$. 
The group $G$ can be described as the usual Thompson group $F$ using pairs of trees (with a single type of caret) but where leaves are labelled by $0$ or $1$ (the elements of $\Z/2\Z$). 
From this point of view it is somehow the most elementary forest-skein group strictly larger than $F$ that we could think of.
However, it has infinite spine and its Garside family $\Gar(\cF)$ si equal to $\cF$ preventing the use of two very general theorems on finiteness properties.

\section{A large class of forest-skein groups}\label{sec:class-example}

As stressed before, given a monoid or more generally a category it is a difficult task to check if it is left-cancellative or satisfies Ore's property.
In this section, we will provide a large class of examples of forest-skein categories that are all Ore categories and thus produce forest-skein groups. 
These categories have skein presentations of a particular kind built from a family of {\it monochromatic} trees. 
The monochromaticity makes possible to prove at once that they all satisfy the property of Ore. 
Moreover, using a general criteria of Dehornoy we prove they are left-cancellative. 
Finally, their spine is as small as it could be: equal to the set of trees with two leaves and a single additional tree.
We start by giving the general construction and then specialise to the two colours case.

\subsection{General construction}

Recall that $\cUF\la S\ra$ denotes the free forest-skein category over a set $S$ and $\cUF\la *\ra$ the {\it monochromatic} free forest-skein category where $*$ denotes its unique colour.

\begin{observation}
Let $t\in\cUF\la*\ra$ be a monochromatic tree and let $\cUF\la*\ra(t)\subset \cUF\la*\ra$ be the quasi-forest subcategory generated by $t$ as defined in Section \ref{sec:forest-sub} (i.e.~all forests made of $I$ and $t$).
Define inductively the following sequence of monochromatic trees $(t_n)_{n\geq 1}$ so that $t_1=t$ and $t_{n+1}=t_n \circ f_n$ where $f_n$ is the forest composable with $t_n$ whose each tree is equal to $t$ (hence $t_{n+1}$ is obtained from $t_n$ by attaching to each of its leaf a copy of $t$).

We have that for any tree $s\in\cUF\la*\ra$, there exists a tree $z\in\cUF\la*\ra(t)$ satisfying $s\leq z$.

In particular, the sequence $(t_n)_{n\geq 1}$ is cofinal in the set of trees of $\cUF\la*\ra$ (i.e.~for all tree $s\in \cUF\la*\ra$ there exists $n\geq 1$ so that $s\leq t_n$ and $t_k\leq t_l$ for all  $1\leq k\leq l$).
\end{observation}

\begin{proof}
Consider two monochromatic trees $t,s\in\cUF\la*\ra$ assuming that $t$ is nontrivial.
We prove the observation by induction on the number of leaves of $s$.
If $s$ has only one caret, then it is smaller than $t$ since $t$ is nontrivial.
Note, this obvious fact uses in a crucial way that $\cUF\la*\ra$ is monochromatic. 
If not, write $s$ as $s_0\circ f$ where $f$ is an elementary forest with one caret say at the $j$th root of $f$.
By the induction assumption we have that $s_0\leq z$ for a certain $z\in\cUF\la*\ra(t)$.
We can decompose $z$ as $s_0\circ h$.
Now, if the $j$th tree of $z$ is nontrivial, then $z\geq s$ since $h\geq f$.
If not, attach to the $j$th root of $h$ the tree $t$.
We obtain a larger tree in $\cUF\la*\ra(t)$ which is larger than $s$.
\end{proof}

We now define our class of examples.
Fix a nonempty index set $S$ and let $(\tau_a:\ a\in S)$ be a family of nontrivial monochromatic trees of $\cUF\la *\ra$ such that all of them have the same number of leaves.
Let $$C_a:\cUF\la*\ra\to \cUF\la S\ra, Y\mapsto Y_{a}$$
be the colouring map consisting in colouring all vertices of a monochromatic forest of $\cUF\la*\ra$ with the colour $a$.
Define $\cF_\tau=\cF_{(\tau_a:\ a\in S)}$ to be the forest-skein category with colour set $S$ and skein relations  $(C_a(\tau_a), C_{b}(\tau_{b}))$ for all $a\neq b$ in $S$.
We may write $\cF=\cF_\tau$ if the context is clear.

\begin{remark}
Note that we do not require $a\in S\mapsto \tau_a$ to be injective. 
Hence, even if there are only finitely many monochromatic trees with a fixed number of leaves we have no limitation on the cardinal of $S$ since we may choose to have $a\mapsto \tau_a$ constant.
\end{remark}

Here is a very satisfactory theorem that provides a huge family of examples and whose proof is rather short. It was discovered after proving that various one parameter families of forest-skein categories (typically with two colours and one relations) were Ore categories. 

\begin{theorem}\label{theo:class-example}
Let $S$ be a nonempty set and $\tau:S\to\cUF\la*\ra$ a map from $S$ into the set of monochromatic trees.
Assume that all the trees $\tau_a$ with $a\in S$ have the same number leaves $n\geq 2$ and define $\cF_\tau$ as above.

The forest-skein category $\cF_\tau$ is left-cancellative, satisfies Ore's property, and the spine of $\cF_\tau$ is equal to $\cF_\tau(2,1)\cup\{C_a(\tau_a)\}$ where $a\in S$ is any colour.

Let $G_\tau:=\Frac(\cF_\tau,1)$ be its fraction group.
When $S$ is finite of cardinal $n$, then $G_\tau$ admits a finite presentation with $4n-2$ generators and $8n^2-4$ relations. 
Moreover, $G_\tau$ is of type $F_\infty$ (and so are its $T,V$, and $BV$-versions).
\end{theorem}

\begin{proof}
Fix $\tau:S\to\cUF\la*\ra$ as above and write $\cF$ for $\cF_\tau$.
By abuse of notation we write $C_a$ for the composition 
$$\chi\circ C_a:\cUF\la*\ra\to\cUF\la S\ra\to\cF$$ 
which consists in taking one monochromatic forest $f$, colouring all its vertices by $a$, and considering its image in the quotient forest-skein category $\cF$. 
Note that $C_a(\tau_a)\in\cF$ does not depend on $a\in S$ and is equal to a certain tree that we denote by $t$.

{\it Ore's property.}
We start by proving that $\cF$ satisfies Ore's property.
Consider a tree $s\in\cF$. Let us show that $s$ is dominated by a tree in $\cF(t)$ (this latter being the quasi-forest subcategory generated by the tree $t$).
We proceed by induction on the number of carets of $s$.
If $s$ is trivial, then this is obvious. 
If not, decompose $s$ as $s_0\circ f_0$ where $f_0$ is an elementary forest having a single caret: say a caret of colour $a$ at the $k$th root.
Consider the forest $f'_0$ with its $k$th tree equal to $t$ and all other trees trivial.
We have that $f_0\leq f'_0$ and thus $s\leq s_0\circ f'_0$.
Consider now the tree $s_0$. If $s_0$ is trivial, then we are done.
If not, $s_0 = s_1\circ f_1$ where $f_1$ is an elementary forest with a caret of a certain colour $b\in S$ at the $r$th root of $f_1.$
Now, the composition $f_1\circ f_0'$ can be interpreted as a forest with only $b$-carets. 
Indeed, $f_0'\in\cF(t)$ and thus is a disjoint union of copies of $t$. 
Since $t=C_b(\tau_b)$ it can be interpreted as a tree with only $b$-carets and thus so does $f_1\circ f_0'$.
Now, the observation of above implies that $f_1\circ f_0'$ is dominated by an element $f_1'$ of $\cF(t)$ since $f_1\circ f_0'$ is monochromatic.
We deduce that $s=s_1\circ f_1\circ f_0\leq s_1\circ f_1'$ with $f_1'\in\cF(t)$ and $s_1$ is a proper rooted subtree of $s_0$.
By continuing inductively this process until $s_k$ is trivial we deduce that there exists $s'\in\cF(t)$ satisfying $s\leq s'$.

It is now easy to deduce that $\cF$ satisfies Ore's property.
Indeed, define the sequence $(t_n)_{n\geq 1}$ as in the observation. It is a cofinal sequence of trees in $\cF(t)$ and thus in $\cF$ by the proof given just above. We deduce that $\cF$ satisfies Ore's property.

{\it Left-cancellative.}
Let us prove that $\cF$ is left-cancellative.
Observe that the skein presentation of $\cF$ is complemented (see Definition \ref{def:complemented}) since for each pair of distinct colours $(a,b)$ there is a unique relation of the form $a_1\cdots=b_1\cdots$ which is $(C_a(\tau_a),C_b(\tau_b))$ and no relation of the form $a_1\cdots=a_1\cdots.$
By Proposition \ref{prop:LC-complemented}, we are reduced to check that 
\begin{equation*}E(a,b,c):=[(a_1\bs b_1)\bs (a_1\bs c_1)]\bs [(b_1\bs a_1)\bs (b_1\bs c_1)]\end{equation*}
is either undefined or equal to $e$ for all triple of distinct colours $(a,b,c)\in S^3.$
For each colour $a\in S$ we decompose the tree $C_a(\tau_a)$ as $a_1\circ f^a$ where $f^a$ is a forest with two roots that we identify with the element $f^a\ot I^{\ot \infty}$ of $\cF_\infty.$
By definition we have that $(a_1\bs b_1) = f^a$ and note that this does not depend on the colour $b$.
We deduce that 
$$E(a,b,c)=[f^a\bs f^a]\bs [f^b \bs f^b] = e\bs e =e.$$
We conclude that $\cF$ is left-cancellative and thus $\cF$ is a Ore category admitting a fraction group $G:=\Frac(\cF,1).$

{\it Spine and finiteness property.}
We now compute the spine of $\cF$.
By definition $\Sp_0(\cF)=\{Y_a:\ a\in S\}.$
Now, $\mcm(Y_a,Y_b)=\{C_a(\tau_a)\} = \{ t\}$ for all $a,b\in S, a\neq b.$
We deduce that $\Sp_1(\cF)=\{t\}$ is a singleton implying that $\Sp_n(\cF)$ is empty for all $n\geq 2$. 
Therefore, the spine of $\cF$ is equal to $$\Sp(\cF)=\{Y_a:\ a\in S\}\cup \{t\}.$$ 
In particular, the cardinal of $\Sp(\cF)$ is smaller or equal to the cardinal of $S$ and a point.

By Theorem \ref{theo:Finfty}, we deduce that if $S$ is finite, then the fraction group $\Frac(\cF,1)$ is of type $F_\infty$ and so are its $T,V,$ and $BV$-versions.

{\it Presentation.}
Fix one colour $a\in S$.
A skein presentation of $\cF$ is given by 
$$(S \ | \ (C_a(\tau_a),C_b(\tau_b)),\ b\in S\setminus\{a\})$$
where the set of relations is in bijection with $S\setminus\{a\}$ (rather than all pairs $(b,c)$ with $b\neq c$).
Consider now the reduced group presentation of $G$ associated to the same colour $a$ in Theorem \ref{theo:groupG-presentation}.
Remove the generators $\wh a_1$ and $\wh a_2$ and the third kind of relations.
Assume now that $S$ is finite of cardinal $n.$
We have $4n-2$ generators.
Moreover, the Thompson-like relations provide $2(4n^2 - 2n)=8n^2 - 4n$ group relations for $G$.
The last kind of group relation of $G$ in Theorem \ref{theo:groupG-presentation} provide $4n-4$ relations since we have $n-1$ skein relations.
In conclusion $G$ admits a group presentation with $4n-2$ generators and $8n^2-4$ relations.
\end{proof}

\begin{remark}\label{rem:tau}
\begin{enumerate}
\item Note that if $\tau_a$ has two leaves for all $a\in S$, then $\cF_\tau$ is nothing else than $\cUF\la*\ra$ the free forest-skein category over one colour and thus $\Frac(\cF_\tau,1)=F.$
When the trees $\tau_a,a\in S$ have at least three leaves, then $S$ is in bijection with $\cF(2,1)$.
\item 
We have proven that the spine of $\cF_\tau$ is equal to its trees with two leaves and $t:=C_a(\tau_a)$ for a given $a\in S$.
This is the smallest spine we could expect from a forest-skein category that is not free. 
Hence, $\cF_\tau$ has a very low complexity with respect to Thumann's algorithm for proving finiteness properties. 
\item
Given $\tau:S\to \cF$ and $\sigma$ a permutation of $S$, note that the two groups $G_\tau$ and $G_{\tau\circ \sigma}$ are isomorphic. Hence, $G_\tau$ only remembers the range of $\tau$ and the cardinal of each pre-image $\tau^{-1}(\{t\})$.
\item In general the forest-skein category $\cF_\tau$ does not satisfy the CGP (introduced in Section \ref{sec:CGP}) and thus the two groups $\Frac((\cF_\tau)_\infty)$ and $\Frac(\cF_\tau,1)$ may not be isomorphic. Such an example is given in Section \ref{sec:example-no-CGP}.
\item
The fact that the skein relations of $\cF_\tau$ are of the form $(t,s)$ with $t,s$ {\it monochromatic} makes possible to prove Ore's property at this level of generality. When the trees involved in a skein presentation are not monochromatic then it makes the analysis much harder. Indeed, there are many skein presentations with two colours and one relation that fails to provide Ore's property.
Consider for instance $$\cF=\FC\la a,b| a_1b_1a_1= b_1b_1a_1\ra.$$
\item
Note that this class of examples is very rich thanks to the two-dimensional structures of our diagrams. 
If we adapt this construction to classical monoids where equations are written on a line we would obtain much reduced and less interesting class of examples: the monoids 
$$\Mon\la S| a^n=b^n, \text{ for all } a,b\in S\ra$$
where $S$ is a set and $n$ a natural number.
\end{enumerate}
\end{remark}

{\bf System of groups.}
Fix a map $\tau:S\to\cUF\la *\ra$ as above. 
Now, for any nonempty subset $S_0\subset S$ we have the skein presentation $(S_0,R_{S_0})$ where $R_{S_0}$ is the set of all pairs $(C_a(\tau_a),C_b(\tau_b))$ with $a,b\in S_0, a\neq b$.
This produces a forest-skein category $\cF_{S_0}$ and forest-skein group $G_{S_0}$.
Note that the canonical embedding $S_0\into S$ provides an injective morphism $\cF_{S_0}\into \cF_S$ and group embedding $\iota_{S,S_0}:G_{S_0}\into G_S$.
Moreover, we have $\iota_{S,S_0}\circ \iota_{S_0,S_1}=\iota_{S,S_1}$ for a chain $S_1\subset S_0\subset S$.
We obtain a directed set of groups $(G_{S_0}:\ S_0\subset S, T\neq\emptyset).$
In particular, even if $S$ is infinite we can describe the group $G_S$ as an inductive limit of $G_{S_0}$ with $S_0$ finite.

\subsection{Two colours}\label{sec:two-colours}
Consider the case where we have two colours, i.e.~$S=\{a,b\}$.
Given any pair $(t,s)$ of nontrivial monochromatic trees with the same number of leaves we have a forest-skein group $G_{(t,s)}$ and its $X$-versions $G_{(t,s)}^X$ for $X=T,V,BV$ that are all of type $F_\infty$ by the last section.
This is a very large class of groups. 
As observed in Remark \ref{rem:tau} we have that $G_{(t,s)}\simeq G_{(s,t)}$. 
Moreover, if $t'$ is mirror image of $t$ we have $G_{(t,s)}\simeq G_{(t',s')}.$
However, it seems there are no easy way to systematically decide if two such groups are isomorphic or not. 
Although, they are all of type $F_\infty$ we suspect they may satisfy rather different properties.

Note that a pair $(t,s)$ defines an element $g$ of Thompson's group $F$ but note that if both $(t,s)$ and $(t',s')$ defines $g$ then we don't have in general $G_{(t,s)}\simeq G_{(t',s')}.$
Indeed, one can prove for instance that the family $$(G_{(t,t)}:\ t \text{ tree } )$$
contains infinitely many isomorphism classes of groups.
Although, each pair $(t,t)$ corresponds to the trivial element of $F$.
Here is a short proof of this fact.
\begin{proof}
Consider the monochromatic tree $t_n = x_1^n$ with $n+1$ equal to a left-vine made of $n$ carets.
This produces a Ore forest-skein category $\cF_n:=\FC\la a,b| a_1^n=b_1^n\ra$ with forest-skein groups $G_n:=\Frac(\cF_n,1)$ and $H_n:=\Frac((\cF_n)_\infty)$.
Observe that $\cF_n$ satisfies the CGP and thus $G_n\simeq H_n$.
Now, $H_n$ admits the infinite group presentation:
$$H_n=\Gr\langle a_j,b_j, j\geq 1 | x_q y_j = y_j x_{q+1}, a_j^n=b_j^n, 1\leq j<q\rangle$$
with abelianisation isomorphic to $\Z^2 \oplus (\Z/n\Z)^2$
implying that the forest-skein groups $(G_n:\ n\geq 2)$ are pairwise non-isomorphic.
\end{proof}

This leads to a question for which we have no answer.

\begin{question}
Is there a clear description of the following equivalence relation $\top$ defined as:
$$(t,s)\top (t',s') \text{ if and only if } G_{(t,s)}\simeq G_{(t',s')}?$$
\end{question}



\newcommand{\etalchar}[1]{$^{#1}$}

\end{document}